\documentclass[12pt,reqno,a4paper]{report}
\usepackage{amsthm}
\usepackage{mathtools}
\usepackage{fancyhdr}
\usepackage{latexsym,amssymb,amsmath}
\usepackage[all,ps,cmtip]{xy}
\usepackage{enumerate}
\usepackage{listings}
\usepackage{courier}
\usepackage{ytableau}
\usepackage{subfig}
\usepackage{tikz-cd}
\usepackage{youngtab}
\usepackage{xfrac}
\usepackage{mathrsfs}
\usepackage{calligra}
\usepackage[refpage]{nomencl}
\newcommand{\bigslant}[2]{{\raisebox{.2em}{$#1$}\big/\raisebox{-.2em}{$#2$}}}

\usepackage[toc,page]{appendix}
\newcommand{\shO}  {\mathcal{O}}

\theoremstyle{plain}
\newtheorem{innercustomthm}{Theorem}
\newenvironment{summarythm}[1]
  {\renewcommand\theinnercustomthm{#1}\innercustomthm}
  {\endinnercustomthm}

\newtheorem{conjecture}{Conjecture}
\newtheorem{theorem}{Theorem}[chapter]
\newtheorem{lemma}[theorem]{Lemma}
\newtheorem{proposition}[theorem]{Proposition}
\newtheorem{corollary}[theorem]{Corollary}
\newtheorem{claim}[theorem]{Claim}
\theoremstyle{definition}
\newtheorem{definition}[theorem]{Definition}

\newtheorem{example}[theorem]{Example}

\theoremstyle{remark}
\newtheorem{remark}[theorem]{Remark}

\usepackage[section]{placeins}
\usepackage[format=hang]{caption}
\newcommand{\um}{\underline{s}}

\usepackage[format=hang]{caption}
\newcommand{\defeq }{\vcentcolon=}
\newcommand{\eqdef}{=\vcentcolon}

\newcommand{\NN} {\mathbb{N}}
\newcommand{\ZZ} {\mathbb{Z}}
\newcommand{\QQ} {\mathbb{Q}}
\newcommand{\RR} {\mathbb{R}}
\newcommand{\CC} {\mathbb{C}}
\newcommand{\EE} {\mathbb{E}}
\newcommand{\kk} {\mathbb{C}}

\newcommand{\PP} {\mathbb{P}}
\renewcommand{\AA} {\mathbb{A}}
\newcommand{\GG} {\mathbb{G}}

\newcommand{\VV} {\mathbb{V}}

\newcommand {\shA}  {\mathcal{A}}

\newcommand {\shB}  {\mathcal{B}}
\newcommand {\shC}  {\mathcal{C}}
\newcommand {\shD}  {\mathcal{D}}

\newcommand {\shE}  {\mathcal{E}}
\newcommand {\shF}  {\mathcal{F}}
\newcommand {\shG}  {\mathcal{G}}
\newcommand {\shH}  {\mathcal{H}}
\newcommand {\shI}  {\mathcal{I}}

\newcommand {\shL}  {\mathcal{L}}

\newcommand {\shQ}  {\mathcal{Q}}
\newcommand {\shR}  {\mathcal{R}}
\newcommand {\shS}  {\mathcal{S}}
\newcommand {\shT}  {\mathcal{T}}
\newcommand {\shU}  {\mathcal{U}}
\newcommand {\shProj} {\mathop{\shP roj}\nolimits}
\newcommand {\shP}  {\mathcal{P}}
\newcommand {\shDrap} {\mathop{\shD rap}\nolimits}
\newcommand {\shV}  {\mathcal{V}}

\newcommand {\shY}  {\mathcal{Y}}

\newcommand {\extE} {\shE^{\scriptscriptstyle{\shF}}}

\newcommand {\Amp}  {\operatorname{Amp}}

\newcommand {\Aut}  {\operatorname{Aut}}

\newcommand {\ch} {\operatorname{ch}}

\newcommand {\Conv} {\operatorname{Conv}}

\newcommand {\DF} {\operatorname{DF}}

\newcommand {\End}  {\operatorname{End}}
\newcommand {\Ext}  {\operatorname{Ext}}
\newcommand {\shEnd} {\mathop{\shE nd}\nolimits}

\newcommand {\fAlg} {\mathtt{FAlg}_{\shO_B}}

\newcommand {\Flag} {\operatorname{\shF \mathit{l}}}

\newcommand {\GL}  {\operatorname{GL}}

\newcommand {\id}  {\operatorname{id}}
\newcommand {\im}  {\operatorname{im}}

\newcommand {\kbar} {$\overline{\mbox{K}}$}

\newcommand {\lev} {\operatorname{lev}}

\newcommand {\lra}  {\longrightarrow}

\newcommand{\ot}  {\operatorname{\otimes}}

\newcommand {\Pic}  {\operatorname{Pic}}
\newcommand {\PGL}  {\operatorname{PGL}}
\newcommand {\PSL}  {\operatorname{PSL}}

\newcommand {\Proj} {\operatorname{Proj}}

\renewcommand {\ch} {\operatorname{ch}}

\renewcommand {\shO} {\mathcal{O}}

\newcommand {\ra}  {\rightarrow}

\newcommand {\rank} {\operatorname{rank}}

\newcommand {\Rees} {\operatorname{Rees}}
\newcommand {\shRees} {\mathop{\shR ees}\nolimits}
\newcommand {\Ric}  {\operatorname{Ric}}
\newcommand {\scal}  {\operatorname{Scal}}

\newcommand\restr[2]{{
  \left.\kern-\nulldelimiterspace 
  #1 
  \vphantom{\big|} 
  \right|_{#2} 
  }}

\newcommand {\sgn} {\mathrm{sgn}}

\newcommand {\SL}  {\operatorname{SL}}
\newcommand {\Spec} {\operatorname{Spec}}

\newcommand {\tr}  {\operatorname{tr}}

\newcommand {\test} {\mathop{Test}\nolimits}

\makeatletter
\newcommand{\ex}{\@ifnextchar^\@extp{\@extp^{\,}}}
\def\@extp^#1{\mathop{\bigwedge\nolimits^{\!#1}}}
\makeatother

\makeatletter
\def\Ddots{\mathinner{\mkern1mu\raise\p@
\vbox{\kern7\p@\hbox{.}}\mkern2mu
\raise4\p@\hbox{.}\mkern2mu\raise7\p@\hbox{.}\mkern1mu}}
\makeatother

\newcommand {\scrL} {\mathscr{L}}
\newcommand {\X} {\mathscr{X}}
\newcommand {\Y} {\mathscr{Y}}

\newcommand {\U} {\mathscr{U}}

\def\mydate{\ifcase\month \or January\or February\or March\or
April\or May\or June\or July\or August\or September\or October\or
November\or December\fi \space\number\day,\space\number\year}

\RequirePackage[T1]{fontenc}
\RequirePackage[utf8x]{inputenc}
\RequirePackage{amssymb}
\RequirePackage{stmaryrd}

\DeclareUnicodeCharacter{183}{\cdot}
\DeclareUnicodeCharacter{915}{\ensuremath{\Gamma}}
\DeclareUnicodeCharacter{916}{\ensuremath{\Delta}}
\DeclareUnicodeCharacter{918}{\ensuremath{\Zeta}}
\DeclareUnicodeCharacter{920}{\ensuremath{\Theta}}
\DeclareUnicodeCharacter{923}{\ensuremath{\Lambda}}
\DeclareUnicodeCharacter{926}{\ensuremath{\Xi}}
\DeclareUnicodeCharacter{928}{\ensuremath{\Pi}}
\DeclareUnicodeCharacter{931}{\ensuremath{\Sigma}}
\DeclareUnicodeCharacter{933}{\ensuremath{\Upsilon}}
\DeclareUnicodeCharacter{934}{\ensuremath{\Phi}}
\DeclareUnicodeCharacter{935}{\ensuremath{\Chi}}
\DeclareUnicodeCharacter{936}{\ensuremath{\Psi}}
\DeclareUnicodeCharacter{937}{\ensuremath{\shOmega}}
\DeclareUnicodeCharacter{945}{\ensuremath{\alpha}}
\DeclareUnicodeCharacter{946}{\ensuremath{\beta}}
\DeclareUnicodeCharacter{947}{\ensuremath{\gamma}}
\DeclareUnicodeCharacter{948}{\ensuremath{\delta}}
\DeclareUnicodeCharacter{949}{\ensuremath{\epsilon}}
\DeclareUnicodeCharacter{950}{\ensuremath{\zeta}}
\DeclareUnicodeCharacter{951}{\ensuremath{\eta}}
\DeclareUnicodeCharacter{952}{\ensuremath{\theta}}
\DeclareUnicodeCharacter{953}{\ensuremath{\iota}}
\DeclareUnicodeCharacter{954}{\ensuremath{\kappa}}
\DeclareUnicodeCharacter{955}{\ensuremath{\lambda}}
\DeclareUnicodeCharacter{956}{\ensuremath{\mu}}
\DeclareUnicodeCharacter{957}{\ensuremath{\nu}}
\DeclareUnicodeCharacter{958}{\ensuremath{\xi}}
\DeclareUnicodeCharacter{959}{\ensuremath{\omicron}}
\DeclareUnicodeCharacter{960}{\ensuremath{\pi}}
\DeclareUnicodeCharacter{961}{\ensuremath{\rho}}
\DeclareUnicodeCharacter{963}{\ensuremath{\sigma}}
\DeclareUnicodeCharacter{964}{\ensuremath{\tau}}
\DeclareUnicodeCharacter{965}{\ensuremath{\upsilon}}
\DeclareUnicodeCharacter{966}{\ensuremath{\varphi}}
\DeclareUnicodeCharacter{967}{\ensuremath{\chi}}
\DeclareUnicodeCharacter{968}{\ensuremath{\psi}}
\DeclareUnicodeCharacter{969}{\ensuremath{\omega}}
\DeclareUnicodeCharacter{8214}{\parallel}
\DeclareUnicodeCharacter{8450}{\mathbb{C}}
\DeclareUnicodeCharacter{8470}{\mathbb{N}}
\DeclareUnicodeCharacter{8474}{\mathbb{Q}}
\DeclareUnicodeCharacter{8477}{\mathbb{R}}
\DeclareUnicodeCharacter{8484}{\mathbb{Z}}
\DeclareUnicodeCharacter{8614}{\mathbin{\mapsto}}
\DeclareUnicodeCharacter{8656}{\Leftarrow}
\DeclareUnicodeCharacter{8657}{\Uparrow}
\DeclareUnicodeCharacter{8658}{\Rightarrow}
\DeclareUnicodeCharacter{8659}{\Downarrow}
\DeclareUnicodeCharacter{8669}{\rightsquigarrow}
\DeclareUnicodeCharacter{8797}{\eqdef}
\DeclareUnicodeCharacter{8870}{\vdash}
\DeclareUnicodeCharacter{8873}{\Vdash}
\DeclareUnicodeCharacter{8871}{\models}
\DeclareUnicodeCharacter{9121}{\lceil}
\DeclareUnicodeCharacter{9123}{\lfloor}
\DeclareUnicodeCharacter{9124}{\rceil}
\DeclareUnicodeCharacter{9126}{\rfloor}
\DeclareUnicodeCharacter{9655}{\triangleright}
\DeclareUnicodeCharacter{9665}{\triangleleft}
\DeclareUnicodeCharacter{9671}{\diamond}
\DeclareUnicodeCharacter{9675}{\circ}
\DeclareUnicodeCharacter{10178}{\bot}
\DeclareUnicodeCharacter{10214}{\llbracket} 
\DeclareUnicodeCharacter{10215}{\rrbracket} 
\DeclareUnicodeCharacter{10229}{\longleftarrow}
\DeclareUnicodeCharacter{10230}{\longrightarrow}
\DeclareUnicodeCharacter{10231}{\longleftrightarrow}
\DeclareUnicodeCharacter{10232}{\Longleftarrow}
\DeclareUnicodeCharacter{10233}{\Longrightarrow}
\DeclareUnicodeCharacter{10234}{\Longleftrightarrow}
\DeclareUnicodeCharacter{10236}{\longmapsto}
\DeclareUnicodeCharacter{10238}{\Longmapsto} 
\DeclareUnicodeCharacter{10503}{\Mapsto}    
\DeclareUnicodeCharacter{10971}{\mathrel{\not\hspace{-0.2em}\cap}}
\DeclareUnicodeCharacter{65294}{\ldotp}
\DeclareUnicodeCharacter{65372}{\mid}

\usepackage[twoside, bindingoffset=1cm]{geometry}

\usepackage{url}
\usepackage[colorlinks=false, linktocpage=true]{hyperref}

\title{K-Stability of Relative Flag Varieties}
\author{Anton Isopoussu}

\begin{document}

\begin{titlepage}
\vspace*{1mm}
\begin{center}
\Huge K-stability of relative flag varieties
\end{center}
\addvspace{18mm}
\begin{center}
\LARGE Anton Isopoussu
\end{center}
\addvspace{23mm}
\begin{center}
\includegraphics[width=30mm]{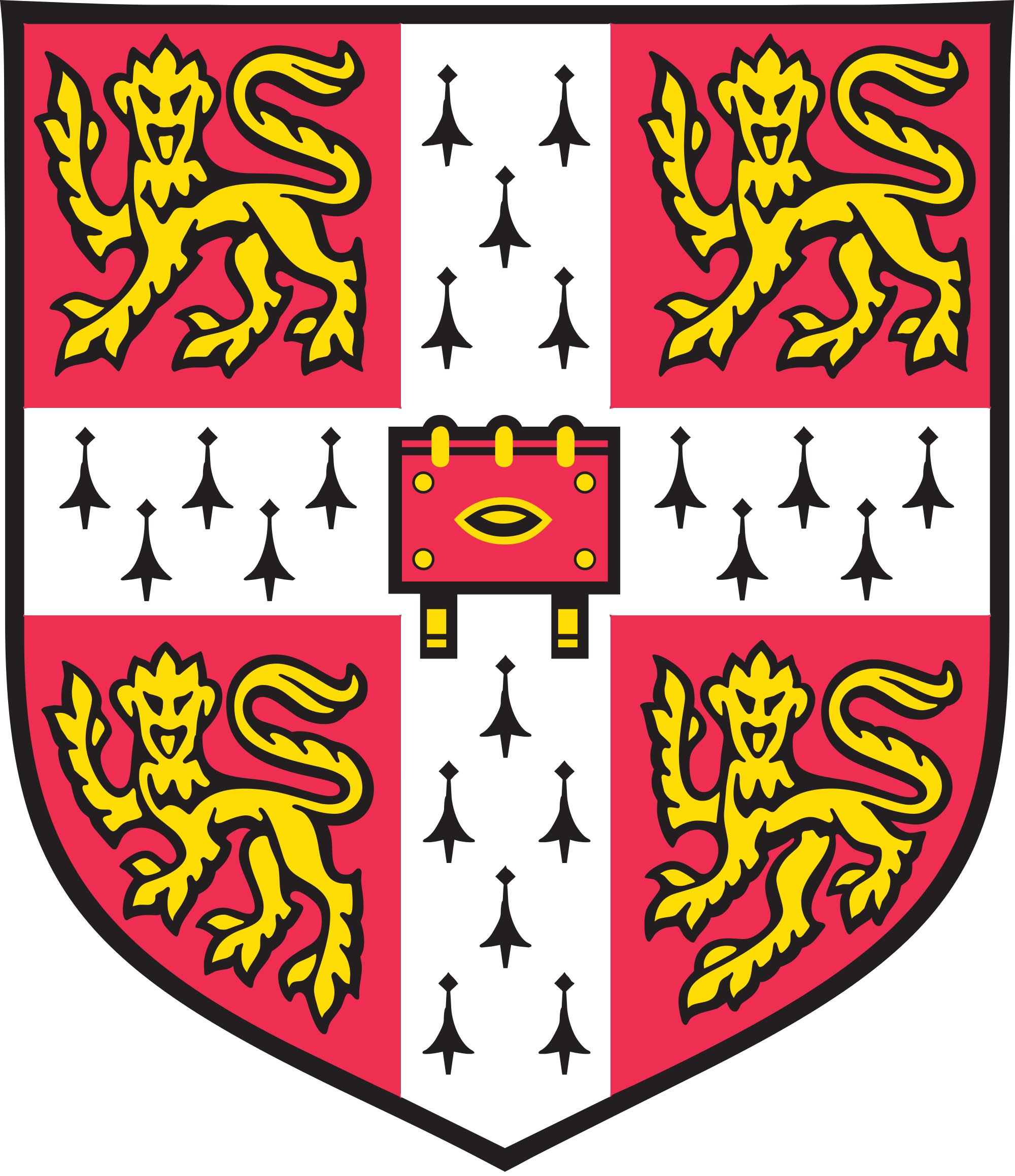}
\end{center}
\addvspace{18mm}
\begin{center}
\Large University of Cambridge\\
Department of Pure Mathematics and Mathematical Statistics\\
Churchill College
\end{center}
\addvspace{8mm}
\begin{center}\Large April 2015\end{center}
\addvspace{12mm}
\begin{center}
This dissertation is submitted for\\
the degree of Doctor of Philosophy
\end{center}
\end{titlepage}
\setcounter{page}{2} 
\pagestyle{empty}

\clearpage

\setcounter{page}{2}

\chapter*{Abstract}
    We generalise partial results about the Yau-Tian-Donaldson correspondence on ruled manifolds to bundles whose fibre is a classical flag variety. This is done using Chern class computations involving the combinatorics of Schur functors. The strongest results are obtained when working over a Riemann surface. Weaker partial results are obtained for adiabatic polarisations in the general case.

     We develop the notion of relative K-stability which embeds the idea of working over a base variety into the theory of K-stability. We equip the set of equivalence classes of test configuration with the structure of a convex space fibred over the cone of rational polarisations. From this, we deduce the openness of the K-unstable locus. We illustrate our new algebraic constructions with several examples.

\thispagestyle{empty}
\chapter*{Declaration}
\thispagestyle{empty}
This dissertation is the result of my own work and includes nothing which is the outcome of work done in collaboration except as declared in the Preface and specified in the text.

It is not substantially the same as any that I have submitted, or, is being concurrently submitted for a degree or diploma or other qualification at the University of Cambridge or any other University or similar institution except as declared in the Preface and specified in the text. I further state that no substantial part of my dissertation has already been submitted, or, is being concurrently submitted for any such degree, diploma or other qualification at the University of Cambridge or any other University of similar institution except as specified in the text
\vspace{15mm}

\noindent Anton Isopoussu\\
\today

\chapter*{Acknowledgements}
\thispagestyle{empty}

I gratefully acknowledge the patient guidance of my supervisor Dr. Julius Ross and the financial support of Osk. Huttunen foundation. Ruadhai Dervan has contributed extensively with invaluable mathematical conversations, sarcasm and comments on my drafts. I also thank Yoshinori Hashimoto for many stimulating discussions. The other members of the group, Cristiano Spotti and David Witt-Nyström have also always been generous with their expertise.

My good angels Alexandra, Hilary and Rick have greatly facilitated the completion of this work, while my friends at Churchill College have offered many memorable distractions over the years from engaging conversations and late night jam sessions to feisty football games in the intercollegiate leagues.

I also wish to acknowledge my teachers and peers who made me fall in love with mathematics. I thank Juhani Pitkäranta $ε$-$δ$-games, Kirsi Peltonen for subriemannian geometry, Matti Vuorinen for quasiregular mappings, Ville Turunen for noncommutative geometry, Juha Kinnunen for Sobolev spaces, Tapani Hyttinen the for liar paradox in number theory, Jouko Väänänen for teams, Sören Illman for \v{C}ech cohomology, Pelham Wilson for Hodge theory and Caucher Birkar for schemes. I also thank my fellow students Samu Alanko for Riemannian sums, Ville Pettersson for grid graphs, Hannu Nissinen for model games, Monique van Beek for elliptic curves with torsion points and John Christian Ottem for much magic with K3 surfaces. \tableofcontents
\pagestyle{plain}
\chapter{Introduction} 
\section{Introduction} 
\label{sec:introduction}
In this thesis we study a mysterious relationship between complex differential geometry and algebraic geometry which has been established around the existence of a best possible Kähler form on a projective complex manifold. Recall that a Kähler manifold is a pair $(X,ω)$ where $X$ is a complex manifold and $ω$ is a closed positive nondegenerate differential form of type (1,1). Kähler manifolds have a wealth of good properties which belies their simple definition. It is natural to study the problem of finding a best possible Kähler metric on $X$.

A wonderfully rich picture arises already for Riemann surfaces, which have been studied both algebraically and analytically for more than a century. The famous Uniformisation Theorem of Poincar\'{e} and Köbe states that a compact Riemann surface can be written as a quotient of a model space, either the hyperbolic disk, the flat complex plane or the round sphere. Alternatively, this can be stated by saying that any compact Riemann surface can be endowed with a metric, unique up to a constant, whose sectional curvature is constant. On the algebraic side, the compactification of the moduli space of curves is, of course, one of the major accomplishments of modern algebraic geometry. The two points of view are connected, for example, in the definition of the Weil-Petersson metric on the compact moduli space of algebraic curves.

A pair $(X,L)$, where $X$ is a variety defined over the complex numbers and $L$ is an ample line bundle, is called a \emph{polarised variety}. We assume for now that $X$ is smooth. A \emph{canonical metric} on the polarised variety $(X,L)$ is a Kähler form $ω$ which is a solution to some naturally defined differential equation, is unique up to an automorphism of $X$ and whose cohomology class is equal to $c_1(L)$. Canonical metrics in this sense are one natural generalisation of the Uniformisation Theorem to higher dimensions. Given the existence of constant sectional curvature metrics on Riemann surfaces, it is tempting to conjecture that canonical metrics should always exist. This turns out to be a subtle question, which has inspired a wealth of new mathematics at the intersection of complex and algebraic geometry.

The theory of \emph{K-stability} connects the question of existence of canonical metrics on higher dimensionals polarised varieties to algebraic geometry. K-stability is a conjecturally equivalent condition to the existence of a canonical metric on $(X,L)$. We call this the Yau-Tian-Donaldson (YTD) correspondence.

A key idea that originates from the work of Hilbert and Mumford is that one can associate numerical invariants to degenerations of $(X,L)$. Let $m$ be a natural number and consider a projective embedding
\begin{equation}
    X\subset \PP^n = \PP(H^0(X,L^m))
\end{equation}
and an action of the multiplicative group $\GG_m$ on $\PP^n$, which acts linearly on the hyperplane bundle on $\PP^n$. Then the orbit of $X$ under the $\GG_m$-action is a family of copies of $(X,L^m)$ which can be compactified over the point $t\to 0$ in $\GG_m$. The resulting family $\X$, which has a special fibre $(X_0, L_0)$ invariant under the $\GG_m$-action, is called a \emph{test configuration}. The precise definition is given in Definition \ref{def:TCdef}. The group $\GG_m$ has a representation on the space of sections of the line bundle $L_0$ which determines an important numerical invariant called the \emph{Donaldson-Futaki invariant}.

With certain refinements which will be discussed in the text, we say that $(X,L)$ is \emph{K-stable} if the Donaldson-Futaki invariant $\DF(\X)$, which will be defined by Equation \eqref{eq:futdef}, is positive for all test configurations $\X$. Otherwise, we say $(X,L)$ is \emph{K-unstable}. Paraphrasing the earlier discussion, a negative Donaldson-Futaki invariant is a conjectural obstruction to the existence of a canonical metric.

Most of this work is dedicated to the study of Donaldson-Futaki invariants in a simple example. We say that a variety $Y$ is a \emph{flag bundle} if it comes with a Zariski-locally trivial projection $p\colon Y\rightarrow B$ to a projective variety $B$, such that the fibres of $p$ are isomorphic to a flag variety. This is a natural generalisation of a geometrically ruled manifold which is the single most studied example in the theory of K-stability. The only rival to this status are toric varieties. Flag bundles retain many of the properties of geometrically ruled manifolds while exhibiting new features which make them worthy of an extended discussion, such as a larger Picard group and richer geometric structure. Flag bundles also provide a working example to test a folklore conjecture that the stability properties of the underlying vector bundle should determine the K-stability of its associated projective manifolds. We give a partial affirmative answer to this conjecture.

Preliminary material is presented in Chapters \ref{chap:preliminaries} and \ref{chap:stability}. The former recalls basic notions of group actions on algebraic varieties and introduces the reader to flag bundles in more detail. The latter is an introduction to the theory of K-stability. Chapter \ref{chap:chern} contains a technical result, which will be crucial in the computation of Donaldson-Futaki invariants in Chapter \ref{chap:flags} where we construct destabilising test configurations for flag bundles. Chapter \ref{chap:UTF} is independent of the rest of the text in which we describe a generalisation of the Uniformisation Theorem to flag bundles whose underlying vector bundle is a polystable vector bundle over a Riemann surface. We thus obtain partial results towards a YTD correspondance on flag bundles. We give additional examples of K-unstable varieties in Chapter \ref{chap:sub}, where we study the K-stability of complete intersections.

Chapter \ref{chap:filtrations} is almost entirely independent of the rest of the work and will discuss a general theme that arises from the particularly simple type of test configuration that was used in previous chapters. Families of simple projective varieties have been a rich source of examples in the past \cite{apostolov2006stability,apostolov2008hamiltonian,Fine2004,fine2005fibrations,KellerRoss,lu2014extremal,RossThomas,Seyyedali2010,Stoppa2010}. We define and attempt to justify the notion of \emph{relative K-stability}. Roughly speaking this term refers to dividing the set of test configurations for $(X,L)$ into collections of simpler test configurations, each of which linked to a projective morphism $X\rightarrow B$, where $B$ is a projective variety.

We develop the theory of filtrations of sheaves with a view towards studying relative K-stability. This generalises the work by Székelyhidi \cite{szekelyhidi2011filtrations} and Witt-Nyström \cite{witt2012test}. Certain constructions of new test configurations from old have already appeared in the work of Ross and Thomas \cite{RossThomas}. We contextualise them using the language of filtrations and obtain new constructions, which we hope will be helpful in exhibiting interesting new behaviour of K-stability in the Kähler cone. We focus particularly on a weighted tensor products on filtered algebras, which allow us to endow the set of test configurations, up to some natural identifications, with a convex structure which is naturally fibred over the cone of polarisations. We show that Donaldson-Futaki invariants behave well under this construction which, in particular, imples the openness of the K-unstable locus.

 %

\section{Background} 
\label{sec:background}
There are three natural higher dimensional analogues to constant sectional curvature metrics in Kähler geometry. A Kähler form $ω$ is \emph{extremal} if the complex gradient vector field of its scalar curvature is holomorphic. The form $ω$ has constant scalar curvature (cscK) if this gradient vector field vanishes identically. This coincides with the usual requirement that the scalar curvature function is constant. The simplest case is to consider the equation
\begin{equation}
    \Ric ω = Cω,
\end{equation}
where $C$ is a constant and $\Ric ω$ is the Ricci form. These metrics are called \emph{Kähler-Einstein} and they form an important special class of cscK metrics.

Uniqueness was proved in increasing generality by Bando and Mabuchi \cite{bando1985uniqueness}, Chen \cite{chen2000space}, Donaldson \cite{DonaldsonScalar}, Mabuchi \cite{mabuchi2004uniqueness} and Berman and Berndtsson \cite{berman2004convexity}, who showed that an extremal metric on an arbitrary Kähler manifold $(X,ω)$ is unique up to automorphisms.

\subsection{Kähler-Einstein metrics and the history of the YTD correspondence} 
\label{sub:KEYTD}

The cohomology class of a Kähler-Einstein metric is equal to a multiple first Chern class $c_1(X)$ of $X$ so Kähler-Einstein metrics can only exist if the first Chern class $c_1(X)$ has definite sign. This is a major topological restriction on $X$. If $c_1(X)$ is trivial, the famous Calabi-Yau theorem \cite{yau1978ricci} states that there is a unique KE metric up to automorphism. In the case $c_1(X)<0$, Yau proved that there is a unique KE metric up to scale and automorphisms of $X$.

The case $c_1(X) > 0$ is more complicated. Matsushima showed that if the automorphism group of $X$ is not reductive, then $X$ does not admit a Kähler-Einstein metric. Donaldson-Futaki \cite{futaki1983obstruction} found another obstruction related to certain pathological vector fields on $X$. Yau then posed the problem of relating the problem of existence of Kähler-Einstein metrics to a stability notion in algebraic geometry \cite{yau2000open}. Ding and Tian proposed \emph{K-stability} as a conjectural solution to Yau's problem \cite{ding1992kahler} defined using an ingenious combination of Futaki's work with algebraic degenerations of $X$. Donaldson gave the fully algebraic definition of K-stability \cite{Donaldson2002}, which is used in this work with minor modifications.

The equivalence between the existence of a Kähler-Einstein metric and K-stability was proved by Chen, Donaldson and Sun which settled one of the most famous modern conjectures in geometry. The problem has inspired many novel ideas, such as the algebraisation of Gromov-Hausdorff limits \cite{donaldson2012gromov}, which is a technique of endowing a limiting object in Riemannian geometry under certain hypotheses with the structure of an algebraic variety. The continuity method for metrics with cone singularities is another new construction that was crystallised in the work of Chen, Donaldson and Sun. These two key ideas were beautifully embedded in the proof of the following theorem \cite{donaldson2012kahler}.
\begin{theorem}[\cite{chen2012kahler,chen2012kahler2,chen2013kahler3,berman2012k}]
    The pair $(X,-K_X)$ is K-stable if and only if $X$ admits a Kähler-Einstein metric.
\end{theorem}

\subsection{The Yau-Tian-Donaldson conjecture for constant scalar curvature Kähler metrics} 
\label{sub:overview}
Constant scalar curvature Kähler metrics can be defined by the equation
\begin{equation}
    \scal(ω) = C,
\end{equation}
where $\scal(ω)$ is defined by the equation $\scal(ω)ω^n = \dim X \Ric(ω)\wedge ω^{n-1}$ with $n$ denoting the dimension of the manifold $X$, and $C$ is a constant. CscK metrics on an arbitrary polarised manifold is the first natural generalisation of the Kähler-Einstein YTD correspondence. We say that $(X,L)$ is cscK if $X$ admits a metric in $c_1(L)$ which is cscK. Donaldson made the following conjecture.
\begin{conjecture}[The Yau-Tian-Donaldson conjecture \cite{Donaldson2002}]\label{conj:ytd}
    Let $(X,L)$ be a polarised smooth complex variety. Then there is a constant scalar curvature Kähler cscK metric in the class $c_1(L)$ if and only if $(X,L)$ is K-polystable.
\end{conjecture}
We refer to Definition \ref{def:kstab} for the definition of K-polystability.

\begin{remark}
    Li-Xu \cite{li2011special} gave an example which contradicted the YTD correspondence as it was originally stated, which included certain pathological test configurations. The solution offered by Li-Xu was to only consider normal test configurations. We follow an alternative convention due to Stoppa \cite{stoppa2011note}, which is to allow nonnormal test configurations whose normalisations are not \emph{trivial}. Székelyhidi used yet another convention by restricting to test configurations with positive \emph{norm}. The final point of view was proven to be equivalent with the first two by Dervan \cite{dervan2014uniform}. The norm and triviality of a test configuration are defined in Section \ref{sec:the_definition_of_k_stability}.
\end{remark}

\subsection{K-stability of cscK manifolds} 
\label{sub:k_stability_of_csck_manifolds}

Donaldson proved an elegant formula which relates scalar curvature with Donaldson-Futaki invariants explicitly.
\begin{proposition}[{\cite{Donaldson2008}}]\label{prop:stab}
    Let $(X,L,ω)$ be a polarised Kähler manifold with $2πω = c_1(L)$ and let $\X$ be a test configuration for $(X,L)$. The following lower bound holds for the Calabi functional
    \begin{equation}
        \|\scal(ω)-\overline{\scal}(ω)\|_{L^2(ω^n)}  \geq -c\frac{\DF(\X)}{\|\X\|}
    \end{equation}
    for some positive constant $c$ independent of the test configuration $\X$ and the Kähler form $ω$. Here $\scal(ω)$ is the scalar curvature of $ω$, $\scal(ω)$ is its average, the norm is taken with respect to integrating with the volume form induced by $ω$, and the quantity $\|\X\|$ is called the \emph{norm} of the test configuration $\X$.

    In particular, if $(X,L)$ is cscK, then it is K-semistable.
\end{proposition}
Arezzo and Pacard constructed cscK metrics on blowups of points of cscK manifolds assuming that the volume of the exceptional divisor is small.
\begin{proposition}[{\cite{arezzo2006blowing}}]\label{prop:AP}
    Let $(X,L)$ be a polarised cscK Kähler manifold with a discrete automorphism group and let $Y$ be the blowup of a point on $X$ with $p:Y\rightarrow X$ being the projection. Then there exists a positive number $ε_0$ such that there is a constant scalar curvature metric on $(Y,p^*L-εE)$ for $0<ε<ε_0$. Here $E$ is the exceptional divisor on $Y$.
\end{proposition}
Stoppa noticed that the Donaldson-Futaki invariant of a particular test configuration $\Y$ on $(Y,L-εE)$, using notation from Proposition \ref{prop:AP} is equal to
\begin{equation}
    \DF(\Y) = \DF(\X) - Cε^{-n+1} + O(ε^{-n}),
\end{equation}
where $\DF(\Y)$ and $\DF(\X)$ are the Donaldson-Futaki invariants of $\Y$ and $\X$, respectively, and $C$ is a positive constant. Stoppa then deduced one implication of the YTD conjecture.
\begin{proposition}[{\cite{Stoppa2009}}]\label{prop:CscKtoStab}
    Let $(X,L)$ be a polarised variety with a discrete automorphism group and assume $(X,L)$ is cscK. Then $(X,L)$ is K-stable.
\end{proposition}
Finally, Berman proved the K-polystability of an anticanonically polarised Fano variety admitting a cscK metric \cite{berman2012k}.


\subsection{Projective bundles} 
\label{sub:pbundles}
Producing cscK metrics remains the main method of finding examples of K-stable varieties since the nonexistence of a test configuration with vanishing or negative Donaldson-Futaki invariant is difficult to prove otherwise. No general method for constructing cscK metrics is known either, but partial results are known in special cases. We believe that eventually the locus of K-stable polarisations in the Kähler cone of $X$ should yield to an explicit description, at least in interesting examples. Projective bundles, and slightly more generally flag bundles, are the simplest nontrivial examples.

We consider the bundle $\PP E$ over a smooth projective variety $B$ whose fibres are spaces of 1-dimensional quotients of a holomorphic vector bundle $E$. Let $\shO(1)$ denote the relative hyperplane bundle on $\PP E$, fix a line bundle $A$ on $B$ and assume that the line bundle
\begin{equation}
    \shL(A) = \shO(d)\ot p^* A.
\end{equation}
is ample. Using the theory of \emph{slope stability} and results of Narasimhan and Sesadri, Ross and Thomas proved that the K-stability of a projective bundle on a curve is very closely related to the K-stability of the base and the stability of the underlying vector bundle.
\begin{theorem}[\cite{Narasimhan1965,RossThomas}]\label{thm:RTC}
    Assume that $B$ is of complex dimension one. Then $E$ is Mumford (semi/poly)stable if $(X,\shL(A))$ is slope (semi/poly)stable. If $E$ is polystable, then $(X,\shL(A))$ admits a cscK metric. Conversely, if $E$ is strictly unstable, then $(X,\shL(A))$ does not admit a cscK metric, and if $E$ is not polystable, then $(X,\shL(A))$ is not K-polystable.
\end{theorem}
Without the assumption on the dimension of $B$, Ross and Thomas proved the following theorem using a result of Hong \cite{Hong2002},
\begin{theorem}[{\cite{Hong2002,RossThomas}}]
    Assume that $A$ is an ample line bundle on $B$. Then $E$ is slope stable if there exists an $m_0$ depending on $B$, $A$ and $E$ such that $(X,\shL(A^m))$ is K-stable for $m>m_0$. Conversely, if $E$ is strictly unstable, then $(X,\shL(A))$ does not admit a cscK metric, and if $E$ is not polystable, then $(X,\shL(A))$ is not K-polystable.
\end{theorem}

Lu and Seyyedali \cite{lu2014extremal} generalised Donaldson's perturbation method \cite{DonaldsonScalar} and constructed extremal metrics in adiabatic classes on projective bundles. Similar techniques have been used by Seyyedali \cite{Seyyedali2010} and \cite{KellerRoss} to construct \emph{balanced metrics} in adiabatic classes on projective bundles. Balanced metrics and asymptotic Chow stability have a pivotal role in the development of the theory of K-stability which is eloquently described in \cite{DonaldsonScalar}. As a general rule, many of the difficult constructions in the theory of K-stability are usually known for projective bundles because of their simplicity.

More explicit constructions are carried out on certain simpler projective bundles by Székelyhidi \cite{szekelyhidi2007extremal,szekelyhidi2009calabi} and Apostolov, Calderbank, Gauduchon and T\o nnesen-Friedman \cite{apostolov2006stability,apostolov2008hamiltonian,apostolov2011extremal}. Apostolov and T\o nnesen-Friedman show in particular that the YTD conjecture holds for geometrically ruled surfaces \cite{apostolov2006remark}.

An example of a $\PP^1$-bundle $Y$ over a product of three high genus curves with a fascinating property is constructed in  \cite{apostolov2008hamiltonian}. The authors prove an analytic obstruction to the existence of an extremal metric and then construct the same obstruction using the theory of slope stability. A priori, slope stability yields a family of test configurations parametrised by an interval in the rational numbers, but this can be formally extended to an interval in the reals, where the obstruction defined in \cite{apostolov2008hamiltonian} appears. It is widely conjectured that no algebraic test configuration destabilises the projective bundle $Y$.

\subsection{Generalisations of the YTD correspondence} 
\label{sub:generalisations}
Before stating our results, we briefly list various generalisations of Conjecture \ref{conj:ytd} that have appeared in the literature.  In its most general form, the Yau-Tian-Donaldson correspondence can be understood to mean the following statement about the existence of special metrics and stability.
    \begin{center}
    {\emph{There is a canonical metric (of specified type) in the class $c_1(L)$ if and only if the projective variety $(X,L)$ is K-stable (in the appropriate sense)}}
    \end{center}
The correspondences that are known to us are summarised in the following list.
\begin{enumerate}[(1)]
    \item The existence of \emph{cscK metrics} on a smooth polarised variety is equivalent to \emph{K-stability} \cite{Donaldson2002}
    \item The existence of \emph{extremal metrics} on smooth polarised varieties is equivalent to \emph{K-stability relative to infinitesimal automorphisms} -- \cite{szekelyhidi2007extremal,stoppa2011relative}
    \item The existence of \emph{Orbifold cscK metrics} on polarised orbifolds is equivalent to \emph{orbifold K-stability}  \cite{ross2011weighted}
    \item The existence of \emph{cscK metrics with cone singularities along a divisor $D$} on a smooth polarised variety is equivalent to \emph{K-stability relative to the divisor $D$} \cite{donaldson2012kahler}
    \item The existence of \emph{twisted cscK metrics} on a smooth polarised variety is equivalent to \emph{twisted K-stability}  \cite{Fine2004,stoppa2009twisted,dervan2014uniform}
\end{enumerate}



\section{Notation and conventions}\label{sec:conventions}
\label{sec:summary_of_notation}

\paragraph{Notation}
\begin{itemize}\itemsep1.5pt
    \item $X,Y,B,C$ are schemes, $\dim_{\CC}C=1$
    \item $\shE,\shF,\shG$ are coherent sheaves.
    \item $\shE^*$ is the dual sheaf of $\shE$.
    \item $E,F,Q$ are vector bundles.
    \item $r_E$ is the rank of the vector bundle $E$.
    \item $L,\shL,\scrL$ and $A$ are line bundles.
    \item $\shA,\shB$ are graded sheaves of $O_B$-algebras which are generated at degree 1.
    \item $F_\bullet,G_\bullet$ and $H_\bullet$ are filtrations of a vector space or a sheaf.
    \item $\GG_m$ is the multiplicative group $\Spec \CC[s,s^{-1}]$, often denoted as $\CC^\times$.
    \item $\AA^1$ is the complex affine line $\Spec\CC[x]$.
    \item $\PP^n$ is the complex projective space $\Proj\CC[x_0,\dotsc,x_n]$.
    \item $\shProj_B\shA$ is the relative proj of $\shA$.
    \item $\PP\shF$ is the scheme $\shProj \bigoplus_{k=0}^\infty S^k\shF$.
    \item $λ,μ,ν$ are partitions of positive integers $|λ|$, $|μ|$ and $|ν|$, respectively, page \pageref{sec:schurdef}.
    \item $r$ is a finite strictly increasing sequence of natural numbers whose largest entry is smaller than a fixed integer $r_E$, page \pageref{def:rseq}.
    \item $S_λ(\shE)$ is the ring $\bigoplus_{k = 0}^\infty \shE^{kλ}$, page \pageref{def:shapealgebra}.
    \item $S(\shE)$ is the ring $\bigoplus_{k = 0}^\infty \shE^{(k)}$, page \pageref{ex:schurpows}.
    \item $\Flag_r(E)$ is the flag bundle of $r$-quotients of $E$, page \pageref{def:flagscheme}.
    \item $σ_{r_E,r}$ is the canonical partition corresponding to the integer ${r_E}$ and the tuple $r$, \pageref{def:canonical}.
    \item $B_i(E,λ)$ are Chern classes appearing in the expression for the Chern character of the bundle $E^λ$, page \pageref{thm:maina}.
    \item $A_i(E,λ)$ are special cases of $B_i(E,λ)$ for $λ = (k)$ for some natural number $k$, \pageref{lem:sym_chern}.
    \item $\test(X,L)$ is the set of test configurations on a polarised scheme $(X,L)$ \pageref{def:kstab}.
    \item $D_{λ,r_E}$ is the leading coefficient of the Hilbert polynomial of a polarised flag variety corresponding to the integer ${r_E}$ and the partition $λ$, \pageref{rem:leadingcoeff}.
    \item $N^λ_{ν,μ}$ are Littlewood-Richardson coefficients, page \pageref{eq:LWR}.
    \item $C_{g,E,A,λ}$ and $D_{E,λ,L,f}$ are positive coefficients appearing in the expressions for the Donaldson-Futaki invariant of a flag bundle, pages \pageref{eq:Cpos} and \pageref{eq:Ddef}.
    \item $F_\bullet,G_\bullet$ and $H_\bullet$ are filtrations, pages \pageref{rem:ringfilt} and \pageref{def:relativefilt}.
    \item $\fAlg$ is the category of admissibly filtered sheaves of algebras, \pageref{def:falg}.
\end{itemize}

\paragraph{Conventions and terminology}
\begin{itemize}\itemsep1.5pt
    \item A \emph{polarised variety} is a pair $(X,L)$, where $X$ is a complex variety and $L$ an ample line bundle on $X$.
    \item A vector bundle is identified with its locally free sheaf of sections
    \item We use the common abbreviation $m\gg 0$, which means that there exists an $m_0$ such that a statement holds for all $m>m_0$
    \item Given a sheaf $\shF$ on $B$, the fibre $\shF\otimes k(x)$ is written as $\shF_x$.
    \item Given a family $\X\rightarrow \AA^1$, we denote the fibres over closed points of $\AA^1$ by $\X_t$, where $t\in \AA^1$ and call the fibre $\X_0$ the \emph{central fibre}
    \item Let $h:\ZZ\rightarrow \QQ$ be a function, whose restriction to $\ZZ_{> k_0}$ for some positive number $k_0$ agrees with a polynomial. If we \emph{only} care about the asymptotics of $h(k)$ as $k$ tends to infinity, we will replace the \emph{function}, by its \emph{polynomial} and abuse notation by using the same symbol. So a Hilbert function becomes a Hilbert polynomial, a weight function becomes a weight polynomial and so on.
\end{itemize}

\section{Statements of selected results} 
\label{sec:summary_of_results}

Fix a smooth projective variety $B$ of dimension $b$ with an ample line bundle $L$. Let $E$ be an algebraic vector bundle of rank $r_E$ over $B$, $r$ a strictly increasing finite sequence of positive numbers and $\Flag_r(E)$ the bundle of $r$-flags of subspaces in $E^*$. Fix a partition $λ = (λ_1,\ldots,λ_l)$ with jumps given by $r$. Let $E^λ$ denote the vector bundle obtained from $E$ and the representation of $\GL(r_E,\CC)$ given by $λ$ and let $p$ be the projection from $\Flag_r(E)$ to $B$ and define the line bundle
\begin{equation}
    \shL_λ(A) = \shL_λ\otimes p^*A,
\end{equation}
on $\Flag_r(E)$, where $A$ is a line bundle on $B$ and $\shL_λ$ is the line bundle associated to the partition $λ$ (cf. Equation \eqref{eq:defLB}). We refer to Sections~ \ref{sec:schurdef}, \ref{sec:flag} and \ref{sec:relative_flag_varieties} for details.

We will often make the following assumption on our choice of partition.
\begin{definition}\label{def:diamond}
    We say that $λ$ and $r$ satisfy the assumption $\diamond$ if at least one of the following holds:
    \begin{enumerate}[(i)]
        \item the length $l(λ)$ of $λ$ is at most 4 (cf. page \pageref{sec:schurdef})
        \item $λ = t σ_{r_E,r}$ for some positive rational number $t$, where $σ_{r_E,r}$ is the canonical partition defined in Section~\ref{sec:chern}
    \end{enumerate}
\end{definition}
\begin{summarythm}{A}[Theorem \ref{thm:curve1}, Section \ref{sec:curve}]\label{summaryA}
    Let $C$ be a smooth projective curve of genus $g$, $E$ an ample vector bundle of rank $r_E$ on $C$ and $A$ an ample line bundle on $C$.
    \begin{itemize}
        \item If $E$ is slope polystable, then any polarised flag bundle $(\Flag_r(E),\shL_λ(A))$ admits a cscK metric. In particular $(\Flag_r(E),\shL_λ(A))$ is K-semistable.
        \item If $λ$ satisfies the assumption $\diamond$ and $E$ is slope unstable, then the flag variety $\Flag_r(E)$ of $r$-flags of quotients in $E$ with the polarisation $\shL_λ(A)$ is K-unstable. If $E$ is properly semistable, then the pair $(\Flag_r(E),\shL_λ(A))$ is properly K-semistable.
        \item Finally, if $E$ is simple, meaning that it has no nontrivial holomorphic automorphisms, and $g>1$, the YTD correspondence holds for any polarisation $\shL_λ(A)$ where $λ$ satisfies the assumption $\diamond$. In particular, $E$ is simple if it is stable.
    \end{itemize}
\end{summarythm}

\begin{summarythm}{B}[Theorem \ref{thm:B}, Section \ref{sec:anybase}]\label{summaryB}
    Let $E$, $B$ and $L$ be as in the beginning of the section. Assume that $r$ and $λ$ satisfy $\diamond$ and that $E$ is slope unstable. Then there exists an $m_0$ such that the flag variety $\Flag_r(E)$ of $r$-flags of quotients in $E$ with the polarisation $\shL_λ(L^m)$ is K-unstable for $m>m_0$.
\end{summarythm}

For $i$ between $1$ and $b$, define the cohomology class $B_i(E,λ)$ to be the Chow degree $i$ term in the expansion
\begin{equation}
    \ch E^λ = \rank E^λ (1 + B_1(E,λ) + B_2(E,λ) + \dotsb + B_b(E,λ))
\end{equation}
of the Chern character of $E^λ$.
\begin{summarythm}{C}[Theorem \ref{thm:maina}, Section~\ref{sec:chern}]\label{summaryC}
    Let $E$ be as in the beginning of the section and let $λ$ satisfy the assumption $\diamond$ for some $r_E$ and $r$, then
    \begin{equation}
        B_1(E,λ)= \frac{c_1(λ)}{r_E}c_1(E)
    \end{equation}
    and
    \begin{equation}\label{eq:STA}
        \begin{split}
                B_2(E,λ)&\equiv_1 \frac{h_2(λ)h_2(E)}{{r_E}({r_E}+1)} + \frac{c_2(λ)c_2(E)}{{r_E}({r_E}-1)} + H_λ A_2(E) + Z.
        \end{split}
    \end{equation}
    where $Z$ is independent of $λ$, and $h_i(λ)$ and $c_i(λ)$ denote the complete symmetric and elementary symmetric polynomials of $λ$, respectively. We denoted
    \begin{equation}
        A_2(E)=\frac{{r_E}-1}{2}\left(\frac{h_2(E)}{{r_E}({r_E}+1)}-\frac{c_2(E)}{{r_E}({r_E}-1)}\right)
    \end{equation}
    and
    \begin{equation}
        H_λ = \frac{r_Ec_1(λ) - \sum_i(2i-1)λ_i}{{r_E}-1}.
    \end{equation}
     The notation $\equiv_1$ means the following weak numerical equivalence: If $U$ and $V$ are $k$-cycles in $B$, then $U\equiv_1 V$ if $c_1(A)^{n-k}.(U-V)$ is the zero cycle for all line bundles $A\in \Pic B$. We also used $c_i(λ)$ and $h_i(λ)$ to denote the elementary and complete symmetric polynomials of degree $i$ for $λ$.
\end{summarythm}

\begin{summarythm}{D}[Theorem \ref{thm:grEx}, Chapter \ref{chap:sub}]\label{summaryD}
    Given any positive integers $p$ and $d$, there exist a K-unstable hypersurface of degree $d$ in a Grassmannian bundle of $p$-planes in a vector bundle on a smooth complex curve.
\end{summarythm}

In Chapter \ref{chap:filtrations} we define the notion of \emph{relative K-stability} and generalise a correspondence between filtrations and test configuration to this context \cite{szekelyhidi2011filtrations}. Let $p\colon Y\rightarrow B$ be a projective morphism and $\shL$ a relatively ample line bundle on $Y$. The definitions and the precise statements of the following two theorems is found in Chapter \ref{chap:filtrations}.
\begin{summarythm}{E}[Theorem \ref{thm:SCor}, Section \ref{sec:filtrations_and_k_stability}]\label{summaryE}
    There is a 1-1 correspondence between $p$-relative test configurations up to a natural identification and admissible finitely generated filtrations of the algebra $\bigoplus_{k=0}^\infty p_*\shL^k$.
\end{summarythm}
\begin{summarythm}{F}[Theorem \ref{thm:convTC}, Section~\ref{sec:operations_on_relative_test_configurations}]\label{summaryF}
    Without fixing a relatively ample line bundle, set of $p$-test configurations for $Y$ is, up to natural identifications, has a convex structure which fibres naturally over the cone of relatively ample polarisations. Moreover, the Donaldson-Futaki invariant is continuous in the variation of the convex combination.
\end{summarythm}
\begin{remark}
    The statements of Theorem E and Theorem F specialise to usual test configurations if we take $B$ to be a point.
\end{remark}
Theorem E and Theorem F immediately imply the following result, which we also believe to be new.
\begin{summarythm}{G}[Theorem \ref{thm:unstable_locus_is_open}]\label{summaryG}
    Let $X$ be a projective variety over the complex numbers. Then the locus of line bundles which are K-unstable is open in the cone of ample $\QQ$-line bundles with respect to the Euclidean topology.
\end{summarythm}
\chapter{Preliminaries}\label{chap:preliminaries} 
This chapter reviews preliminary material. We briefly review background on geometric invariant theory in Sections~\ref{sec:stab} and Sections~\ref{sec:review_of_stability}. Sections~\ref{sec:stability_of_vector_bundles} recalls the definition of Mumford stability of vector bundles and the Narasimhan-Seshadri extension of the Uniformisation Theorem to vector bundles. Sections~\ref{sec:schurdef}, Sections~\ref{sec:flag} and Sections~\ref{sec:relative_flag_varieties} review preliminaries on flag varieties and their relative counterpart, flag bundles.


\section{Group actions and linearisations}\label{sec:stab} 

In this section we recall basic notions of group actions on complex projective varieties \cite[Section 4.2]{Huybrechts2010}. In particular, we briefly describe the equivariant set-up for flag bundles and families of projective varieties over $\AA^1$, which we will use in later sections. Let $X$ be a complex projective scheme with a $G$-action, that is a regular map
\begin{equation}
    ρ:X\times G\rightarrow X
\end{equation}
The scheme $X$ together with the action $ρ$ is called a \emph{$G$-scheme}. This notion also extends to sheaves on $X$. Let $\shF$ be a coherent sheaf on $X$. A \emph{$G$-linearisation} of $\shF$ is an isomorphism of $\shO_{X\times G}$-sheaves $Φ: ρ^* \shF\ra p_1^*\shF$ satisfying the condition
\begin{equation}
    (\id_X\times μ)^* Φ = p_{12}^*Φ\circ(σ\times \id_G)^*Φ,
\end{equation}
where $p_{12}$ denotes the projection $p_{12}:X\times G\times G\ra X\times G$ onto the first two factors. A $G$-linearisation on $F$ induces an action on the schemes functorially constructed from $\shF$. A $G$-linearised sheaf is often referred to simply as a $G$-sheaf. If we assume that $\shF$ is locally free and denote the total space of $\shF$ by $F$, linearisations are equivalent to $G$-actions on $F$ whose projections $F\rightarrow X$ are equivariant and restrict to linear isomorphisms
\begin{equation}
    F_x \cong F_{ρ(x,g)}
\end{equation}
for all $(x,g)\in X\times G$. \emph{A polarised $G$-variety} $(X,L)$ is a $G$-variety $X$ with an ample line bundle with a $G$-linearisation.

The most important actions in the theory of K-stability are ones by the complex multiplicative group $\GG_m$.
\begin{example}[Actions of the multiplicative group on polarised varieties]\label{ex:Gm}
    Consider an action of the multiplicative group $\GG_m$ over $\kk$ on a projective variety $(X,L)$, where $L$ is a very ample line bundle. Let $R$ be the ring $H^0(X,\bigoplus_{k=0}^\infty L^k)$. Then the $\GG_m$-linearisation on the line bundle $L$ determines a representation of the group $\GG_m$ on the vector space $H^0(X,L^k)$ for all $k\geq 0$ by setting
    \begin{equation}\label{eq:GmRep}
        s.f(x) = f(s^{-1}x)
    \end{equation}
    for all $s\in\GG_m$, $x\in X$ and $f\in H^0(X,L^k)$. This determines a homomorphism
    \begin{equation}
        h\colon R\rightarrow R[s]
    \end{equation}
    by sending
    \begin{equation}
        f\mapsto s^{-w(f)}f
    \end{equation}
    for any $f$ which lies the space of \emph{weight} $-w(f)$ elements of the representation. If we extend this map linearly, it follows from Equation~\eqref{eq:GmRep} that the homomorphism $h$ preserves the grading on $R$. Conversely, any $\GG_m$-action on a very amply polarised complex scheme arises from a homomorphism $R\rightarrow R[s]$, where $R$ is a graded algebra.

    Another way to describe the map $h$ is by lifting the $\GG_m$-action to an action on the affine cone \cite{mumford1994geometric}
    \begin{equation}
        \Spec R\times\GG_m\rightarrow\Spec R,
    \end{equation}
    which by definition corresponds uniquely to a homomorphism
    \begin{equation}
        R\rightarrow R[s,s^{-1}].
    \end{equation}
\end{example}

\begin{lemma}\label{lem:induceG}
    Given a $G$-sheaf $\shF$, the Schur powers and shape algebras of the sheaf are $G$-sheaves. Moreover, if $\shA$ is a sheaf of $\shO_X$-algebras with a $G$-linearisation which respects the algebra structure, the relative Spec construction yields a $G$-scheme $Y$ such that the natural morphism $Y\rightarrow X$ is $G$-invariant. If $\shA$ is graded, the same statement is true for the relative $\shProj$ where the $\shO(1)$-line bundle comes with a natural linearisation of the action.
\end{lemma}
\begin{proof}
The Schur power part of the statement follows as tensor algebras of linearised sheaves have natural induced linearisations. We refer to \cite[pp. 94-95]{Huybrechts2010} for the remaining statements whose proofs are straightforward verifications.
\end{proof}


\section{Geometric invariant theory}\label{sec:review_of_stability} 
We review aspects of Mumford's geometric invariant theory (GIT). The books \cite{mumford1994geometric} and \cite{mumford1977stability} have been an invaluable reference, and contain the germs of many ideas contained in this work and in the theory K-stability at large.

The idea of stability appears when one attempts to form quotients in the category of quasi-projective varieties. Mumford realised that given an action of an algebraic group $G$ on a polarised variety $(X,L)$, there is a $G$-invariant open subset $X_s$ of \emph{stable locus} such that the orbit set $X_s/G$ can be given a natural structure of a quasiprojective variety. Moreover, the Zariski closure of $X_s/G$ can be naturally identified with a quotient of a larger set $X_{ss}$ of \emph{semistable locus} by $G$. This construction is called the \emph{GIT quotient} of $X$ by $G$ and it depends on a choice of $G$-linearisation on the line bundle $L$.

We begin with the definition of stability for linear representations. Suppose $G$ is a complex algebraic group with a linear representation $V$. We say that a point $p\in V$ is
\begin{itemize}
    \item \emph{stable} if $0\not\in \overline{G.p}$ and $\mathrm{Stab}_G(p)$ is finite,
    \item \emph{semistable} if $0\not \in \overline{G.p}$ and
    \item \emph{unstable} if $0\in\overline{G.p}$.
\end{itemize}
For any $x\in \PP V$, we say that $x$ is stable, semistable or unstable if some (and hence each) nonzero lift of $x$ to $V$ is.

There is an induced action of $G$ on the vector spaces $H^0(X,L^k)$ for all $k\in\NN$ given by
\begin{equation}
    (g.s)(p)=s(g^{-1}p)
\end{equation}
for $s\in H^0(X,L^k)$ and $p\in X$.

\begin{definition}\label{def:GITstab}
        Let $x$ be a point in a scheme $X$ with an ample line bundle $L$.
        \begin{itemize}
            \item $x$ is stable (with respect to a chosen linearisation) if there is an invariant section $s\in H^0(X,L^k)$ for some $k\in\NN$ such that the open set $U_s=\{x: s(x)\neq 0\}$ is affine and invariant, and the orbits of closed points in $U_s$ are closed.
            \item $x$ is polystable if there is an invariant section $s\in H^0(X,L^k)$ for some $k\in\NN$ such that the open set $U_s=\{x: s(x)\neq 0\}$ is affine and invariant, and the orbits of closed points in $U_s$ are closed in the semistable locus,
            \item $x$ is semi-stable if there is an invariant section $s\in H^0(X,L^k)$ for some $k\in\NN$ such that the open set $U_s=\{x: s(x)\neq 0\}$ is affine and invariant and
            \item $x$ is unstable otherwise.
        \end{itemize}
\end{definition}

\paragraph{One parameter subgroups and the Hilbert-Mumford criterion}
Mumford discovered a powerful criterion for determining whether a point is stable in the sense of Definition \ref{def:GITstab}. A \emph{one parameter subgroup} (1-PS) of a complex algebraic group $G$ is a homomorphism $χ:\GG_m\rightarrow G$. Assume that $G$ acts on $X$ and that $ρ:X\times G\rightarrow X$ is proper. Given a point $x\in X$, one parameter subgroup $χ$ determines a morphism
\begin{equation}
    f\colon\AA^1\rightarrow X
\end{equation}
which maps $x$ to a point in the closure of the orbit of $χ$. Then the induced $\GG_m$-linearisation of the action $ρ\circ χ$ on $L$ restricts to a character of $\GG_m$ on the complex line $\restr{f^* L}{\{0\}}$. Let $χ(t)=t^r$ be this character and define the integer
\begin{equation}
    μ^L(x,χ) = -r.
\end{equation}
\begin{proposition}[The Hilbert-Mumford criterion]
    Let $X,L,G$ and $ρ$ be as above and $x$ a point in $X$. Then
    \begin{itemize}
        \item $x$ is stable if and only if $μ^L(x,χ)>0$ for all 1-PS $χ$.
        \item $x$ is semistable if and only if $μ^L(x,χ)\geq 0$ for all 1-PS $χ$
        \item $x$ is unstable otherwise.
    \end{itemize}
\end{proposition}
\begin{remark}[Stability of varieties]
    The Hilbert scheme and the Chow scheme are two constructions, which are powerful tools in the study of families of projective varieties. They enable us to identify a projective scheme $(X,L)$ with a fixed embedding $\PP(H^0(X,L^r)^*)$ as a point in a parameter scheme. The choice of basis on $\PP(H^0(X,L^r)^*)$ implies a natural GIT problem for Hilbert and Chow stability, whose solution ultimately depends on understanding the Hilbert-Mumford criterion on certain Grassmannians into which both the Hilbert scheme and the Chow scheme are embedded.

     The stability of $(X,L)$, in either the Hilbert scheme or the Chow scheme, depends on the parameter $r$. Mumford suggested study of asymptotic stability, or whether there exists an $r_0$ such that $(X,L)$ is stable for $r>r_0$. Mabuchi proved the equivalence of asymptotic Hilbert stability and asymptotic Chow stability in \cite{Mabuchi2008}. K-stability, which will be defined in Chapter \ref{chap:stability}, is a minor modification on the Hilbert-Mumford criterion for asymptotic stability of $(X,L)$.
\end{remark}

\section{Stability of vector bundles} \label{sec:stability_of_vector_bundles}
Let $\shE$ be a coherent sheaf of rank ${r_E}$ on a smooth projective variety $B$. Define the \emph{determinant} of $\shE$ by
\begin{equation}
    \det \shE = (\ex^e \shE)^{**}.
\end{equation}
The define \emph{first Chern class} by $c_1(\det \shE)$, and the degree and the slope of $\shE$ by
\begin{equation}
\deg \shE=\int_X c_1(\shE).c_1(L)^{n-1}.
\end{equation}
and
\begin{equation}
    μ_E=\deg \shE/\rank \shE,
\end{equation}
respectively. If $\shE$ is locally free in a subset $U\subset B$ whose complement is contained in a codimension 2 subscheme, we say that $\shE$ is \emph{locally free in codimension 2}. In this case the first Chern class of $\shE$ can defined to be the pushforward
\begin{equation}
    c_1(\shE) = (i_U)_* c_1(\restr{\shE}{U}),
\end{equation}
where
\begin{equation}
    i:U\rightarrow B
\end{equation}
is the inclusion.

Let $\operatorname{TF}(\shE)$ denote the set of torsion free subsheaves $\shF$ of $\shE$ with $0 < \rank \shF < \rank \shE$ \cite{kobayashi2014differential}.
\begin{definition}[Mumford-Takemoto slope stability {\cite[Definition 1.2.12]{Huybrechts2010}}]\label{def:slopestab}
 Let $E$ be a vector bundle. We say that $E$ is
    \begin{itemize}
        \item \emph{slope stable} if $μ_\shF<μ_E$ for all $\shF\in \operatorname{TF}(E)$
        \item \emph{slope polystable} if $μ_\shF\leq μ_E$ for all $\shF\in \operatorname{TF}(E)$ and in the case of equality, $E$ is a direct sum $\shF\oplus \shQ$ with $μ_\shF = μ_\shQ$,
        \item \emph{slope semistable} if $μ_\shF\leq μ_E$ for all $\shF\in \operatorname{TF}(E)$ and
        \item \emph{slope unstable} otherwise.
    \end{itemize}
    A torsion free subsheaf $\shF$ with $μ_\shF > μ_E$ is called a \emph{destabilising} subsheaf.
\end{definition}

The following generalisation of the Uniformisation theorem holds for polystable vector bundles on Riemann surfaces.
\begin{proposition}[{\cite[Theorem 2.7]{kobayashi2014differential}}]\label{prop:ns}
    A vector bundle $E$ of rank ${r_E}$ on a Riemann surface $Σ$ is slope polystable if and only if it admits a projectively flat structure, that is the associated $\PGL(\CC,\rank E)$-bundle $\EE$ is \emph{flat}, meaning that it arises from a representation
    \begin{equation}
        ρ:π_1(Σ)\rightarrow \PGL(r_E,\CC)
    \end{equation}
    of the fundamental group $π_1(Σ)$ of $Σ$ as the quotient
    \begin{equation}
        \EE = \widetilde{Σ}\times_ρ  \PGL(r_E,\CC).
    \end{equation}
\end{proposition}

\begin{remark}[The Hitchin-Kobayashi correspondence]
    If $E$ is a vector bundle and $h$ is a Hermitian metric with curvature $F_h$, we say that $h$ is Hermitian-Einstein if it satisfies
    \begin{equation}
        \sqrt{-1} Λ_ω F_h = μ_E \id_E,
    \end{equation}
    where $Λ_ω$ is the dual of the Lefschetz operator \cite[pp. 114-115]{huybrechts2006complex}.

    The YTD correspondence is closely related to a result which relates the existence special connections on vector bundles to Mumford stability. This is called the Hitchin-Kobayashi correspondence proved by Narasimhan-Seshadri \cite{Narasimhan1965}, Donaldson \cite{donaldson1987infinite} and Uhlenbeck-Yau \cite{uhlenbeck1986existence}. It states that a Hermitian vector bundle $E$ on a projective manifold $(M,L)$ admits a Hermitian-Einstein metric if and only if it is Mumford stable.
\end{remark}
%


\section{Schur functors}\label{sec:schurdef} 
We define Schur functors using the classical formulation in terms of Young symmetrisers. Let $A$ be a finitely generated $\QQ$-algebra and $M$ is a finite $A$-module of dimension $d_M$.

A \emph{partition} $λ=(λ_1,\ldots,λ_l)$ is a finite nonincreasing sequence of natural numbers. Define the \emph{length} $l(λ) = l$ and the area $|λ|=\sum_{i=1}^lλ_i$ of $λ$. Also define the natural operations on partitions. Let $λ$ and $μ$ be partitions of equal length and let $k$ and $n$ be a natural numbers. Define
\begin{itemize}
    \item the componentwise sum $λ + μ$,
    \item the componentwise product $λμ$,
    \item the sum and product with a natural number, understood to be a constant partition of the correct length, and
    \item repeated indices $(k^n) \defeq  (\underbrace{k,\dotsc,k}_n)$.
\end{itemize}

A partition $λ$ is uniquely represented by a Young diagram $D_λ$ consisting of $λ_i$ boxes in the $i$th row. Define the conjugate partition of $λ$ to be the partition $λ'$ represented by the Young diagram obtained from $D_λ$ via reflection in the diagonal axis of reflection starting from the top left corner. In other words, the difference between $D_λ$ and $D_{λ'}$ is that the roles of rows and columns are reversed.

\begin{figure}[h]
    \vspace{0.5cm}
    \begin{center}
    $\yng(4,2,1)\qquad \mapsto \qquad \yng(3,2,1,1)$
    \caption*{Example of conjugating the partition $λ=(4,2,1)$ by a reflection of its Young diagram. }
    \end{center}
\end{figure}

\begin{definition}\label{def:schur}
    Let $A$ be a ring containing $\QQ$, let $λ$ be a partition such that $d=|λ|$ and denote $I=(i_1,\ldots,i_d)$. Consider the $d$th tensor power of $M$ and let $m_{i_1},\ldots,m_{i_d}$ be elements of $M$. Denote $m_I=m_{i_1}\otimes\dotsm\otimes m_{i_{|λ|}}$ and define map
    \begin{equation}
        c_λ:m_I\mapsto \frac{1}{d_λ}    \sum_{σ,τ} (\sgn τ) m_{σ\circ τ(I)},
    \end{equation}
    called the \emph{Young symmetriser}. The rational number $d_λ$ is chosen so that $c_λ$ is idempotent. This requirement fixes $d_λ$ uniquely. Explicitly, we have $d_λ = d_M!/\dim M^λ$.

     The summation is taken over all $σ$ ($τ$, respectively) which preserve the rows (columns) of the diagram. Define the \emph{Schur power $M^λ$} of $M$ associated to the partition $λ$ by
    \begin{equation}
        M^λ =c_λ\left(M^{\ot|λ|}\right).
    \end{equation}
\end{definition}

\begin{remark}
    The proof that the rational number $d_λ$ exists can be found in \cite[Theorem 4.3]{FultonRep}.
\end{remark}

\begin{lemma}\label{lem:funct}
    The Schur power construction is a functor from the category of $A$-modules to itself and it commutes with change of base. We use the term \emph{Schur functor} synonymously with the term Schur power.
\end{lemma}
\begin{proof}
    Let $M$ and $N$ be $A$-modules and let $f\colon M\rightarrow N$ be a homomorphism. Then the natural homomorphism $f^λ$ defined as restriction of
    \begin{equation}
        m_1\otimes \dotsm \otimes m_d\mapsto f(m_1)\otimes \dotsb \otimes f(m_d)
    \end{equation}
    is well defined as a map $M^λ\ra N^λ$. It is clear that this construction respects identity and composition.

    Tensor powers commute with base change so the same is true for Schur powers.
\end{proof}
In particular, Schur powers are therefore defined on the category of coherent sheaves on schemes.
\begin{definition}\label{def:schursheaf}
    Given a quasicoherent sheaf $\shF$ and a partition $λ$, we define the Schur power $\shF^λ$ to be the quasicoherent sheaf locally obtained by Definition \ref{def:schur}.
\end{definition}
To be more explicit, let $\{U_α\}$ an open affine cover of $B$ such that $\restr{\shF}{U_α}$ is the quasicoherent sheaf corresponding to a $\shO_B(U_α)$-module. We define $\shF^λ$, the Schur power of $\shF$ for the partition $λ$, by its restrictions to $\restr{\shF^λ}{U_α}$. The transition maps are induced by localisation and functoriality. Denote the Schur power of $\shF$ by $\shF^λ$.

\begin{definition}\label{def:rseq}
    If $r$ is a finite increasing sequence of natural numbers and $λ$ is a partition, we say that \emph{the jumps of $λ$ are given by $r$} if $λ_i>λ_{i+1}$ precisely at indices $i$ belonging to $r$ with the additional requirement that $λ_e$ is zero for some integer ${r_E}$. Later, the integer ${r_E}$ will be taken to be the dimension of a fixed vector space or the rank $r_E$ of a fixed vector bundle $E$. Denote the set of such partitions by $\shP(r)$.
\end{definition}

The following algebra is at the centre of a relationship between geometry, algebra and representation theory that we make use of in later chapters.
\begin{definition}[Algebra structure \cite{towber1977two}]\label{def:shapealgebra}
    Given an $A$-module $M$ we define the \emph{universal shape algebra}
    \begin{equation}\label{eq:shapealgebra}
        \mathbb{S}(M) = \bigoplus_λ M^λ
    \end{equation}
    where the summation is over all partitions $λ$ and the ring structure is defined by the projection
    \begin{equation}
        m_λ\ot m_{μ} \mapsto d_{λ+μ}^{-1}c_{λ+μ} (m_λ\ot m_μ).
    \end{equation}
    for any $m_λ\in M^λ$ and $m_μ\in M^μ$.

    We also define two natural subalgebras of $\mathbb{S}(M)$. Given any partition $λ$, we define the $\ZZ$-graded subalgebra
    \begin{equation}
        S_λ(M) = \bigoplus_{k=0}M^{kλ} = A\oplus M^λ \oplus M^{2λ}\oplus\dotsb
    \end{equation}
    called the \emph{shape algebra} of $M$ for the partition $λ$. In the case $λ=(k)$ we simply write
    \begin{equation}
        S_{(k)}(M)  = S(M)
    \end{equation}
    for the symmetric algebra of $M$. Given a finite strictly increasing sequence of natural numbers $r$, we define the $\ZZ^c$-graded subalgebra
    \begin{equation}
                \mathbb{S}_r(M) = \bigoplus_{ν\in \shP(r)} M^{ν}
    \end{equation}
    called the \emph{total coordinate ring of the scheme of $r$-flags in $M$}.

    The terminology is justified in Section~\ref{sec:flag} and Section~\ref{sec:relative_flag_varieties}.
\end{definition}
\begin{proposition}
    The algebras $\mathbb{S}(M)$, $S_λ(M)$ and $\mathbb{S}_r(M)$ are associative and commutative $A$-algebras. The algebra $S_λ(M)$ is finitely generated as an $A$-algebra.
\end{proposition}
\begin{proof}
    Associativity and commutativity follow directly from the properties of the Young symmetriser. Finite generation is clear since $S_λ(M)$ is generated in degree one.
\end{proof}

\begin{example}[Examples of Schur functors in the category of coherent sheaves]\label{ex:schurpows}
    Let $\shE$ be a coherent sheaf on an integral scheme $B$. Define the \emph{rank} $e=\rank \shE$ of $\shE$ to be the dimension of the fibre of $\shE$ over the generic point of $B$ \cite[p. 74]{hartshorne1977algebraic}. We define the \emph{determinant} of $\shE$ by
    \begin{equation}
        \det \shE = \shE^{(1^{r_E})},
    \end{equation}
    which we also denote by $\ex^e \shE$, and the \emph{symmetric power} of $\shE$ by
    \begin{equation}
        S^k \shE = \shE^{(k)}.
    \end{equation}
\end{example}

\begin{remark}[Schur functors for vector bundles]\label{rem:schurvb}

If $E$ is a locally free sheaf, then there is a convenient description of the Schur power. Let $\EE$ be the frame bundle of the vector bundle corresponding to $E$ with fibre $\GL(V)$, where $V$ is a dimension $\rank E$ complex vector space. Then we may define $E^λ$ to be the sheaf of sections of the vector bundle
\begin{equation}
    \EE\times_G V^λ.
\end{equation}
\end{remark}
\begin{remark}
    Either from Remark \ref{rem:schurvb} or from the definition of a Schur functor, we see that $E^λ\ot L^{c_1(λ)} = (E\ot L)^λ$, where $c_1(λ)$ is the sum $\sum_{i=1}^l λ_i$.
\end{remark}

The following proposition will be important for applying the standard constructions of algebraic geometry to shape algebras.
\begin{proposition}[Positivity of Schur powers \cite{Hartshorne1966}]\label{prop:schurpos}
    If the vector bundle $E$ is ample, then the Schur power $E^λ$ is ample for any partition $λ$.
\end{proposition}


\section{Flag varieties}\label{sec:flag}
In this section we present a short introduction to classical flag varieties and the Borel-Weil theorem for the general linear group, which relates the space of sections of an equivariant line bundle on a flag variety to a representation of the general linear group. Our main reference is Weyman's book, but we use a dual convention for partitions \cite[Chapters 2 and 3]{Weyman}.

Given a vector $(r_1,\ldots, r_c)$ of strictly increasing integers, we define an \emph{$r$-flag} of quotients of a vector space $V$ to be a sequence
\begin{equation}
    V\rightarrow V_c\rightarrow V_{c-1}\rightarrow\dotsb\rightarrow V_1\rightarrow 0
\end{equation}
of successive quotients where $\dim V_i =r_i$ which we assume not to be injective for all $i$. Dually, this corresponds to a sequence
\begin{equation}
    0\subset V_1^*\subset\dotsm\subset V_c^*\subset V^*
\end{equation}
of nested subspaces. We make the assumption that the largest element of $r$ is smaller than $\dim V$ from now on without further mention.

Let $G = \GL(e,\CC)$ and consider the subgroup of matrices of the form
\begin{equation}
    \begin{pmatrix}
    B_1 & * & * & \cdots & * \\
    0 & B_2 & * & \cdots & *\\
    \vdots & \vdots & \vdots & \vdots & \vdots \\
    0 & 0 & 0 & B_{c} & *\\
    0 & 0 & 0 & 0 & B_{c+1}
    \end{pmatrix},
\end{equation}
where $B_i$ is in $\GL(r_i-r_{i-1},\CC)$ and the entries marked with $*$ are arbitrary. Matrices of this form is the isotropy subgroup $P_r\subset G$ of a flag of coordinate subspaces
\begin{equation}\label{eq:flag1}
    0 = \langle{e_1,\ldots ,e_{r_1}}\rangle \subset \langle{e_1,\ldots, e_{r_2}}\rangle\subset\dotsm\subset \langle{e_1,\ldots, e_{r_c}}\rangle\subset \CC^e.
\end{equation}

Let $V$ be a vector space of dimension ${r_E}$ and $r$ a properly increasing sequence of positive integers. A \emph{classical flag variety} $\Flag_r(V)$ is the set of all possible nested subspaces
\begin{equation}\label{eq:incidence}
    0 = V_{r_{1}} \subset V_{r_{2}}\subset\dotsm\subset V_{r_c}\subset V_{r_{c+1}}= V^*
\end{equation}
where $\dim V_{r_j}=r_j$ for all $j$. The set of flags of this type has the structure of a homogeneous space $G/P_r$ where $P_r$ is the stabiliser of the flag in Equation~\eqref{eq:flag1}.

\begin{remark}
    There is a 1-1 correspondence between quotients and subspaces of the complementary dimension. Dualising $V$ in Equation~\eqref{eq:incidence} corresponds to working with quotients of $V$ instead of subspaces.
\end{remark}

The Pl{\"u}cker embedding, which sends each plane spanned by vectors $v_1,\ldots,v_{r_j}\in V^*$ to the point $[v_1\wedge\dotsb\wedge v_{r_j}]\in\PP(Λ^{r_j}V)$, determines an embedding from the flag variety $\Flag_r(V)$ to the product of projective spaces
\begin{equation}
    \PP=\PP(\Lambda^{r_1}V)\times\dotsb\times\PP(\Lambda^{r_c}V).
\end{equation}
The image is cut out by incidence relations determined by Equation~\eqref{eq:incidence} and quadratic relations on each of the factors $\PP(\Lambda^{r_j}V)$. The coordinate ring of $\Flag_r(V)$ can be beautifully written in terms of Schur functors as follows.
\begin{proposition}[{\cite[Proposition 3.1.9]{Weyman}}]
    Equip the coordinate ring of $\PP$ with its standard $\NN^c$-grading. Then the $(s_1,\ldots,s_c)$-component of the multigraded coordinate ring $\CC[\Flag_r(V)]$ is isomorphic to the Schur module $V^{λ}$, where the conjugate of $λ$ satisfies
    \begin{equation}\label{eq:part1}
        λ' = (r_c^{s_c},\ldots,r_1^{s_1}).
    \end{equation}
\end{proposition}
Another way to write this proposition is by using the Borel-Weil theorem, which we state in the case of an ample line bundle on a flag variety of the general linear group. Let $λ$ be a partition of length $l<e$. Then we can define a subgroup $P_r$ of $G$ by letting $r$ be the set of indices $i$ such that $λ_i<λ_{i+1}$. Define the line bundle $\shL_λ$ by
\begin{equation}\label{eq:plucker}
    \shL_λ =  p^*_1 \shO_{\PP(\Lambda^{r_1}V)}(s_1)\otimes \dotsb\otimes p^*_c \shO_{\PP(\Lambda^{r_c}V)}(s_c),
\end{equation}
where the $s_i$ are determined by the requirement $λ'=(r_1^{s_1},\ldots,r_c^{s_c})$. Then the classical Borel-Weil theorem \cite[Théorème 4.]{serre1954representations}, \cite[Proposition 10.2]{bott1957homogeneous} implies that
\begin{equation}
    H^0(\Flag_r(V),\shL_λ) = V^λ
\end{equation}
for $s_i> 0$.

\begin{remark}\label{rem:SMult}
    A basic fact is that the tensor product of two line bundles $\shL_λ$ and $\shL_μ$ indexed by partitions is given by
    \begin{equation}
        \shL_λ\ot\shL_μ=\shL_{λ+μ}.
    \end{equation}
    Note that only globally generated line bundles can be written using partitions. Formally, it is common to denote the dual of a line bundle $\shL_λ$ by $\shL_{-λ}$ (cf. the proof of Proposition \ref{prop:gentype}).
\end{remark}

\section{Flag bundles and the Borel-Weil Theorem}\label{sec:relative_flag_varieties} 

Let $B$ be a projective scheme and let $E$ be a vector bundle of rank ${r_E}$ on $B$.

\begin{definition}
    Let $G$ be a group. \emph{A (Zariski locally trivial) principal $G$-bundle} over $B$ is a morphism $p\colon Y\rightarrow B$ such that
\begin{itemize}
    \item $Y$ is equipped with a $G$-action under which $p$ an invariant map, and
    \item there exists a Zariski open cover $\{U_i\}_{i\in I}$ with an isomorphism $t_i p^{-1}U_i \cong  G\times U_i$ for all $i\in I$ such that $G$ acts by left translation on itself and trivially on $U_i$.
\end{itemize}
\end{definition}

Let $\EE$ be the \emph{frame bundle} of $E$ constructed as follows. Let $U_1,\dotsc, U_N$ be open subsets of $B$ such that
\begin{equation}
    \bigcup_{i=1}^N U_i = B
\end{equation}
and $\restr{E}{U_i}\cong U_i\times \CC^e$ and define $\EE$ to be the principal $\GL(r_E,\CC)$-bundle obtained from the collection $U\times \GL(r_E,\CC)$ with the same transition functions as $E$. The natural $\GL(r_E,\CC)$-action on $\EE$ is algebraic.

Define the \emph{relative flag variety} or \emph{flag bundle} $\Flag_r(E)$ to be the quotient $\EE/P_r$ and let $p_r:\Flag_r(E)\ra B$ be the projection. We often refer to $\Flag_r(E)$ as simply the flag variety of $E$ of $r$-quotients.

There is a sequence of tautological vector bundles
\begin{equation}
    0=\shR_{0}\subset \shR_1\subset \dotsb \subset \shR_c \subset \shR_{c+1} = p_r^*E^*,
\end{equation}
on $\Flag_r(E)$, where $\rank\shR_i=r_{t-i}$. The fibre of $\shR_i$ at $x\in\Flag_r(E)$ is the $r_i$-plane in $E^*$ determined by $x$.

Define the line bundle $\shL_λ$ on $\Flag_r(E)$ to be the pullback of the $Π_{i=1}^cp_i^*\shO(s_i)$ line bundle on
\begin{equation}\label{eq:defLB}
    \Flag_r(E)\hookrightarrow\PP(\Lambda^{r_1}E)\times\dotsb\times\PP(\Lambda^{r_c}E),
\end{equation}
which can also be written as the line bundle
\begin{equation}
    (\det \shR_1)^{s_1}\otimes\dotsm\otimes(\det\shR_c)^{s_c}
\end{equation}
with the same relationship between the $s_i$ and $λ$ as in Equation \ref{eq:plucker}.

The Borel-Weil-Bott theorem computes the cohomology of vector bundles which can be written as tensor products of Schur powers of the successive quotients $\shR_i/\shR_{i-1}$ for $i=1,\ldots,t+1$, with the vector bundle $\shR_{t+1}$ understood to be $p^*E$. We state the theorem for line bundles $\shL_λ$, where $λ$ is a partition whose jumps are given by $r$.
\begin{proposition}[{\cite[Theorem 4.1.4)]{Weyman}}]\label{prop:pushforward}
Let $λ$ be a partition in $\shP(r)$, and $r$ and $s$ are as above. In other words, $λ_{i}>λ_{i-1}$ if and only if the index $i$ is contained in $r$, in which case $λ_i-λ_{i-1} = s_{t-i}$. The derived pushforwards of $\shL_λ$ satisfy
\begin{equation}\label{eqbbw}
    p_* \shL_λ = E^{λ},
\end{equation}
and
\begin{equation}
    R^i p_* \shL_λ = 0,
\end{equation}
for $i>0$.
\end{proposition}
\begin{remark}
    Since the conjugate of a partition $λ\in\shP(r)$ can be written as $\left(r_c^{s_c},r_{c-1}^{s_{c-1}},\ldots,r_1^{s_1}\right)$ for some positive integers $s_1,\ldots,s_c$, we have
    \begin{equation}
        λ=\left({S_c}^{r_1},S_{c-1}^{r_2-r_1},\ldots,S_1^{r_c-r_{c-1}}  \right),
    \end{equation}
    where $S_n = \sum_{i=1}^n s_i$.
\end{remark}

\begin{remark}[The coordinate algebra of a flag bundle]\label{rem:universal}
     The shape algebra of a vector bundle is defined by functoriality and Definition \ref{def:shapealgebra}. An explicit description of the generators and relations of the sheaf of shape algebras locally shows that it is isomorphic to the sheaf of algebras determined by the Pl\"{u}cker embedding.  In other words the flag bundle $\Flag_r(E)$ is isomorphic to a relative projectivisation
     \begin{equation}
        \shProj_B S_λ(E)
     \end{equation}
     of the shape algebra. The consequence of this is that the algebra $\mathbb{S}(E)$ is the relative analogue of a total coordinate ring.

     We are not aware of a reference for the above statements, but it follows from the local statement \cite[Chapter 9]{Fulton1997}. If $\Flag_r(E_p)$ is a flag fibre over a point $p\in B$, then the equations of $\Flag_r(E_p)$ inside $\PP(E_p^λ)$ extend to a neighborhood of $p$ where $E$ is a trivial vector bundle. By taking the sheaf of ideals generated locally in this way we get the relations of $S_λ(E)$ inside $S^k(E^λ)$.
\end{remark}
Viewing a flag bundle as a relative projectivisation of a shape algebra implies a natural generalisation to arbitrary coherent $\shO_B$-modules.
\begin{definition}\label{def:flagscheme}
     If $\shE$ is a coherent $\shO_B$-module, we define the \emph{relative scheme of $r$-flags} (or \emph{relative flag scheme})
    \begin{equation}
        \Flag_r(\shE) = \shProj_B S_λ(\shE).
    \end{equation}
    It is naturally endowed with a relatively ample line bundle, determined by the pair $(\shE,λ)$, which we also denote by $\shL_λ$. We refer to this line bundle simply the \emph{Serre line bundle} on $\Flag_r(\shE)$ if $λ$ is clear from context.
\end{definition}

The statement of the following Lemma holds more generally \cite[Proposition II.7.10]{hartshorne1977algebraic}, but we prove a special case to spell out the relationship between the line bundle $\shL_λ$ and the projective embeddings of $\Flag_r(E)$.
\begin{lemma}\label{lem:m}
    Let $E$ be a vector bundle of $\shO_B$-algebras and let $S_λ(E)$ be a shape algebra for the partition $λ$ and let $p$ be the projection $\Flag_r(E)\rightarrow B$. There exists an $m_0$ such that the line bundle $\shL_λ(L^m)$ is ample for $m\gg 0$. Morever, if $E$ itself is ample, then $\shL_λ$ is ample.
\end{lemma}
\begin{proof}
    Assume that $E$ is a vector bundle on $B$. For any $k>0$ and $m>kc_1(λ)$ we have a natural isomorphism
     \begin{equation}
        \left(\Flag_r(E),\shL_λ(L^m)\right)\cong (\Flag_r(E\otimes L^k),\shL_λ(L)(L^{m-kc_1(λ)})).
    \end{equation}
 The vector bundles $\ex^{r_i}(E\otimes L^k)$ are ample for all $i=1,\ldots,c$ by \cite[Corollary 5.3]{Hartshorne1966}, so the hyperplane bundles on $\PP(\ex^{r_i}(E\otimes L^k))$ are ample for $i=1,\ldots, c$. We can regard the pair $\left(\Flag_r(E),\shL_λ(L^m)\right)$ as a subvariety in the product
    \begin{equation}
        \PP=\PP(\ex^{r_1}(E\otimes L^k))\times\cdots\times\PP(\ex^{r_s}(E\otimes L^k)),
    \end{equation}
    where the line bundle $\shL_λ(L^m)$ is the restriction of
    \begin{equation}
        \shO_\PP(s_1,\ldots,s_c)\ot p^*L^m
    \end{equation}
    which is ample. The map $p$ is the projection $p\colon \PP\rightarrow B$. The second claim follows from the same proof with $m=0$.
\end{proof}

\begin{lemma}\label{lem:FPic}
    The Picard group of a flag bundle $\Flag_r(E)$ is generated by line bundles of the form $\shL_λ(A)$, where $A$ is a line bundle on $B$ and the partition $λ$ is in $\shP(r)$.
\end{lemma}
\begin{proof}
    This proof goes along the same lines as \cite[Proposition 4.1.3]{Weyman}.
\end{proof}

Lemma \ref{lem:induceG} applies to flag bundles of $G$-linearised vector bundles.
\begin{proposition}\label{prop:FLin}
    Let $\shE$ be a $G$-linearised coherent $\shO_B$-module of rank ${r_E}$ on a $G$-variety $B$ and let $λ$ be a partition. Then the \emph{affine relative flag scheme}
        \begin{equation}
            \mathop{\shS \it{pec}}\nolimits_X S_λ(\shF)
        \end{equation}
        and the relative flag scheme
        \begin{equation}
            \shProj_X S_λ(\shF)
        \end{equation}
        are $G$-schemes. The relatively ample line bundle $\shL_λ$ comes with a natural $G$-linearisation.
\end{proposition}
\begin{proof}
    The diagram
    \[
    \begin{tikzcd}
    E_x^λ\otimes E_x^μ \arrow{r} \arrow[swap]{d} & E_x^{λ+μ} \arrow{d} \\
    E_{ρ(x,g)}^λ\otimes E_{ρ(x,g)}^μ \arrow{r} & E_{ρ(x,g)}^{λ+μ}
    \end{tikzcd}
    \]
    clearly commutes so the algebra $S_λ(E)$ is a sheaf of $G$-algebras with a linearization that preserves the grading. Hence, Lemma \ref{lem:induceG} implies that the scheme $(\shProj,\shL_λ)$ has a $p$-invariant $\GG_m$-action.
\end{proof}

\begin{remark}[The functorial definition of flag schemes]
    One may also define an object we call the \emph{flag-quot scheme} $\shDrap(r,\shE)$, which represents a functor from the category of schemes to the category of sets defined by
    \begin{equation}
        T\mapsto \left\{
        \begin{aligned}
            &\qquad\text{locally free quotients }\\
            &\shO_T\ot \shE \rightarrow \shQ_1\rightarrow\dotsb\rightarrow\shQ_c \rightarrow 0 \\
            & \text{on } B\times T  \text{ with ranks given by }  r.
        \end{aligned}
        \right\}
    \end{equation}
    We believe the scheme $\Flag_r(\shE)$ is isomorphic to $\shDrap(r,\shE)$.
\end{remark}


\chapter{A review of K-stability}\label{chap:stability}
This chapter reviews the preliminaries for the study of K-stability. In Section~\ref{sec:the_definition_of_k_stability} we define K-stability following Donaldson \cite{Donaldson2002} with a refinement due to Li-Xu, Stoppa and Székelyhidi \cite{li2011special,stoppa2011note,szekelyhidi2011filtrations}. In Section~\ref{sec:testprop} we give a self-contained introduction to test configurations with the aim of providing background for Chapter \ref{chap:filtrations}. Eisenbud's book \cite{eisenbud1995commutative} was a valuable reference for Section \ref{sec:testprop}.

\section{K-stability} 
K-stability is given in terms of the following abstraction of the Hilbert-Mumford criterion defined in Section~\ref{sec:review_of_stability}.
\label{sec:the_definition_of_k_stability}
\begin{definition}\label{def:TCdef}\cite[Definition 2.1.1]{Donaldson2002}
    Let $X$ be a smooth projective variety with an ample polarisation $L$. A \emph{test configuration} for the polarised variety $(X,L)$ is given by the following data:
    \begin{itemize}
        \item a flat morphism $π:\X\rightarrow\AA^1$ of schemes together with an isomorphism $π^{-1}\{1\}\cong X$,
        \item an $f$-ample line bundle $\scrL$  on $\X$ such that the isomorphism given above lifts to an isomorphism between $\restr{\scrL}{\X_1}\cong L^r$ for some positive integer $r$, where $\X_1$ denotes the fibre $π^{-1}\{1\}$, and
        \item an $\scrL$-linearised action $\rho:\GG_m\times\X\rightarrow\X$ on $\X$ that covers the usual action on $\AA^1$.
    \end{itemize}
    The integer $r$ is called the \emph{exponent} of the test configuration. The fibre $f^{-1}\{0\}$ is called the central fibre.
\end{definition}
\begin{remark}
    We will often refer to a test configuration simply by the scheme $\X$ if the rest of the triple $(\X,\scrL,\rho)$ is either irrelevant to the discussion or clear from the context.
\end{remark}
\begin{definition}
    Let $(X,L)$ be a polarised $\GG_m$-variety where the action is denoted by $α$. Then the natural action on the product $X\times\AA^1$ given by $s.(x,y) = (s.x,sy)$, for $(s,x,y)\in \GG_m\times X \times\AA^1$, is called a \emph{product test configuration} and denoted by $\X_α$.

    We also say that a test configuration $\X$ is \emph{almost trivial} if the normalisation of $\X$ is $\GG_m$-equivariantly isomorphic to a product test configuration induced from a trivial action.
\end{definition}
    Let $(\X,\scrL,ρ)$ be a test configuration. Then the pair $(\X_0,\scrL_0)$ is a $\GG_m$-scheme, which induces a $\GG_m$-representation on the vector space $H^0(\X_0,\scrL_0^k)$. We define the \emph{total weight} to be the trace of the infinitesimal generator $A_k$ of the $\GG_m$-representation on $H^0(X_0,\scrL_0^k)$. Alternatively, the total weight can be defined to be the weight of the $\GG_m$-action on the vector space $\det H^0(\X_0,\scrL_0)$. In order to define the norm of a test configuration we also define the trace squared function as the trace of the square of the infinitesimal generator $A_k$.
\begin{lemma}\label{lem:polys}\cite{Donaldson2008}
    There exist numbers $a_0,a_1,b_0,b_1$ and $c_0$ such that for $k$ sufficiently large we have
    \begin{align}\label{eq:FFuns}
                h(k)&\defeq  χ(Z,Λ^k) = a_0k^n+a_1 k^{n-1}+O(k^{n-2}),\\
                w(k)&\defeq \tr(A_k) = b_0k^{n+1}+b_1k^n +O(k^{n-1}),\\
                \intertext{and}
                d(k)&\defeq  \tr(A_k^2) = c_0k^{n+2}+O(k^{n}).
    \end{align}
    We call the three functions $h(k),w(k)$ and $d(k)$ defined in Equation~\eqref{eq:FFuns} the Hilbert function, weight function and the trace squared function, respectively, following \cite{codogni2015non}.
\end{lemma}

Following Donaldson, we define \emph{Donaldson-Futaki invariant} of $(\X,\scrL,ρ)$ by
\begin{equation}\label{eq:futdef}
    \DF(\X)=\frac{b_0a_1-a_0b_1}{a_0^2}.
\end{equation}

Define the \emph{norm} $\|\X\|$ of a test configuration $\X$ for $(Z,Λ)$ with exponent $r$ by
\begin{equation}\label{eq:norm}
    \|\X\| = r^{-n-2} \left(c_0 - \frac{b_0^2}{a_0} \right).
\end{equation}

\begin{definition}\label{def:kstab}
    Let $\test(X,L)$ denote the set of test configurations of $(X,L)$ which are not almost trivial. We say that $(X,L)$ is
    \begin{itemize}
        \item \emph{K-stable} if $\DF(\X) > 0$ for all $\X\in\test(X,L)$,
        \item \emph{K-polystable} if $\DF(\X) \geq 0$ for all $\X\in\test(X,L)$ and $\DF(\X)=0$ implies that $\X$ is a product test configuration,
        \item \emph{K-semistable} if $\DF(\X) \geq 0$ for all $\X\in\test(X,L)$,
        \item \emph{properly K-semistable} if $(X,L)$ is K-semistable but not K-polystable, and
        \item \emph{K-unstable} $(X,L)$ is not K-semistable.
    \end{itemize}
    If a test configuration $\X$ contradicts any of the first three properties, we say that $\X$ is \emph{destabilising}.
\end{definition}

\begin{remark}[Complements]
    Examples of all of the above notions are known in the strict sense. Any cscK projective manifold which admits infinitesimal automorphisms is at most strictly K-polystable. Keller gave examples of properly K-semistable ruled manifolds \cite{KellerRoss,keller2014projectivisation}. Slope unstable vector bundles on curves have K-unstable projectivisations (cf. Chapter \ref{chap:flags}). Thus, examples of all stability phenomena can already be found in the case of projective bundles.
\end{remark}

\begin{remark}[Invariance of K-stability under scaling]\label{rem:scale}
    K-stability is well-defined in the \emph{cone of polarisations}
    \begin{equation}\label{eq:Vdef}
        \VV(X) = \Amp(X) / \QQ_{>0},
    \end{equation}
    where $\Amp(X)$ is the cone of ample line bundles with rational coefficients. Replacing a Kähler form $ω$ by a multiple $kω$ scales the cohomology class by the same multiple $k$ and preserves constant scalar curvature metrics. Therefore being cscK is well defined in the projectivised Kähler cone as well.

    We may also \emph{rescale} the action by replacing $\X$ by a pullback under a covering map $t\mapsto t^r$ of $\AA^1=\Spec\kk[t]$. This has the effect of changing the weight function by a multiple of the Hilbert polynomial, which does not affect the Futaki invariant.
\end{remark}

\begin{remark}[K-stability and the Kähler cone]
    A natural way to approach the YTD correspondence is to compare the loci of K-polystable and cscK points in $\VV(B)$. If we assume that $\Aut(X)$ is discrete it follows from the work of LeBrun and Simanca \cite{lebrun1994extremal} that the cscK locus is open in the Euclidean topology. Not much is known about the K-stable locus in general.

     We return to the question of variation of the polarisation in Section~\ref{sec:operations_on_relative_test_configurations}.
\end{remark}
We would like to thank Dervan for pointing out the following example \cite{dervan2014alpha}.
\begin{example}[Explicit K-stable and K-unstable regions on blowups.]
    Let $X$ be a blowup of $\PP^2$ at 8 points with the polarisation $L_a = 3H  - E_1- a\sum_{i=2}^8E_i$, where $H$ is the hyperplane divisor and $E_1,\ldots, E_8$ are the exceptional divisors and $a\in \RR_{>0}$. Dervan showed, building on the work of Odaka-Sano \cite{odaka2012alpha}, that $(X,L_a)$ is K-stable for
    \begin{equation}
        \frac{1}{9}(10-\sqrt{10})<a < \frac{1}{9}(\sqrt{10}-2).
    \end{equation}
    Furthermore, by results of Ross and Thomas \cite[Example 5.30]{RossThomas}, there exists an $a_0>0$ such that $(X,L_a)$ is K-unstable for $0 < a < a_0$.
\end{example}

\begin{example}[K-stable and K-unstable polarisations on a ruled threefold.]
    Keller gave an example of a ruled threefold where there exist both K-stable and K-unstable polarisations \cite[Theorem 6.1.1]{kellermemoire}. The K-stable examples are constructed using results of Hong \cite{Hong2002}, Arezzo-Pacard \cite{arezzo2009blowing} and Stoppa \cite{Stoppa2009}, while the unstable examples are obtained by an explicit calculation of Futaki invariants somewhat similar to that done in Chapter~\ref{chap:flags}.
\end{example}

\begin{remark}
    It is also natural to study \emph{real polarisations} which may not define a line bundle, parametrised by
    \begin{equation}
        \VV(B)_\RR = \Amp(B)\ot\RR / \RR_{>0}.
    \end{equation}
    While Definition \ref{def:kstab} gave does not make sense for irrational polarisations, for example the theory of slope stability due to Ross and Thomas does \cite{RossThomas}. Chapter \ref{chap:filtrations} gives a method for parametrising test configurations along line segments of $\VV(B)$ where it may be possible to make sense of the irrational points.
\end{remark}

\section{An introduction to test configurations}\label{sec:testprop}
A test configuration can be embedded into a projective space by Kodaira maps of powers of the polarisation. Let $(\X,\scrL)$ be a test configuration for $(X,L)$. By Remark \ref{rem:scale} we may assume that $\scrL$ is very ample and that the exponent of $(\X,\scrL)$ is 1. Then we have an embedding $ι$ such that the diagram
\[
\begin{tikzcd}
\X \arrow[hook]{r}{ι} \arrow{dr}{π} & \PP\left(π_*\scrL\right) \arrow{d}\\
& \AA^1
\end{tikzcd}
\]
commutes. It follows by \cite[Lemma 2]{Donaldson2008} that there is an equivariant embedding
\begin{equation}\label{eq:tcemb}
    \X \hookrightarrow \PP^n\times\AA^1,
\end{equation}
where the usual $\GG_m$-action on $\AA^1$ is lifted to an action on the pair $(\PP^n,\shO(1))$.

\begin{remark}
    A tacit identification $(X,L)\cong (\X_1,\scrL_1)$ is always made when choosing a test configuration.
\end{remark}

In the following example we will give a description of the degeneration beginning with the projective embedding.
\begin{example}[Test configurations embedded in projective space (cf. Example \ref{ex:Gm})]\label{ex:1pms}
    Consider the projective scheme $(X,L)$ associated to a graded ring $A = R/I$, where
    \begin{equation}
        R=\kk[x_0,\ldots,x_n]
    \end{equation}
    and $I$ is an ideal generated by homogeneous elements of $R$. Let
    \begin{equation}\label{eq:ADet}
        φ:R\ot\kk[t,\tfrac{1}{t}]\rightarrow R\ot \kk[t,\tfrac{1}{t}] \ot \kk[s,\tfrac{1}{s}]
    \end{equation}
    be a homomorphism determined
    \begin{align}
        φ(x_i) &= s^{-w_i}x_i, \text{ for } i=0,\dotsc,n \\
        φ(t) &= s^{-1}t
    \end{align}
    where the integer $w_i$ is called the \emph{weight} of the variable $x_i$ in the (co)action $φ$. We assume that all weights are nonnegative without loss of generality. Similarly define the \emph{weight of a monomial} $x_1^{α_1}\cdots x_n^{α_n}$ to be $α_1w_1 + \dotsb + α_nw_n$, and the \emph{initial term} $\operatorname{in}(f)$ of $f\in R$ to be the sum of terms of highest weight in $t$ in $f$.

    Define a family
    \begin{equation}
        X\times\GG_m\subset \PP^n\times\GG_m
    \end{equation}
    whose ideal $J\subset R[t,\tfrac{1}{t}]$ is defined by making generators of $I$ invariant by multiplying the variables with an appropriate power of $t$. If $f$ is a generator of $I$, we define a generator $g$ of $J$ by
    \begin{equation}
        g(x_0,\dotsc,x_n,t) = t^cf(t^{-w_1}x_0,\dotsc,t^{-w_n}x_n),
    \end{equation}
    where $c$ is the weight of the terms of $\operatorname{in}(f)$. The Zariski closure of the scheme
    \begin{equation}
        \Proj_{\AA^1}R[t,\tfrac{1}{t}]/J \subset \PP^n\times \AA^1
    \end{equation}
    is a flat family over $\AA^1$ whose central fibre is defined by the ideal
    \begin{equation}
        \operatorname{In}(I) \defeq  \big(\operatorname{in}(f) : f\in I\big).
    \end{equation}
    The family of projective varieties $\shProj_{\AA^1}R[t]/J$ determined by the bigraded ring $R[t]/J$ is a test configuration for $(X,L)$.
\end{example}

\begin{remark}[The filtration associated to an embedded test configuration: A continuation of Example \ref{ex:1pms}]
    Here is another way to realise the ring $R[t]/J$. By rescaling the action if necessary we may assume that the largest of the weights $w_i$ is equal to -1. We then define a filtration of $A$ by $\kk$-vector spaces $F_iA$ by setting
    \begin{equation}
        \qquad \quad    F_i A = \operatorname{Span}_{\kk}\Bigg\{ f\in A : \mspace{-120mu}
        \begin{split}
             f \text{ can be written as a sum of monomials}  \\
                \mspace{-50mu}\text{of weight $i$ or less modulo } I
        \end{split}\:\Bigg\}.
    \end{equation}
    For any element $f\in A$ we define the \emph{level} of $f$ to be the number $\lev(f) = \min\{i: f\in F_iA\}$.

    The ring $R[t]/J$ is equivariantly isomorphic to the ring
    \begin{equation}
        \Rees F_\bullet A \defeq \bigoplus_{i=0}^\infty t^i F_i\left(A\right) \subset A[t],
    \end{equation}
    called the \emph{Rees algebra of $F_\bullet A$}, by the isomorphism taking $x_i$ to $t^{w_i}x_i$. Over the central fibre $(t)\in \AA^1$ we have
    \begin{equation}
        \frac{A[t]}{(t) + J} \cong A/\operatorname{In} I,
    \end{equation}
    and a corresponding isomorphism for the Rees algebra
    \begin{equation}\label{eq:ringCF}
         \frac{\Rees F_\bullet A}{(t)} \cong \bigoplus_{i=0}^\infty \frac{F_{i+1}A}{F_i A},
    \end{equation}
    where the latter ring is called the \emph{graded algebra of $F_\bullet A$}.
\end{remark}

\begin{remark}[A generalisation of K-stability]\label{rem:ringfilt}
    The filtration
    \begin{equation}\label{eq:AFilt}
        F_\bullet A : 0\subset \kk = F_0 A \subset F_1 A \subset \dotsb \subset A,
    \end{equation}
    defined in Example \ref{ex:1pms}, is due to Witt-Nyström and Székelyhidi \cite{witt2012test,szekelyhidi2011filtrations} and it has the following properties.
    \begin{enumerate}[(i)]
        \item It is \emph{multiplicative} meaning that it satisfies $(F_iA) (F_jA) \subset F_{i+j}$,.
        \item It is \emph{homogeneous}, that is, homogeneous parts of any element of $F_i A$ are all in $F_iA$.
        \item Every element in $A$ has finite level.
        \item The Rees algebra $\Rees F_\bullet A$ is finitely generated.
    \end{enumerate}
    The test configuration $\X$ from Equation \eqref{eq:tcemb} can be recovered from the filtration \ref{eq:AFilt} uniquely up to rescaling the action.

    A filtration satisfying properties (i)-(iii) is called \emph{admissible}. These properties were taken as an axiom by Székelyhidi in his formulation of \emph{\kbar-stability}, which enlarges the set of test configurations $\test(X,L)$ to include filtrations whose Rees algebra is not finitely generated. Without the assumption (iv) it is still possible to consider a corresponding \emph{sequence $\left(\X_j\right)_{j\in\NN}$ of test configurations}. The test configuration $\X_j$ is determined by an approximation $S_j$ of the Rees algebra $A$, where $S_j$ is the algebra generated by the submodule
    \begin{equation}
        \bigoplus_{k=0}^j F_k At^k\subset \Rees F_\bullet A.
    \end{equation}
    It is easy to show that for $i$ sufficiently large $\Proj_{\AA^1} S_j$ is a test configuration for $(X,L)$. Székelyhidi defined the Futaki invariant of this sequence to be
    \begin{equation}
        \liminf_{i\to\infty} \DF(\X_i)
    \end{equation}
    and proved, together with Boucksom and Stoppa \cite{stoppa2011relative}, the \kbar-stability of a cscK polarised variety $(X,L)$, assuming it has no infinitesimal automorphisms.

    While the limit of the sequence $\X_i$ is not an algebraic object, it has an analytic interpretation in the space of Kähler potentials \cite{ross2014analytic}. Therefore the set of test configurations has a limited analytic compactification with respect to these very special sequences.
\end{remark}



\chapter{A formula for the Chern character of a Schur power}\label{chap:chern}
This chapter is entirely devoted to a technical result used in the computation of the weight polynomial of a flag bundle. We let $r$ and $λ$ be such that $λ\in\shP(r)$ throughout. We also fix a smooth proper scheme $B$ of dimension $b$ and a vector bundle $E$ of rank ${r_E}$. Let $p$ be the projection $p\colon\Flag_r(E)\rightarrow B$.

 Of independent interest would be finding a more general and more elegant formulation for Theorem \ref{thm:maina} (Theorem \ref{summaryC}), which gives a formula for the second order asymptotics of the polynomial $\ch E^{kλ}$ under certain hypotheses. Laurent Manivel has previously calculated the highest order term in \cite[Section 3]{manivel1994theoreme}. Background on Chern classes can be found in the seminal work of Grothendieck \cite{grothendieck1958theorie}.

\section{A formula for the Chern character}\label{sec:chern}
\label{sec:formula}

    If $P$ is a symmetric polynomial and $E$ is a vector bundle with Chern roots $x_1,\dotsc, x_{r_E}$, we write $P(E)=P(x_1,\dotsc,x_{r_E})$. On the other hand it also makes sense to consider the polynomial $P$ on the algebra generated by line bundles on a variety and operations defined by direct sums and tensor products. In this case we write $P(L_1,\dotsc,L_{r_E})$ for the resulting vector bundle, not to be confused with $P(E)$, which is a cohomology class.

Let
\begin{equation}
    c_r(x_1,\dotsc,x_{r_E}) = \sum_{1\leq i_1 < i_2 < \dotsb < i_r \leq {r_E}} x_{i_1} \dotsm x_{i_r}
\end{equation}
denote the $r$th elementary symmetric polynomial in $x_1,\dotsc,x_{r_E}$. Similarly we have the complete symmetric polynomial
\begin{equation}
    h_r(x_1,\dotsc,x_{r_E})=\sum_{1\leq i_1 \leq i_2 \leq \dotsb \leq i_r \leq {r_E}} x_{i_1} \dotsm x_{i_r}.
\end{equation}
Recall that Schur polynomials are a basis of the algebra of symmetric funtion, which appear naturally when computing the cohomology of Schur powers of vector bundles. We define Schur polynomials by using the Giambelli formula \cite[Appendix A]{FultonRep} as
\begin{equation}\label{eq:schurdef}
    s_λ = \det \left(h_{λ_i-i+j}\right)_{1\leq i,j\leq l}
\end{equation}
associated to a partition $λ$. In particular, $s_{(k)} = h_k$ and $s_{1^k} = c_k$.

\begin{definition}\label{def:canonical}
    Define the \emph{canonical partition} $σ=σ_{{r_E},r}$ depending on the parameter $r$ by
    \begin{equation}\label{eq:canonicalpart}
        σ_i = {r_E} + l(λ) - r^+(i) - r^-(i)
    \end{equation}
    where $r^+(i)$ is the smallest integer in $r$ satisfying $r^+(i)\geq i$ and $r^-(i)$ the largest integer in $r$ satisfying $r^-(i)<i$.
\end{definition}
\begin{example}[The canonical bundle of a Grassmannian]
    Consider the Grassmannian case $r = (p)$, where $1\leq p < r_E$. Now the canonical partition $σ$ is the constant partition $(r_E^p)$, which corresponds to the $r_E$th multiple of the hyperplane bundle in the case $p = 1$. Note that the relative canonical bundle of $\PP E$ over $B$ is the dual of the corresponding line bundle $\shL_{σ}$.
\end{example}

\begin{theorem}\label{thm:maina}
    Let $E$ be a vector bundle of rank $E$ and $λ$ a partition whose jumps are given by $r$. Assume that $λ$ satisfies at least one of the following conditions
    \begin{itemize}
        \item $l(λ)\leq 4$
        \item $λ=tσ_{{r_E},r}$ for some $t\in\QQ$ and ${r_E}>r_c$.
    \end{itemize}
    Then there exist polynomials $B_i(E,λ) \in \QQ[λ_1,\dotsc λ_l,c_1(E),\dotsc,c_{r_E}(E)]$ such that
    \begin{equation}\label{eq:conj}
        \ch E^\lambda = \rank E^\lambda \left( 1 + B_1(E,λ) + B_2(E,λ) + \dotsb + B_n(E,λ) \right)
    \end{equation}
    where $B_i(E,λ)$ is homogeneous of degree $i$ as an element of the Chow ring of $X$ and of degree $i$ in the $λ_i$. The polynomials $B_1(E,λ)$ and $B_2(E,λ)$ are given by
    \begin{equation}\label{eqmaina1}
        B_1(E,λ)= \frac{c_1(λ)c_1(E)}{r_E}
    \end{equation}
    and
    \begin{equation}\label{eqmaina2}
        \begin{split}
            B_2(E,λ)&\equiv_1 \frac{h_2(λ)h_2(E)}{{r_E}({r_E}+1)} + \frac{c_2(λ)c_2(E)}{{r_E}({r_E}-1)} \\
            &+ \frac{{r_E}c_1(λ) - \sum_i(2i-1)λ_i}{2} \left(\frac{h_2(E)}{{r_E}({r_E}+1)}  - \frac{c_2(E)}{{r_E}({r_E}-1)} \right) + O(1).
        \end{split}
    \end{equation}
    where $O(1)$ denotes a term independent of $λ$. By the equivalence $\equiv_1$ we mean the following: If $U$ and $V$ are $k$-cycles in $B$, then $U\equiv_1 V$ if $c_1(A)^{n-k}.(U-V)$ is equal to 0 for all line bundles $A\in \Pic B$.
\end{theorem}
It is straightforward to check in cases which yield to computer analysis that it is not necessary to assume $\diamond$ for the identity in Equation~\eqref{eq:STA} to hold, but we were unable to find a proof in the general case. Under the assumption $\diamond$, we prove the statement using the following determinantal identity, which the author learned from a paper \cite{Bruckmann2008} pointed out by Will Donovan.
\begin{lemma}[Determinantal identity]\label{lem:determinantal}
    Let $E$ be a vector bundle of rank ${r_E}$ and $\lambda$ a partition of length $l$. The Chern character of a Schur power of $E$ is
    \begin{align}\label{eqschurchar}
        \mathrm{ch} E^\lambda  = \det \left(\mathrm{ch} (S^{\lambda_i+j-i}E)\right)_{i,j}
    \end{align}
\end{lemma}
\begin{proof}
    By the splitting principle \cite[Remark 3.2.3]{Fulton1998} we may assume that $E=L_1\oplus\dotsb \oplus L_{r_E}$. Let $p$ be a polynomial function on the set of factors $L_1,\dotsc,L_{r_E}$ with integral coefficients $a_I$ for $I=(i_1,\dotsc, i_{r_E})$. We denote
\begin{equation}
    p(L_1,\ldots,L_{r_E})=\bigoplus_I \left(L_1^{i_1}\otimes\dotsm\otimes L_{r_E}^{i_{r_E}}\right)^{\oplus a_I},
\end{equation}
Schur powers of decomposable vector bundles can be expressed in as
\begin{equation}
    E^λ = s_λ(L_1,\ldots,L_{r_E}),
\end{equation}
which we expand as a determinant using Equation~\eqref{eq:schurdef}
\begin{equation}
    s_\lambda(L_1,\ldots,L_{r_E}) = \det\left(h_{\lambda_i+j-i}(L_1,\ldots,L_{r_E})\right)_{i,j}.
\end{equation}
Taking Chern characters on both sides completes the proof of the Lemma.
\end{proof}

\begin{lemma}\label{lem:sym_chern}
    Let $E$ be a vector bundle of rank ${r_E}$. The Chern character of the bundle $S^k E$ is
    \begin{equation}
        \binom{k+{r_E}-1}{{r_E}}\left(1 + \frac{c_1(E)}{r_E}k + A_1(E)k^2 + A_2(E)k + Z\right),
    \end{equation}
    where $A_1(E),A_2(E)\in\QQ[x_1\ldots x_{r_E}]$ are given by
    \begin{equation}
        A_1(E)=\frac{h_2(E)}{{r_E}({r_E}+1)}\label{eqa},
    \end{equation}
    \begin{equation}
        A_2(E)=\frac{{r_E}-1}{2}\left(\frac{h_2(E)}{{r_E}({r_E}+1)}-\frac{c_2(E)}{{r_E}({r_E}-1)}\right)\label{eqaa}
    \end{equation}
    and $Z$ is a sum of terms of Chow degree 3 and higher.
\end{lemma}
\begin{proof}\label{pf:sym_chern}
    Recall the definition of the monomial symmetric function $m_μ$ of partition $μ$ of length at most $n$. Given variables $y=(y_1,\ldots,y_n)$ we set
    \begin{equation}
        m_μ(y) = \sum_{σ\in\mathfrak{S}_n} y^{μ_1}_{σ(1)}\dotsm y^{μ_n}_{σ(n)}.
    \end{equation}
    We have
    \begin{equation}
        \begin{split}
            \mathrm{ch}(S^k E) &= \ch h_k(E)\\&=\ch\sum_\mu m_μ(E)\\
             &= \sum_\mu (1+\mu_1x_1+μ_1^2 x_1^2/2+\dotsb)\cdot \dotsm \cdot (1+μ_{r_E}x_{r_E}+μ_{r_E}^2x_{r_E}^2/2+\dotsb)
        \end{split}
    \end{equation}
    where the sum is over all ${r_E}$-tuples that sum to $k$. The rest of the computation is an elementary summation. The Chow-degree one part of $\mathrm{ch}(S^k E)$ is
    \begin{equation}
        \ch(S^k E)_1=  \rank \left(S^k E\right) \frac{c_1(E)}{{r_E}},
    \end{equation}
    where
    \begin{equation}
        \rank \left(S^k E\right) = \binom{k+{r_E}-1}{{r_E}-1}.
    \end{equation}
    The degree two term can be written as
    \begin{equation}
        \sum_{i=1}^k\sum_{j=1}^{k-i} ij {{r_E}-3+k-i-j \choose {r_E}-3} \sum_{l <  m}^{r_E} x_l x_m + \sum_{i=1}^k i^2 {{r_E}-2+k-i \choose {r_E}-2} \sum_{l=1}^{r_E} x_l^2/2,
    \end{equation}
    which using the combinatorial identities proved in the appendix simplifies to
    \begin{equation}
        \frac{(k+{r_E}-1)!}{(k-2)!({r_E}+1)!} \sum_{m < l}^{r_E} x_m x_l + \frac{({r_E}+2k-1)(k+{r_E}-1)!}{(k-1)!({r_E}+1)!} \sum_{m=1}^{r_E} x_m^2/2.
    \end{equation}
    Picking out the rank ${r_{S^kE}}$ of $S^k E$ as a common factor yields
    \begin{equation}
        \mathrm{ch}_2(S^k E)={r_{S^kE}} \left( \frac{k(k-1)}{{r_E}({r_E}+1)}\sum_{m < l}^{r_E} x_m x_l + \frac{2k^2 +k({r_E}-1)}{r_E(r_E+1)} \sum_m x_m^2/2  \right)
    \end{equation}
    Recall that the Chern classes of $E$, when written in terms of the $x_i$, are
    \begin{equation}
        c_1(E)^2= h_2(E) + c_2(E)= \sum_{m=1}^{r_E} x_m^2 + 2 \sum_{m < l}^{r_E}x_mx_l
    \end{equation}
    and
    \begin{equation}
        c_2(E) = \sum_{m < l}^{r_E} x_m x_l.
    \end{equation}
    Thus we have
    \begin{equation}\label{eqsym}
        \ch\left( S^k E \right) = \rank (S^k E)\left( 1 + \frac{c_1(E)}{{r_E}} k + A_1(E) k^2 + A_2(E) k + Z\right),
    \end{equation}
    where
    \begin{equation}
            A_1(E)=\frac{h_2(E)}{{r_E}({r_E}+1)},\\
    \end{equation}
    \begin{equation}
            A_2(E)=\frac{({r_E}-1)c_1(E)^2}{2{r_E}({r_E}+1)} - \frac{c_2(E)}{{r_E}+1} =\frac{{r_E}-1}{2}\left(\frac{h_2(E)}{{r_E}({r_E}+1)}-\frac{c_2(E)}{{r_E}({r_E}-1)}\right)
    \end{equation}
    and $Z$ is a sum of terms of Chow degree 3 and higher
\end{proof}

\begin{remark}
    The length of a partition $λ$ whose jumps are given by $r$ is the largest integer $r_c$ in $r$.
\end{remark}

\begin{proposition}
    Theorem \ref{thm:maina} holds for partitions up to length 4.
\end{proposition}
\begin{proof}
    This is an easy calculation for a computer using Lemma \ref{lem:sym_chern} and Lemma \ref{lem:determinantal} \cite[Calculation of Chern classes for Schur powers]{codepage}.
\end{proof}
\begin{remark}[{\cite[Section 3]{manivel1994theoreme}}]
    Alternatively one may expand the Chern character of $S^k E$ as
        \begin{equation}
            \sum_{p,q} x^p \prod_{i=1}^r \frac{a_{p_i,q_i}}{p_i!}\binom{k+{r_E}-1+|q|}{{r_E}-1+|p|}
        \end{equation}
        where $p,q$ range over $r$-tuples of nonnegative integers and $a_{i,j}$ is the $j$th coefficient of the $i$th Euler polynomial \cite[Proposition 2.2]{manivel1994theoreme}. This way the existence of claimed decomposition
    \begin{equation}
        \ch (S^k E) = \rank (S^k E) A(k)
    \end{equation}
    is clear for higher degree terms as well. The determinantal identity implies that we have
    \begin{equation}
    \ch(E^λ) =  \sum_{p_i,q_j\in\NN^{r_E}}\frac{x^{p_1+\dotsb+p_l}}{p_1!\dotsm p_l!}a_{p_1,q_1}\dotsm a_{p_l,q_l}\det\left(\binom{{r_E}+λ_i+|q_i|-i+j-1}{{r_E}+|p_i|-1}\right)_{1\leq i,j \leq l}
    \end{equation}
\end{remark}
Let $p:\PP E\rightarrow X$ denote the projection. It is well known that we have the pushforward formula
\begin{equation}
    \int_{\PP E}p_* c_1\left(\shO_{\PP E}(1)\right)^{n+r-1} =  \int_X h_n(E).
\end{equation}
This formula generalises to the following theorem by Laurent Manivel.
\begin{theorem}[{\cite[Proposition 3.1]{manivel1994theoreme}}]\label{thm:manivel}
Let $λ$ be a partition whose jumps are given by $r$ and $m\ZZ_{\geq 0}$. Then we have
\begin{equation}
    p_{*} \frac{c_1(\shL_λ)^{N + m}}{(N + m)!} \equiv_1 C_{λ,{r_E}} \sum_{|μ|=m,l(μ)\leq l(λ)}\frac{s_μ(λ)s_μ(E)}{\prod_{k=1}^{l(λ)}(r_E+μ_k-k)!},
\end{equation}
where $C_{λ,{r_E}} = \prod_{i=1}^{l(λ)}(s^+(i)-i)!\prod_{λ_i>λ_j}(λ_i-λ_j)$. For $m=n$ we have equality of cycles, while for $m<n$, the relation $\equiv_1$ is the one defined in Theorem \ref{thm:maina}
\end{theorem}

\begin{remark}
    The result stated in \cite{manivel1994theoreme} actually claims equality at the level of cycle classes. As we were unable to reproduce the details which were left for the reader in the paper, we state a slightly weaker result, but this is enough for our purposes.
\end{remark}

\begin{remark}
    Although the highest order term of each $B_i(E,λ)$ is a symmetric function with respect to the $λ$, this is not the case for the lower order terms, or indeed for the entire Chern character.
\end{remark}

\begin{remark}\label{rem:leadingcoeff}
    In particular, Theorem \ref{thm:manivel} computes the leading coefficient
    \begin{equation}
        D_{λ,{r_E}}:=\frac{C_{λ,{r_E}}}{\prod_{i=1}^{l(λ)} ({r_E}-i)!}
    \end{equation}
    of the Hilbert polynomial of a fibre $π^{-1}(x)$ for any $x\in B$.
\end{remark}

\begin{remark}
    We can write the line bundle $\shL_σ$ in terms of the tautological subbundles as
    \begin{equation}
        \bigotimes_{i=1}^c\left(\det \shR_i^*\right)^{r_{i+1}-r_{i-1}}.
    \end{equation}
\end{remark}

\begin{lemma}[Canonical bundle of the flag variety]\label{lem:fcanonical}
    The canonical class of $\Flag_r(E)$ is
    \begin{equation}\label{eq:fcanonical}
        c_1(\shL_{-σ}\otimes p^*\left(K_B\otimes\det E^{l(σ)}\right)),
    \end{equation}
        where $σ$ is the canonical partition defined Definition \ref{def:canonical} and $\shL_{-σ}$ denotes the dual of $\shL_σ$.
\end{lemma}

\begin{proof}
    Consider the exact sequence
    \begin{equation}\label{eq:relK}
        0\lra \shV_{\Flag_r(E)}\lra \shT_{\Flag_r(E)}\lra \shH_B\lra 0
    \end{equation}
    where $\shV_{\Flag_r(E)}$ is the relative tangent bundle of the fibration $\Flag_r(E)\ra B$, $\shT_{\Flag_r(E)}$ is the tangent bundle and $\shH_B$ is isomorphic to the pullback of the tangent bundle of the base $B$. The relative tangent bundle $\shV_{\Flag_r(E)}$ has a filtration
    \begin{equation}
        0\subset F_1\subset\dotsb\subset F_N\subset \shV_{\Flag_r(E)}
    \end{equation}
    such that
    \begin{equation}\label{eq:lb1}
        \bigoplus_{i=1}^N F_{i+1}/F_{i} = \bigoplus_{1\leq i<j\leq c} \shQ_i\otimes\shQ_j^*
    \end{equation}
    This can be seen by successive fibrations by bundles of $r'$-flags, where $r'$ is a subset of $r$ \cite{lam1975formula}. We have
    \begin{equation}
        \det(\shV_{\Flag_r(E)})^* \cong \det\left(\bigoplus_{1\leq i<j\leq c} \shQ_i\otimes\shQ_j^*\right)^*
    \end{equation}
    Denote $\det \shR_i^* = L_i$ and define
    \begin{equation}
            A(k)\defeq\det\left(\bigoplus_{1\leq i<j\leq k+1}  \shQ_i\otimes\shQ_j^* \right) = \det\left(\bigoplus_{1\leq i<j\leq k} \shR^*_i/\shR^*_{i-1}\otimes\shR_j/\shR_{j-1}\right).
    \end{equation}
    We expand the determinant of the vector bundle of Equation~\eqref{eq:lb1} as
    \begin{equation}
            A(c)=\det\left(\bigoplus_{1\leq i<j\leq c+1}  \shQ_i\otimes\shQ_j^* \right) = \det\left(\bigoplus_{1\leq i<j\leq c} \shR^*_i/\shR^*_{i-1}\otimes\shR_j/\shR_{j-1}\right).
    \end{equation}
    This is convenient to write in additive notation as
    \begin{equation}\label{eq:telescope}
        \sum_{1\leq i < j \leq c+1} \left(- (r_i-r_{i-1}) \left(L_j-L_{j-1}\right) + (r_j-r_{j-1}) \left(L_i-L_{i-1}\right)\right).
    \end{equation}
    We have
    \begin{equation}
        A(k)-A(k-1) = r_k L_{k-1} - r_{k-1} L_k.
    \end{equation}
    for any $1\leq k\leq c$. Therefore, we can see that the sum in Equation \ref{eq:telescope} telescopes and we find
    \begin{equation}
        A(c)  =  \sum_{i=1}^{c} (r_{i+1}-r_{i-1})L_i - r_c L_{c+1}.
    \end{equation}
    Finally, the identity
    \begin{equation}
        K_{\Flag_r(E)}= -A(c) + p^*K_B,
    \end{equation}
    follows from Equation \ref{eq:relK}. This completes the proof of the Lemma.
\end{proof}

\begin{lemma}\label{lem:combinatorics}
    Let $r$ be an increasing sequence of $c$ positive integers. Then $σ=σ_{{r_E},r}$ is a partition of length $r_c$ with $r_c<{r_E}$. We have
    \begin{equation}
        |σ| = {r_E}r_c,
    \end{equation}
    \begin{equation}
        \sum_{i=1}^{r_c}(2i-1)σ_i = r_c^2{r_E} - \sum_{i=1}^{c-1} r_ir_{i+1}(r_{i+1}-r_i),
    \end{equation}
    and
    \begin{equation}
        \begin{split}
            h_2(σ) &= \frac{1}{2}\left(r_c{r_E}^2(r_c+1) + \sum_{i=1}^{c-1}r_ir_{i+1}(r_{i+1}-r_i)\right),
        \end{split}
    \end{equation}
\end{lemma}

\begin{proof}
    The proof is a direct calculation. We prove the third identity, which is marginally more difficult than the first two.
    First notice that given an integer $n$ and an $l$-tuple $λ$, we have
    \begin{equation}
        h_2(n+λ) = \frac{l(l+1)}{2}n^2 + (l+1)n|λ| + h_2(λ).
    \end{equation}
    where $n$ is considered to be the constant $l$-tuple $(n,\dotsc,n)$. Applying this in the case $n=r_E+r_c$ and $λ=-(r^++r^-)$ it suffices to show that
    \begin{equation}
        h_2(r^++r^-) = \frac{1}{2}\left(r_c^3(r_c+1) + \sum_{i=1}^{c-1}r_ir_{i+1}(r_{i+1}-r_i)\right).
    \end{equation}
    This is proved by induction. Let $s$ be the tuple $(r_1,\dotsc,r_{c-1})$. We then have
    \begin{equation}
        \begin{split}
            h_2(r^++r^-)-h_2(s^++s^-) &= (r_c+r_{c-1})^2(r_c-r_{c-1})(r_c-r_{c-1} + 1)/2 \\
            &+\sum_{i=1}^{c-1}(r_i-r_{i-1})(r_i+r_{i-1})(r_c-r_{c-1})(r_c+r_{c-1})\\
            &=r_c^3(r_c+1)/2+r_{c-1}^3(r_{c-1}+1)/2 + r_c r_{c-1}(r_c-r_{c-1})/2
        \end{split}
    \end{equation}
    from which the claim follows.
\end{proof}

Let $N_{{r_E},r}$ denote the relative dimension of a bundle of $r$-flags, given by
\begin{equation}\label{eq:flagdim}
    N_{{r_E},r} = \sum_{i=1}^cr_i(r_{i+1}-r_i),
\end{equation}
with the convention $r_{c+1} = {r_E}$.

\begin{proof}[Proof of Theorem \ref{thm:maina}]
Retain the notation in the statement of the Theorem and denote $N=N_{{r_E},r}$. Assume that $λ=tσ$ for some $t\in\QQ$. The leading order term of $B_2(E,kλ)$ in $k$ is
\begin{equation}
    p_{r*} \frac{c_1(\shL_λ)^{N + 2}}{(N + 2)!} \equiv_1 D_{λ,{r_E}} \left(\frac{h_2(λ)h_2(E)}{{r_E}({r_E}+1)} + \frac{c_2(λ)c_2(E)}{r_E({r_E}-1)}\right),
\end{equation}
by Theorem \ref{thm:manivel}. The term $B_1(E,kλ)$ can be computed easily using the splitting principle. In general, we have
\begin{equation}
    c_1(E^λ) = \rank E^λ c_1(λ)c_1(E)/{r_E}.
\end{equation}
It suffices to verify that the $k$-linear term of $B_2(E,kλ)$ satisfies the claimed identity.

For any line bundle $L$ on the base $B$, the Hirzebruch-Riemann-Roch formula applied to the vector bundle $\left(E\ot L\right)^{kλ}$ yields
\begin{equation}\label{eq:hrr1}
    \begin{split}
        χ(B,E^{ktσ})&=\int_B \ch L^{k|λ|} \ch E^{kλ}\mathrm{Td}_B\\
        &=\int_B \sum_{i=0}^b \frac{\left(k|λ|c_1(L)\right)^i}{i!}\ch E^{kλ}\mathrm{Td}_B.
    \end{split}
\end{equation}
Moreover, we have
\begin{equation}
    \frac{c_1(\shL_λ(A))^{N+n}}{(N+n)!} = \sum_{i=1}^n  \frac{c_1(\shL_λ)^{N+i}}{(N+i)!} p^* \frac{c_1(A)^{n-i}}{(n-i)!}
\end{equation}
for all $n\geq 1$ and $A\in \Pic B$.

By the asymptotic Hirzebruch-Riemann-Roch formula on $\Flag_r(E)$ for the line bundle $\shL_λ(L^{|λ|})^{\otimes k}$, we have
\begin{equation}\label{eq:HRRcanonical}
    \begin{split}
        χ(\Flag_r(E),\shL_λ(L^{|λ|})^k) = \int_{\Flag_r(E)} &\left( \frac{c_1(\shL_λ(L^{|λ|}))^{N + b}}{(N + b)!} k^{N+b}\right. \\
        & \left. -\frac{c_1(\shL_λ(L^{|λ|}))^{N + b-1} K_{\Flag_r(E)}}{2(N + b-1)!} k^{N+b-1} \right)  +O(k^{N+b-2})
    \end{split}
\end{equation}
The remaining part of the statement then follows by comparing the $k$-degree $b-1$ coefficients of the $c_1(L)^{b-2}$ term in Equation~\eqref{eq:hrr1} and Equation~\eqref{eq:HRRcanonical}, latter of which is equal to
\begin{equation}
    k^{N+1}\int_X \left( \frac{\left(p_{r*}c_1(\shL_λ)\right)^{N+2}}{2t(N + 1)!} - \frac{\left(p_{r*}c_1(\shL_σ)\right)^{N + 1} \left(c_1((\det E)^{\otimes l(λ)}) + K_B\right)}{2(N + 1)!}\right)\frac{c_1(L^{b-2})}{(b-2)!}.
\end{equation}
by Lemma \ref{lem:fcanonical}.
We write
\begin{equation}
    B_2(E,kλ) = k^2B_{2,2} + kB_{2,1} + O(k^0)
\end{equation}
and expand the Chern character in of $E^{kλ}$ as
\begin{equation}\label{eq:factored}
    \ch E^{kλ} = D_{λ,{r_E}} \left( k^{N}+\frac{N}{2t}k^{N-1} + O(k^0)\right) \left( 1 + B_1(E,kλ) + k^2B_{2,2} + kB_{2,1}\right).
\end{equation}
We can see that
\begin{equation}
    B_{2,1} = \left(\frac{h_2(λ)h_2(E)}{t{r_E}({r_E}+1)} + \frac{c_2(λ)c_2(E)}{t{r_E}({r_E}-1)} - \frac{l(λ) |λ|c_1(E)^2}{2{r_E}}\right),
\end{equation}
which can be written as
\begin{equation}
      \frac{t\left(({r_E}-1)h_2(σ)-({r_E}+1)c_2(σ)\right)}{2{r_E}} \left(\frac{h_2(E)}{{r_E}({r_E}+1)} - \frac{c_2(E)}{{r_E}({r_E}+1)}\right),
\end{equation}
Finally by Lemma \ref{lem:combinatorics} we have
\begin{equation}
    \begin{split}
        \frac{({r_E}-1)h_2(σ)-({r_E}+1)c_2(σ)}{2{r_E}}  &= h_2(σ) - \frac{({r_E}+1)er_c^2}{2}\\
        &=\frac{\sum_i r_i r_{i+1}(r_{i+1}-r_i)}{2}\\
        &=\frac{t\left(e|σ| - \sum_i(2i-1)σ_i\right)}{{r_E}-1}\\
        &=\frac{e|λ| - \sum_i(2i-1)λ_i}{{r_E}-1}
    \end{split}
\end{equation}
This completes the proof.
\end{proof}

\begin{remark}\label{rem:abrelation}
    In general, there is a simple relation between the classes $B_{2,0}(λ,E)$ and $A_2(E)$. Namely we have
    \begin{equation}
        B_{2,0} - \frac{2({r_E}+1)}{{r_E}-1}A_2(λ)A_2(E) = \frac{c_1(λ)^2c_1(E)^2}{2{r_E}^2}.
    \end{equation}
\end{remark}

\begin{remark}
    The same calculation can be used to find the codegree 1 asymptotics of $B_i(E,kλ)$ in any Chow degree, when $λ=kσ$ for some $k\in\QQ$. Keeping to the same notation as in the proof, we have
    \begin{equation}
        \begin{split}
            B_m(E,kλ) & = k^mC_{λ,r}\frac{\sum_{|μ|=m}s_μ(λ)s_μ(E)}{\prod_{i=1}^l (r_E+μ_i-i)!}  \\
            &  + k^{m-1}C_{λ,r} \left(\frac{m\sum_{|μ|=m}s_μ(λ)s_μ(E)}{2t\prod_{i=1}^l (r_E+μ_i-i)!} - \frac{|λ|c_1(E)\sum_{|μ|=m-1}s_π(λ)s_π(E)}{2\prod_{i=1}^l (r_E+p_i-i)!}\right) \\
            &+ O(k^{m-2}),
        \end{split}
    \end{equation}
    for any $m\geq 2$.
\end{remark}


\chapter{K-stability of relative flag varieties}\label{chap:flags}
Fix the following notation. Let $E$ be a vector bundle of rank ${r_E}$ on a polarised smooth complex variety $(B,L)$ of dimension $b$, and $\Flag_r(E)$ the flag bundle of $r$-quotients of $E$ with projection $p$ onto $B$. Also fix an ample line bundle $\shL_λ(A) = \shL_λ\ot p^*A$ on $\Flag_r(E)$, where $λ$ is in $\shP(r)$ and $A$ is an ample line bundle on $B$.

In Section~\ref{sec:test} we construct a test configuration $(\Y_\shF,\scrL_λ(A))$ which we conjecture to be sufficient for detecting the K-instability of the flag bundle $(\Flag_r(E),\shL_λ(A))$ assuming that the base $B$ is stable.

From now on, we assume that $λ$ is in $\shP_{\diamond}(r)$. Section \ref{sec:curve} calculates the Donaldson-Futaki invariant of $\Y_\shF$ if we assume the base to be a curve.
\begin{theorem}\label{thm:DFC}
    Assume that $B$ is a curve, $E$ is ample and $F$ is a subbundle of $E$ whose degree is positive. There exists a test configuration $\Y_F$ for $(\Flag_r(E),\shL_λ(A))$ such that
    \begin{equation}
        \DF(\Y_{F},\shL_λ(A)) = C \left(μ_E - μ_F\right).
    \end{equation}
    for some positive constant $C$ depending on $E,F,g$ and $r$.
\end{theorem}

In Section \ref{sec:anybase} we outline a similar calculation for adiabatic polarisations on a flag bundle over a base of arbitrary dimension.
\begin{theorem}\label{thm:DFB}
    Assume that $\shF$ is a saturated torsion free subsheaf of $E$. Let $L$ be an ample line bundle on $B$ and assume that $A=L^m$. Then there exists an integer $m_0$ and a test configuration $\Y_\shF$ for $(\Flag_r(E),\shL_λ(L^m))$ such that for $m>m_0$ the Donaldson-Futaki invariant of $\Y_\shF$ is given by
    \begin{equation}
        \DF(\Y_\shF, \scrL_λ(L^m)) = C \left(μ_E - μ_\shF\right) \tfrac{1}{m} + O(\tfrac{1}{m^2})
    \end{equation}
    for some positive constant $C$ depending on $E,F,B$ and $r$.
\end{theorem}

These results immediately imply the stability statements of Theorem \ref{summaryA} and Theorem \ref{summaryB} from Section~\ref{sec:summary_of_results}.
\begin{theorem}[The K-instability statements of Theorem A]\label{thm:curve1}
    Assume that $B$ is a curve, $E$ is an ample vector bundle on $B$ and $A$ is ample. If $E$ is slope unstable and $λ$ is in $\shP_\diamond(r)$, then the flag bundle $(\Flag_r(E),\shL_λ(A))$ is K-unstable. If $E$ is not polystable, then the pair $(\Flag_r(E),\shL_λ(A))$ is not K-polystable.
\end{theorem}
\begin{proof}
    Fix a destabilising subsheaf $\shF$ of $E$ with maximal slope. The saturation, which by definition has a torsion free quotient, also destabilises. Torsion free coherent sheaves on a curve are locally free, so we may assume that $F$ is a subbundle. In particular $E/F$ is locally free. The claim then follows from Theorem \ref{thm:DFC}.

    To prove the second assertion, let $F$ be a subbundle of $E$ with maximal slope such that $μ(F) = μ(E)$ and assume that $F$ is not a direct summand. The scheme $\Y_F$ is smooth, so in particular it is normal. It follows that the test configuration is almost trivial only if it the total space $\Flag_r(\shE)$ is isomorphic to $\Flag_r(E)\times\AA^1$ \cite{stoppa2011note}. The two schemes $\Flag_r(E)$ and $\Flag_r(F\oplus E/F)$ are not isomorphic since it is possible to construct an isomorphism of underlying vector bundles from an isomorphism of flag bundles which preserves the polarisation. Therefore the bundle $\Flag_r(E)$ is not K-stable.
\end{proof}

\begin{theorem}[Theorem B]\label{thm:B}
    If $E$ is slope unstable and $λ$ is in $\shP_\diamond(r)$, then there exists an $m_0$ such that the flag variety $\Flag_r(E)$ of $r$-flags of quotients in $E$ with the polarisation $\shL_λ(L^m)$ is K-unstable for $m>m_0$.
\end{theorem}

\begin{proof}
    Follows immediately from Theorem \ref{thm:DFB}.
\end{proof}

An identical argument to \cite[Proposition 5.25]{RossThomas} which will not be repeated here shows the following instability result which is also discussed in Example \ref{ex:RTPB}.
\begin{proposition}\label{prop:flagbase}
    If the base $(B,L)$ is \emph{strictly slope unstable} in the sense of \cite[Definition 3.8]{RossThomas}, then there exists an $m_0>0$ such that $\left(\Flag_r(E),\shL_λ(L^m)\right)$ is K-slope unstable for $m>m_0$.
\end{proposition}

\section{Simple test configurations on flag bundles} 
\label{sec:test}
In this section we define the relative test configuration $(\Y_\shF,\scrL_λ(A))$. First, recall the following standard construction.
\begin{definition}[The extension group of a coherent sheaf]
    Let $\shF$ and $\shQ$ be coherent sheaves on $B$ and let $p_1\colon B\times \AA^1\rightarrow B$ be the first projection. An \emph{extension} of $\shQ$ by $\shF$ is a coherent sheaf $\shE'$ together with maps of $\shO_B$-modules which fit the short exact sequence
    \begin{equation}
        0\rightarrow \shF \rightarrow \shE' \rightarrow \shQ \rightarrow 0.
    \end{equation}
    Extensions are parametrised by the vector space $\shV=\Ext^1(B,\shQ,\shF)$ and there is a universal extension $\shU$ on $B\times \shV$ whose fibres are the corresponding extensions $\shE'$. The sheaf $\shU$ is naturally $\CC^\times$-equivariant for the scaling action on $B\times \shV$ which acts trivially on $B$.
\end{definition}
Consider the reverse point of view where $E$ is a fixed vector bundle fitting an exact sequence
\begin{equation}
    0\rightarrow \shF \rightarrow E \rightarrow \shQ \rightarrow 0.
\end{equation}
\begin{remark}[Turning off an extension]\label{rem:ext}
    Let $E$ be a locally free sheaf on $B$ and $\shF$ a quasicoherent subsheaf of $E$ with quotient $\shQ$. We abuse notation by writing $p_1^*E$ as
    $E[t]$ (we tacitly identify the algebra $\kk[t]$ with the associated sheaf on $\AA^1$), and identify $\extE$ as the subsheaf
    \begin{equation}
        \extE = p_1^*\shF + tp_1^* E \subset p_1^*E = E[t].
    \end{equation}
    The sheaf $\extE$ is naturally isomorphic to the pullback of the universal extension under the inclusion
    \begin{equation}
        B\times\AA^1\rightarrow B\times \Ext^1(B,\shQ,\shF).
    \end{equation}

    There is a natural $\GG_m$-linearisation on $\extE$ of the standard $\GG_m$-action on $B\times \AA^1$.  The fibre over $s\in\AA^1$ of the sheaf $\extE$ is given by
    \begin{equation}
            \frac{\extE}{(t-s)\extE}\cong \begin{cases} E &\mbox{if } s \neq 0 \\
            \shF\oplus \shQ & \mbox{if } s=0. \end{cases}
    \end{equation}
    In particular, the fibre of $\shE$ over $s=0$ is fixed by the $\GG_m$-action, and so are all the fibres of $\shF\oplus \shQ$ over $B\times \{ 0 \}$, so the linearisation is determined by a simple scaling action on the sections. Over the central fibre a section over an open set $U\subset B$ can be written as
    \begin{equation}
        σ=f+te + t\extE(U)\in \frac{\extE}{t\extE}\left(U\right)
    \end{equation}
    Therefore we can write $σ$ uniquely as $f+t\left(e+\shF(U)\right) + t^2E(U)$. The scaling action on $\AA^1$ acts on the section $t$ with weight $-1$.
\end{remark}
We may renormalise the natural $\GG_m$-linearisation on $\extE$ to scale sections of $\shF$ with weight 1 and sections of $\shQ$ with weight 0 over the central fibre. By Lemma \ref{lem:induceG}, we have an induced $\GG_m$-action on the relative flag scheme
\begin{equation}
    \Flag_r(\extE) = \shProj_{B\times\AA^1} S_λ(\extE)
\end{equation}
with a natural linearisation on the Serre line bundle which we denote by $\scrL_λ$. The central fibre is isomorphic to $\Flag_r(\shF\oplus\shQ)$.

Let $\scrL_λ$ be the line bundle on $\Y_\shF=\Flag_r(\extE)$ corresponding to a partition $λ\in\shP(r)$. The $\GG_m$-action on $E$ induces a linearised action on $(\Y_\shF,\scrL_λ)$. We extend this action trivially to any line bundle $\scrL_λ(f^*A)$, where $A\in\Pic B$ and $f\colon B\times\AA^1 \rightarrow B$ is the projection. We will abuse notation by writing this line bundle simply as $\scrL_λ(A)$.

\begin{claim}\label{cl:2}
    Assume that $B$ is a curve, $E$ is an ample vector bundle on $B$ and $A$ is an ample line bundle on $B$. Let $F$ be a subbundle of $E$ of positive degree and maximal slope with quotient $Q$. Then $(\Y_F,\scrL_λ(A),ρ)$ is a test configuration for $(\Flag_r(E),\shL_λ)$.
\end{claim}
\begin{proof}
It suffices to show that the polarisation $\scrL_λ(A)$ is ample over the central fibre. Since $E$ is ample, we may assume that $A=\shO_B$. By Proposition \ref{prop:schurpos} it suffices to show that $F\oplus Q$ is ample.

The bundle $E/F$ is ample since it is a quotient of an ample bundle. The subbundle $F$ has positive degree and it is stable so it is ample by \cite[Section 2]{Hartshorne1971}. Therefore the Schur power $(F\oplus Q)^λ$ is ample by Proposition \ref{prop:schurpos}, which proves the claim.
\end{proof}
\begin{remark}
    We fully expect the statement of Claim \ref{cl:2} to be true if $F$ is as above and we only assume $\shL_λ(A)$ to be ample.
\end{remark}
\begin{claim}\label{cl:1}
    Let $L$ be an ample line bundle on $B$. Then the pair $(\Y_F,\scrL_λ(L^m),\rho)$ is a test configuration for $m\gg 0$.
\end{claim}
\begin{proof}
    This follows immediately from \cite[Proposition 7.10]{hartshorne1977algebraic}.
\end{proof}
We call the $\GG_m$-linearised pair $(\Flag_r(E),\scrL_λ(A))$ the \emph{simple test configuration induced by $\shF$}.

Assume that the scheme $(\Y_\shF,\scrL_λ(A))$ is a test configuration and let $h(k)$ and $w(k)$ be the Hilbert and weight polynomials. Let $p_1$ and $p_2$ be the two projection of the product $B\times \PP^1$ and define the vector bundle
\begin{equation}\label{eq:tildeE}
    \widetilde{E}=p_1^*F\otimes p_2^*\shO_{\PP^1}(1)\oplus p_1^*Q.
\end{equation}
We write the vector bundle $\widetilde{E}$ simply as $\widetilde{E}=F(1)\ot Q$.

\begin{lemma}\label{lem:weight}
    The weight function $w(k)$ of the action $ρ$ and the Hilbert function $h(k) = h^0(\Flag_r(E),\shL(A)^k)$ satisfy the identity
    \begin{equation}
         w(k) + h(k) =  χ(B\times\PP^1,\widetilde{E}^λ\otimes p_1^*A).
    \end{equation}
\end{lemma}
\begin{proof}
    Assume first of all that $A=\shO_B$. By the Littlewood-Richardson rule (see \cite[(2.3.1) Proposition]{Weyman}) we have the decomposition
    \begin{equation}\label{eq:LWR}
        \widetilde{E}^λ=\bigoplus_{ν,μ}(F(1)^ν\otimes Q^μ)^{\oplus M^λ_{νμ}},
    \end{equation}
    where the sum is over all partitions $\nu$ and $\mu$ whose sizes sum up to the size of $\lambda$ and the coefficient $M^λ_{\nu,\mu}$ is the \emph{Littlewood-Richardson coefficient}. Using the K\"{u}nneth formula, Riemann-Roch on $\PP^1$ and additivity of the Euler characteristic we see that
\begin{equation}
    \begin{split}
            \chi(B\times\PP^1,\widetilde{E}^λ) &= \sum_{\nu,\mu,\lambda} M^λ_{\nu,\mu} \chi(B\times\PP^1,F^\nu\otimes Q^\mu \otimes \shO_{\PP^1}(|\nu|))\\
            &=\sum_{\nu,\mu,\lambda} (|\nu| + 1) M_{ν,μ}^λ χ(B,F^\nu\otimes Q^\mu)\\
            &=\chi(B,E^λ) + \sum_{|\nu|+|\mu|=|\lambda|} |\nu| χ(B,(F^\nu\otimes Q^\mu)^{\oplus M^λ_{\nu,\mu}}).
    \end{split}
\end{equation}
Assuming that the vector bundles $\widetilde{E}^λ$ and $E^λ$ are ample, the weight $w(k)$ is given by
\begin{equation}
    \begin{split}
            w(k) &=  \sum_{|\nu|+|\mu|=|\lambda|} |\nu| h^0\left(B,(F^\nu\otimes Q^\mu)^{\oplus M^λ_{\nu,\mu}}\right).
    \end{split}
\end{equation}
Finally, the calculation works verbatim if the bundle $A$ is nontrivial.
\end{proof}

Using Lemma \ref{lem:weight} we can calculate both the Hilbert and the weight polynomials using the Hirzebruch-Riemann-Roch formula. For the former, we have
\begin{equation}\label{eq:rrh}
    h(k) =\int_B \mathrm{ch}(E^{kλ}) \ch (A) \mathrm{Td}_B,
\end{equation}
and similarly for the latter, we have
\begin{equation}\label{eq:rrw}
    w(k) = \int_{B×\PP^1} \mathrm{ch}(\widetilde{E}^{kλ})\ch (A) \mathrm{Td}_{B×\PP^1} - h(k).
\end{equation}
There exist integers $a_0,a_1,b_0$ and $b_1$ so that we can write
\begin{equation}\label{eqh}
    χ(B,E^{kλ}) = \rank E^{kλ} \left(a_0k^b+a_1k^{b-1}+O(k^{b-2})\right)
\end{equation}
and
\begin{equation}\label{eqhh}
    χ(B×\PP^1,\widetilde{E}^{kλ}) = \rank E^{kλ} \left( b_0k^{b + 1}+b_1k^{b}+O(k^{b-1})\right).
\end{equation}
The common factor cancels and we get
\begin{equation}\label{eq:futaki}
    \DF\left(\Y_\shF,\scrL_λ(A)\right)=\frac{b_0a_1-b_1a_0 + a_0^2}{a_0^2}
\end{equation}
for the Donaldson-Futaki invariant.

The Chern classes of the twisted bundle $\widetilde{E}$ appearing in Equations~\eqref{eqh} and \eqref{eqhh} are given by the following Lemma.
\begin{lemma}\label{lem:intersections}
    Let $\widetilde{E}$ be the vector bundle defined in Equation~\eqref{eq:tildeE} and $\mathbf{h}$ is the fibre of a point under $p_2$. We have
    \begin{equation}
        \begin{split}
            h_2(\widetilde{E}) &= r_Fp_1^*c_1(E)\mathbf{h} + p_1^*c_1(F)\mathbf{h} + p_1^*h_2(E) + \frac{r_F(r_F+1)\mathbf{h}^2}{2}\\
            c_2(\widetilde{E}) &=  r_Fp_1^*c_1(E)\mathbf{h} - p_1^*c_1(F)\mathbf{h} + p_1^*c_2(E) + \frac{r_F(r_F-1)\mathbf{h}^2}{2}\\
            c_1(\widetilde{E}) &= p_1^*c_1(E) + r_F\mathbf{h}\\
            A_2(\widetilde{E}) &= - \frac{r_F}{r_E+1} \left(\frac{p_1^*c_1(E)\mathbf{h}}{r_E}-\frac{p_1^*c_1(F)\mathbf{h}}{r_F}\right) + Z
        \end{split}
    \end{equation}
    where $Z$ is contained in the image of $p_1^*$ and the class $A_2(\widetilde{E})$ is defined in Lemma \ref{lem:sym_chern}.
\end{lemma}
\begin{proof}
    The proposition follows by direct computation from the Whitney sum formula \cite[Theorem 3.2]{Fulton1998} and the general fact that we have
    \begin{equation}
        c_k(\shF\otimes L) = \sum_{j=0}^k\binom{r-i+j}{j}c_{k-j}(\shF)c_1(L)^j
    \end{equation}
    for any locally free sheaf $\shF$ and line bundle $L$ \cite[Example 3.2.2]{Fulton1998}. Alternatively, one may get the result using the splitting principle.
\end{proof}

\begin{remark}[Optimal test configurations]\label{rem:optimal}
Before proceeding with the proofs of Theorems \ref{thm:DFC} and \ref{thm:DFB}, we make a naive but natural conjecture to make about the optimality of the test configuration $\Y_\shF$. Assume that $B$ is K-stable and $\shF$ has maximal slope in the set of torsion free subsheaves of $E$. We conjecture that the test configuration $\Y_\shF$ is a \emph{maximally destabilising} test configuration of $(\Flag_r(E),\shL_λ(A))$ in the sense that the quantity $\frac{\DF(\Y)}{\|\Y\|}$ is bounded below by $\frac{\DF(\Y_\shF)}{\Y_\shF}$.

Optimality of test configurations in this sense was studied by Székelyhidi in the case of toric varieties \cite{szekelyhidi2008optimal}. The difficulty in the general case stems from the difficulty of parametrising the collection of test configurations, which is a partial motivation for our work on filtrations in Chapter \ref{chap:filtrations}.
\end{remark}

\section{Flag variety over a curve}\label{sec:curve} 
The aim of this section is to prove Theorem \ref{thm:DFC}.
\begin{proof}[Proof of Theorem \ref{thm:DFC}]
Let $B$ be a curve. Let $F$ be a subbundle of $E$ and $A$ a line bundle on $B$ such that the polarised scheme $(\Y,\scrL_λ(A))$, where $\Y = \Flag_r(\extE)$, is a test configuration for $(\Flag_r(E),\shL(A))$. We may assume that $\widetilde{E}^λ\ot A$ is ample, since twisting by the pullback $\shO_{\PP^1}(1)$ leaves Equation~\eqref{eq:futaki} invariant. We will show that the Donaldson-Futaki invariant of the test configuration $\left(\Y,\scrL_λ(π^*A)\right)$ satisfies
\begin{equation}
    \DF(\Y) = C_{g,E,A,λ}(μ_E-μ_F),
\end{equation}
where $C$ is a positive number depending on $B,A,E,F$ and $λ$.
By Riemann-Roch the Hilbert polynomial of $\shL_λ^k(A)$ satisfies
\begin{equation}
    χ(\Flag_r(E),\shL^k)=\rank E^{kλ}\left(a_0k + a_1 \right),
\end{equation}
where
\begin{equation}
    \begin{split}
        a_0&=c_1(λ)μ_E + μ_A,\\
        a_1&=1-g.
    \end{split}
\end{equation}
Using the Riemann-Roch formula on $B\times\PP^1$, we can write
\begin{equation}\label{eq:RRC}
     \chi(B,E^\lambda\otimes L^{mk}) =\int_{B\times\PP^1} {r_E}^{kc_1(A)} \ch(\widetilde{E}^{kλ}) \mathrm{Td}_{B\times\PP^1}.
\end{equation}
By Theorem \ref{thm:maina} we have
\begin{equation}
    h^0(B\times \PP^1,\widetilde E^{kλ})=\rank E^{kλ} (b_0k^2+b_1 k + O(1)),
\end{equation}
where denoted
\begin{equation}\label{eq:b0eq}
        b_0=\frac{h_2(λ)h_2(\widetilde E)}{r_E(r_E+1)}+\frac{c_2(λ)c_2(\widetilde E)}{r_E(r_E-1)} + \frac{c_1(λ)}{r_E}c_1(\widetilde{E}).c_1(A)
\end{equation}
and
\begin{equation}\label{eq:b1eq}
    b_1=H_λA_2(\widetilde E) - \frac{c_1(λ)c_1(\widetilde E).K_{B\times \PP^1}}{2r_E} - \frac{c_1(A).K_{B\times\PP^1}}{2}.
\end{equation}
Here the class $A_2(\widetilde{E})$ is defined in Equation~\eqref{eqaa} and we write
\begin{equation}
    H_λ=\frac{r_Ec_1(λ) - \sum_{i=1}^{r_c}(2i-1)λ_i}{r_E-1}.
\end{equation}
Let $\mathbf{g}$ and $\mathbf{h}$ be the two fibres of the first and second projection of the product $B\times\PP^1$, respectively. The intersection matrix with respect to this basis is
\begin{equation}
    \left( \begin{array}{cc}
    0 & 1  \\
    1 & 0  \end{array} \right).
\end{equation}
As a special case of Lemma \ref{lem:intersections} we have
\begin{equation}
    c_1(\widetilde{E})^2 = 2{r_F}{r_E}μ_E.
\end{equation}
Calculating the intersection classes appearing in Equations~\eqref{eq:b0eq} and \eqref{eq:b1eq} gives
\begin{equation}
    \begin{split}
        -\frac{c_1(\widetilde E).K_{B\times\PP^1}}{2}&=(f\mathbf{h}+(r_Eμ_E) \mathbf{g}).(\mathbf{h}+(1-g)\mathbf{g})\\
        &=r_Eμ_E + \frac{(1-g)c_1(\widetilde E)^2}{2r_Eμ_E},\\
        -\frac{c_1(A).K_{B\times\PP^1}}{2} &= μ_A\mathbf{g}.(\mathbf{h} + (1-g)\mathbf{g}) =  μ_A, \mbox{ and}\\
        c_1(\widetilde{E}).c_1(A) &= {r_F}μ_A.
    \end{split}
\end{equation}
Let $y=(y_1,\ldots,y_l)$ be variables. For any such $y$ define the symmetric polynomial
\begin{equation}
    A_2(y) = \frac{{r_E}-1}{2}\left(\frac{h_2(y)}{{r_E}({r_E}+1)}-\frac{c_2(y)}{{r_E}({r_E}-1)}\right).
\end{equation}
Using the above calculations and Remark \ref{rem:abrelation} we then have
\begin{equation}\label{eq:CW}
    \begin{split}
        b_0&=\frac{2({r_E}+1)}{{r_E}-1}A_2(λ)A_2(\widetilde E) + \frac{c_1(λ)^2c_1(\widetilde E)^2}{2r_E^2} + \frac{c_1(λ){r_F}μ_A}{r_E},\\
        b_1&=H_λA_2(\widetilde E) + a_0 + \frac{(1-g)c_1(λ)c_1^2(\widetilde E)}{2r_E^2μ_E},
    \end{split}
\end{equation}
By direct calculation, and Lemma \ref{lem:intersections} the Donaldson-Futaki invariant defined in Equation~\eqref{eq:futaki} is given by
\begin{equation}
    \begin{split}
        \DF(\Y) &= \left(a_1b_0 - a_0b_1 + a_0^2\right) / a_0^2\\
        &= C_{g,E,A,λ}(μ_E-μ_F),
    \end{split}
\end{equation}
where the constant $C_{g,E,A,λ}$ is given by
\begin{equation}\label{eq:Cpos}
    C_{g,E,A,λ}=\frac{r_F}{(r_E+1)\left(c_1(λ)μ_E +μ_A\right)^2} \left(H_λ\left(c_1(λ)μ_E + μ_A\right)+\frac{2(g-1)(r_E+1)A_2(λ)}{{r_E}-1}\right).
\end{equation}
We are left to verify that the constant $C_{g,E,A,λ}$ is positive. For $g\geq 1$, it suffices to show that $H_λ$ and $A_2(λ)$ are positive since $c_1(λ)μ_E+μ_A$ is positive as $\shL_λ(A)$ is ample.

Using $r_E-1\geq r_c$ and recalling that $r_c$ is the length of $λ$, we have
\begin{equation}
    \begin{split}
        (r_E+1)r_E(r_E-1)A_2(λ) &= (r_E-1)c_1(λ)^2 - 2r_Ec_2(λ)\\
        &= (r_E-1)\sum_{i=1}^{r_c}λ_i^2-2\sum_{1\leq i< j\leq r_c}λ_iλ_j\\
        &\geq\sum_{1\leq i < j\leq r_c}(λ_i-λ_j)^2 >0.
    \end{split}
\end{equation}
We have
\begin{equation}
     \sum_{i=1}^l(2i-1)λ_i = \sum_{j=1}^{c} (λ_i')^2,
\end{equation}
where $λ'$ denotes the conjugate partition of $λ$. To see that the first term of Equation~\eqref{eq:Cpos} is positive, notice that
\begin{equation}
    ec_1(λ)-\sum_i(2i-1)λ_i =  \sum_{j=1}^{s} λ'_i(r_E-λ'_i)>0,
\end{equation}
which is positive since $r_E>r_c\geqλ'_i$ for all $i$. Hence $C_{g,E,A,λ}>0$ for all $g\geq1$. A similar calculation shows that $C_{0,E,A,λ}$ is positive.
\end{proof}

\section{Flag variety over a base of higher dimension}\label{sec:anybase}

Our aim is to prove Theorem \ref{thm:DFB}. We proceed in two stages. First, we assume for simplicity that the test configuration is induced by a subsheaf of $E$. Finally, we use Proposition \ref{prop:cohsub} that this can be done without loss of generality.

\begin{proof}[Proof of Theorem \ref{thm:DFB}]
By Proposition \ref{prop:cohsub} we may assume that $F$ is a subbundle. We will show that the leading term in $m$ in the Donaldson-Futaki invariant of the test configuration $(\Y,\scrL_λ(p_1^*L^m))$
is
\begin{equation}\label{eq:Ddef}
    D_{E,λ,L,r_F}(μ(E)-μ(F)),
\end{equation}
where $D_{E,λ,L,r_F}$ is a positive number depending on $B,L,E,F$ and $λ$. Here $p_1$ is the first projection from $B\times \AA^1$. Expand the Chern character of $E^{kλ}$ as
\begin{equation}
    \ch E^{kλ} = \sum_{i=0}^b \ch_i E^{kλ}
\end{equation}
and the Todd class of $B$ as
\begin{equation}
    \operatorname{Todd}(B) = \sum_{i=0}^b \operatorname{Todd}_i(B).
\end{equation}
We then have
\begin{equation}\label{eqrrA}
    \begin{split}
             \chi(\Flag_r(E),\shL_λ(L^m)^{\otimes k}) &=  \chi(B,E^{kλ}\otimes L^{mk}) \\
             &= \int_B {r_E}^{mk\omega} \ch(E^{kλ}) \mathrm{Td}(B)  \\
            &=\frac{(m  k)^b}{b!}\omega^b\rank(E^{kλ})  \\
              +\frac{(m k)^{b-1}}{(b-1)!}\omega^{b-1}&\Big(\rank(E^{kλ})\frac{c_1(B)}{2} + \frac{kc_1(λ)c_1(E^λ)}{r_E} \Big) \\
              +\frac{(m k)^{b-2}}{(b-2)!}\omega^{b-2}\Big(\rank(E^{kλ})&\operatorname{Todd}_2(B) + \frac{kc_1(λ)c_1(E^λ).c_1(B)}{2r_E} + \ch_2(E^{kλ})\Big)\\
             + O(k^{b-3})&,
    \end{split}
\end{equation}
which follows from Riemann-Roch and the pushforward formula of Proposition \ref{prop:pushforward}. Here $\mathrm{Td}_2(B)$ is the second Todd class of $B$. Using Riemann-Roch on $B\times\PP^1$, we similarly compute the Hilbert polynomial of $\widetilde{E}^λ\otimes p_1^*L^m$, where $p_1$ is the first projection.

To apply Lemma \ref{lem:weight}, choose $m_0$ so that the bundle $E\otimes L^{\frac{m_0}{c_1(λ)}}$ is ample and assume from now on that $m>m_0$.

As in Section~\ref{sec:curve}, we write
\begin{equation}
    \begin{split}
        h^0(B,E^{kλ}\otimes L^{mk}) &= \rank E^{kλ} \left(a_0 k^b+a_1 k^{b-1} + O(k^{b-2})\right),\\
        h^0(B\times \PP^1, \widetilde{E}^{kλ}\otimes L^{mk})&=\rank E^{kλ}\left(b_0 k^{b+1}+b_1 k^{b} + O(k^{b-1})\right).
    \end{split}
\end{equation}
Next, we expand the $a_i$ and the $b_i$ in powers of $m$ as
\begin{eqnarray}\label{eq:atr}
b_0&=&b_{0,0} m^{b} + b_{0,1} m^{b-1} +O(m^{b-2}),  \\
b_1&=&b_{1,0} m^{b} + b_{1,1} m^{b-1} +O(m^{b-2}), \\
a_0&=&a_{0,0} m^b + a_{0,1} m^{b-1} +O(m^{b-2}),\\
a_1&=&a_{1,0} m^b + a_{1,1} m^{b-1} +O(m^{b-2}).
\end{eqnarray}
Let $ω=c_1(L)$ and $\eta=p_1^*ω$. Using Theorem \ref{thm:maina} and equation \eqref{eqrrA}, we see that
\begin{eqnarray*}
            b_{0,0} &=& \frac{c_1(λ)}{ r_E\cdot b!}\eta^b.c_1(\widetilde E)\\
            b_{0,1} &=&  \frac{1}{(b-1)!}\eta^{b-1}.\left( \frac{h_2(λ)h_2(\widetilde E)}{{r_E}({r_E}+1)}+\frac{c_2(λ)c_2(\widetilde E)}{{r_E}({r_E}-1)} \right) \\
            b_{1,0} &=&  -\frac{ \eta^b.K_{B\times\PP^1}}{2\cdot b!} \\
            b_{1,1} &=&  \frac{1}{(b-1)!}\left(\eta^{b-1}.H_λA_2(\widetilde E)  -   \frac{c_1(λ)\eta^{b-1}.K_{B\times\PP^1}.c_1(\widetilde E)}{2r_E}\right) \\
            a_{0,0} &=&  \frac{\omega^{b}}{b!} = \frac{ \deg L}{b!} \\
            a_{0,1} &=&  \frac{c_1(λ)}{r_E(b-1)!}\omega^{b-1}.c_1(E)  \\
            a_{1,0} &=& 0 \\
            a_{1,1} &=&- \frac{\omega^{b-1}.K_B}{2 (b-1)!} = -\frac{\deg K_B}{2 (b-1)!}.
\end{eqnarray*}
The proof of the following lemma is a straightforward calculation.
\begin{lemma}
    The intersection numbers appearing above are
    \begin{eqnarray*}
        \omega^{b}&=&\deg L \\
        \omega^{b-1}.c_1(E)&=&r_Eμ_E\\
        \omega^{b-1}.K_B&=&\deg K_B\\
        \eta^b.c_1(\widetilde E) &=& \deg L({r_E}\alpha +{r_E}) \\
        \eta^{b-1}.c_1(\widetilde E)^2 &=&  2 {r_F}{r_E}μ_E  \\
        \eta^{b-1}.c_2(\widetilde E) &=&  {r_F}{r_E}μ_E -{r_F}μ_F  \\
        \eta^{b-1}.K_{B\times\PP^1}.c_1(\widetilde E) &=& f\deg K_B - 2{r_E}μ_E \\
        \eta^{b}.K_{B\times\PP^1} &=& -2 \deg L \\
        \eta^b.A_2(\widetilde E) &=& \frac{{r_E}(μ_E-μ_F)}{{r_E}+1}.
    \end{eqnarray*}
\end{lemma}
We write Laurent expansion of the Donaldson-Futaki invariant in $m$
\begin{equation}
        \DF(\Y, \shL_{\shE,m},\rho) = F_0 + F_1m^{-1} + O(m^{-2}),
\end{equation}
where
\begin{equation}
            F_0 a_0^2 =  \underbrace{a_{1,0}b_{0,0}}_{=0} - a_{0,0}b_{1,0} + a_{0,0}^2 =- \left(\frac{\deg L}{b!}\right)^2 + \left(\frac{\deg L}{b!}\right)^2 = 0
\end{equation}
and
\begin{equation}
             F_1 a_0^2=  \underbrace{a_{1,0}b_{0,1}}_{=0} + a_{1,1}b_{0,0} - a_{0,1} b_{1,0} - b_{1,1}a_{0,0} + 2a_{0,0}a_{0,1}.
\end{equation}
An elementary calculation similar to the one we did in Section~\ref{sec:curve} shows that
\begin{equation}
    \mathrm{DF}(\Y, \scrL_λ(p_1^*L^m)) =D_{E,λ,L,{r_E}}(μ_E-μ_F)m^{-1} + O(m^{-2})
\end{equation}
where
\begin{equation}
    D_{E,λ,L,{r_E}} =\frac{r_FbH_λ}{({r_E}+1)\deg L}
\end{equation}
is a positive constant by the same argument as in Section~\ref{sec:curve}. Theorem \ref{thm:DFB} then follows from the following Proposition.
\end{proof}
\begin{proposition}\label{prop:cohsub}
    Using notation from Section \ref{sec:test}, let $\left( \Flag_r(\extE),\scrL_λ(L^m) \right)$ be a test configuration for $(\Flag_r(E),\shL_λ(L^m))$ where $\shF$ is a saturated torsion free subsheaf of $E$. Then the formula
    \begin{equation}
        \mathrm{DF}(\Flag_r(\extE), \scrL_λ(L^m)) =D_{E,λ,L,{r_E}}(μ_E-μ_\shF)m^{-1} + O(m^{-2})
    \end{equation}
    for the Donaldson-Futaki invariant still holds for $m\gg 0$.
\end{proposition}
\begin{proof}
    It follows that $E/\shF$ is also torsion free, and $\shF$ and $E/\shF$ are both locally free over an open subset $U$ whose complement is of dimension at least 2. The leading order terms in $m$ of $h(k)$ and $w(k)$ given in Equation~\eqref{eq:atr} only involve the first Chern classes of $\shF$ and $E/\shF$. But the first Chern classes can be computed over the open set $U$ where $F$ and $E/\shF$ are locally free. The Schur functor commutes with localisation, so Theorem \ref{thm:maina} holds for the restriction $\restr{\left(\shF\oplus E/\shF^λ\right)}{U}$. Therefore, we may assume without loss of generality that $\shF$ is a subbundle.
\end{proof}



\chapter{Uniformisation theorem for flag bundles over Riemann surfaces} 
\label{chap:UTF}
We show that there is a simple extension of the Uniformisation Theorem to flag varieties of polystable vector bundles over Riemann surfaces.

Throughout this chapter we let $C$ be a curve and denote its fundamental group by $Γ$ without reference to the choice of a base point. Let $\widehat{C}$ be the universal cover of $C$, which is one of the three model spaces given by the Uniformisation theorem. Let $π$ be the canonical projection $\widehat{C}\rightarrow C$ and $σ$ the covering action $\widehat{C}\timesΓ\rightarrow \widehat{C}$.

\begin{theorem}\label{thm:csck}
    Let $E$ be a polystable vector bundle on $C$ and let $\Flag_r(E)$ be a flag bundle of $E$ over $C$. All Kähler classes in $\Flag_r(E)$ are cscK. In particular, $\Flag_r(E)$ is K-semistable for all polarisations.
\end{theorem}
We obtain a partial Yau-Tian-Donaldson correspondence for flag bundles on high genus curves using Theorem \ref{thm:csck}.
\begin{theorem}\label{thm:ytdc}
    Let $(\Flag_r(E),\shL_λ(A))$ be a polarised flag bundle on $C$.

    If $E$ is polystable, the flag bundle $(\Flag_r(E),\shL_λ(A))$ is K-semistable. If $E$ is stable and $g\geq 2$, then the variety $(\Flag_r(E),\shL_λ(A))$ is K-stable.

     Finally, if $E$ is simple and $g\geq 2$, the YTD correspondence holds for any line bundle $\shL_λ(A)$ with $λ\in \shP_\diamond(r)$ and $A$ ample.
\end{theorem}
We prove the following Lemma in Section~\ref{sec:constant_scalar_curvature_kahler_metrics_on_flag_bundles}.
\begin{lemma}\label{lem:autF}
    If the vector bundle $E$ is simple and the genus satisfies $g\geq 2$, then the automorphism group of $\Flag_r(E)$ is discrete.
\end{lemma}
\begin{proof}[Proof of Theorem \ref{thm:ytdc}]
    The first statement follows directly from Theorem \ref{thm:csck} and Proposition \ref{prop:stab}.

    For the second statement, we also need Lemma \ref{lem:autF} and Proposition \ref{prop:CscKtoStab} which strengthens Proposition \ref{prop:stab} in the case of a discrete automorphism group.

    If $E$ is polystable, the final statement follows from the second statement. If $E$ is simple but not polystable, then we can construct a destabilising test configuration for $(\Flag_r(E),\shL_λ(A))$ by Theorem \ref{thm:curve1}.
\end{proof}

\begin{remark}
    In order to prove a full YTD correspondence on flag bundles over curves one would need to analyse the delicate cases when $\Flag_r(E)$ admits vector fields. By Equation \eqref{eq:rtes} and the preceding discussion we see that this may happen when the base curve $C$ is an elliptic curve and when $E$ is properly polystable, that is, isomorphic to a direct sum of stable vector bundles of equal slopes. If the base curve $C$ is isomorphic to $\PP^1$, Grothendieck's theorem states that any holomorphic vector bundle $E$ can be decomposed into a direct sum $\bigoplus_{i=1}^{r_E} \shO_{\PP^1}(m_i)$ for some $m_i\in \ZZ$ for $i=1,\ldots,r_E$ \cite{grothendieck1957classification}.
\end{remark}

\section{Construction of flag bundles from representations of the fundamental group} 
\label{sec:construction_of_flag_bundles_from_representations_of_the_fundamental_group}

Let $G$ be an algebraic group and $ρ\colonΓ\rightarrow G$ be a representation. We define the \emph{associated bundle with fibre $G$} \cite{kobayashi2014differential}
\begin{equation}\label{eq:FId}
    \EE_ρ =  \bigslant{\widehat{C}\times G}{Γ}
\end{equation}
by the identification
\begin{equation}
    (c,g)\sim(σ(γ,c), ρ(γ)g)
\end{equation}
for $(c,g)\in \widehat{C}\times G$ and $γ\in Γ$. The quotient space $\EE_ρ$ is an algebraic principal bundle over the curve $C$.

A representation $ρ\colonΓ\rightarrow\GL(e,\CC)$ determines a vector bundle $E_ρ$ by setting
\begin{equation}
    E_ρ = \bigslant{C\times \CC^{r_E}}{Γ}
\end{equation}
by the identification in Equation~\eqref{eq:FId} with $\GL(e,\CC)$ acting on $\CC^{r_E}$ in the usual way. The vector bundle $E_ρ$ and its associated frame bundle $\EE_ρ$ have natural Zariski trivial algebraic structures since the fibre of $\EE_ρ$ is $\GL({r_E},\CC)$ \cite{serre1958espaces}.

A \emph{locally trivial holomorphic fibration} with fibre $F$ is a holomorphic map $f\colon M\rightarrow M'$ of complex manifolds $M$ and $M'$ such that each point $x\in M'$ has an analytic neighborhood $U\subset M'$ such that the restriction of $f$ to $U$ is given by the first projection $U\times F\rightarrow U$.
\begin{theorem}\label{thm:unif}
    Suppose that $E$ is polystable vector bundle over a (complex, smooth, projective) curve $C$. Let $\bar{P}_r$ denote the image of the parabolic subgroup $P_r\subset \GL({r_E},\CC)$ in $\PGL({r_E},\CC)$. Then there exists representation $ρ\colonΓ\rightarrow \PGL({r_E},\CC)$ such that the holomorphic quotient map
    \begin{equation}\label{eq:FQuot}
          \bigslant{\widehat{C}\times \PGL(r,E)}{\bar{P}_r} \rightarrow \Flag_r(E)
    \end{equation}
    is a holomorphic locally trivial fibration with fibre $Γ$.
\end{theorem}

\begin{proof}[Proof of Theorem \ref{thm:unif}]
    Let $\EE$ be the frame bundle of $E$ and define the \emph{projectivised frame bundle}
    \begin{equation}
        \bar{\EE} \defeq \bigslant{\EE}{\GG_m},
    \end{equation}
    where $\GG_m$ acts via the inclusion
    \begin{equation}
        λ\mapsto λ I\in \GL({r_E},\CC)
    \end{equation}
    for $λ\in\GG_m$. By the Narasimhan-Seshadri Theorem \ref{prop:ns} there exists a representation $ρ:Γ\rightarrow \PGL({r_E},\CC)$ such that $\bar{\EE}$ is the associated bundle
    \begin{equation}
        \bar{\EE} = \bigslant{\big(\widehat{C} \times \PGL({r_E},\CC)\big)}{Γ}.
    \end{equation}
    of the representation $ρ$
    Since multiples of the identity matrix are contained in $P_r$ we can write
    \begin{equation}
        \Flag_r(E) = \bar{\EE} / \bar{P}_r.
    \end{equation}
    Hence the representation $ρ$ induces an action of $Γ$ on $\Flag_r(E)$. The double quotient
    \begin{equation}
        \widehat{C}\times \PGL({r_E},\CC) \longrightarrow \EE \longrightarrow \Flag_r(E)
    \end{equation}
    can be factorised in two ways. We define the map
    \begin{equation}
        \hat{π}:\widehat{C}\times \PGL({r_E},\CC)/\bar{P}_r \longrightarrow \Flag_r(E)
    \end{equation}
    by
    \begin{equation}
        (x,g\bar{P}_r)\mapsto \left(σ(Γ,x),ρ(Γ)g\bar{P}_r\right)\in\Flag_r(E).
    \end{equation}
    The map $\hat{π}$ fits into the diagram
    \[
    \begin{tikzcd}
    \widehat{C}\times \PGL({r_E},\CC) \arrow{r}{} \arrow[swap]{d} & \EE \arrow{d} \\
    \widehat{C} \times \PGL({r_E},\CC)/\bar{P}_r \arrow{r}{\hat{π}} & \Flag_r(E)
    \end{tikzcd}
    \]
    and is a locally trivial holomorphic fibration with fibre $Γ$, since $π$ is.
\end{proof}

\section{Constant scalar curvature Kähler metrics on flag bundles and K-polystability} 
\label{sec:constant_scalar_curvature_kahler_metrics_on_flag_bundles}

We begin with a proof of Theorem \ref{thm:csck}, then turn to the proof of Lemma \ref{lem:autF}.
\begin{proof}[Proof of Theorem \ref{thm:csck}]
    Let $G$ denote the group $\PGL({r_E},\CC)$. The Picard group of $\Flag_r(E)$ is generated by line bundles of the form $\shL_λ(A)$ where $λ$ is in $\shP(r)$ and $A$ is a line bundle on $C$ by Lemma \ref{lem:FPic}.

    Fix a line bundle $M=\shL_λ\otimes A$ with $A\in \Pic C$ and $λ\in\shP(λ)$. Let
     \begin{equation}
        π \colon \widehat{C}\times G/P_r\rightarrow \Flag_r(E)
     \end{equation}
    be the projection constructed in Theorem \ref{thm:unif}.

    There is a Kähler-Einstein (hence cscK) metric $ω_0$ in $c_1(\shL_λ)$, unique up to the action of $G$, by results of Koszul and Matsushima \cite{alekseevskii1986invariant}. Let $s_0$ be the (constant) scalar curvature of $ω_0$. Let $ω_A$ be a constant scalar curvature metric such that $2π[ω_A] = c_1(A)$ with scalar curvature $s_1$ and let $ω_1$ be the pullback to $\widehat{C}$. Since $ω_0+ω_1$ is $Γ$-invariant, it descends to a form $ω$ on $\Flag_r(E)$ with constant scalar curvature $s_0 + s_1$.
\end{proof}

Let $V$ be a complex vector space of dimension ${r_E}$. In order to apply a classical result of Demazure, we regard $\Flag_r(V)$ as a quotient of $\PGL(r,V)$. Let $Q_r$ be the image of a stabiliser of an $r$-flag of subspaces in $\PSL({r_E},\CC)$ and let $\mathfrak{q}_r$ be its Lie algebra. Also let $\mathfrak{psl}({r_E},\CC)$ denote the Lie algebra of $\PSL({r_E},\CC)$. We have a well known exact sequence
\begin{equation}
    0\lra  (\PSL({r_E},\CC) \times \mathfrak{q}_r)/Q_r \lra \PSL({r_E},\CC)/Q_r \times \mathfrak{psl}({r_E},\CC)\lra
    \shT_{\Flag_r(V)}\lra 0.
\end{equation}
where $Q_r$ acts on $\mathfrak{q}_r$ by the adjoint action and $\shT_{\Flag_r(V)}$ is the tangent bundle.

It follows from results of Demazure and Bott \cite[Section 4.8]{akhiezer1995lie} that we have
\begin{equation}\label{eq:dem}
    H^i\left(\Flag_r(V), \shT_{\Flag_r(V)}\right) =
    \begin{cases} \mathfrak{psl}({r_E},\CC), \text{ if } i=0  \\
                    0, \text{ otherwise.}
    \end{cases}
\end{equation}

Let $p:\Flag_r(E)\rightarrow C$ be the projection. Since $\Flag_r(E)$ is Zariski locally trivial on $C$, this generalises in a straightforward manner. Let $h$ be a hermitian metric on $E$ and let $\End^0(E)$ denote the sheaf of trace-free endomorphisms on $E$. Let $U$ be a Zariski open set in $C$ such that
\begin{equation}
    \Flag_r(E)\cong U \times \Flag_r(V).
\end{equation}
We have a natural identification
\begin{equation}
    \restr{\left(\shEnd^0(E)/\CC \right)}{U} \cong \restr{\shO_B}{U}\ot \mathfrak{psl}({r_E},\CC),
\end{equation}
where the $\CC$ denotes the constant sheaf included in $\shEnd^0(E)$ as multiples of the identity. Let  $\shV_{\Flag_r(E)}$ denote the relative tangent bundle of $\Flag_r(E)$ with respect to the projection $p$. We obtain from Equation~\eqref{eq:dem}
\begin{equation}
    R^ip_* \shV_{\Flag_r(E)} =
    \begin{cases} \shEnd^0(E)/\CC \text{ if } i=0 \text{ and} \\
                    0 \text{ otherwise,}
    \end{cases}
\end{equation}

\begin{proof}[Proof of Lemma \ref{lem:autF}]
    We must show that the vector space $H^0(\Flag_r(E),\shT_{\Flag_r(E)})$ is trivial. We have the exact sequence
    \begin{equation}\label{eq:rtes}
        0\longrightarrow \shV_{\Flag_r(E)} \longrightarrow \shT_{\Flag_r(E)} \longrightarrow p^*\shT_C\longrightarrow 0
    \end{equation}
    where $\shT_C$ is the tangent bundle of the curve $C$. It suffices to show that $H^0(\Flag_r(E),\shV_{\Flag_r(E)}) = 0$ since $H^0(C,\shT_C) = 0$ as the genus $g(C)$ satisfies $g(C) > 1$. The vector bundle $E$ is simple, therefore we have $H^0(C,\shEnd(E)) = \CC \cdot \mathrm{Id}_E$. The claim follows by identifying $H^0(C,\shEnd^0(E))$ as a subspace of $H^0(C,\shEnd(E))$.
\end{proof}



\chapter{K-stability of complete intersections}\label{chap:sub} 
The objective of this chapter is to provide additional examples of K-unstable varieties. We describe a situation in which the Donaldson-Futaki invariant of a complete intersection can be calculated. In Section~\ref{sec:koszul} and apply the result in the case of flag bundles in Section~\ref{sec:subvarieties_in_flag_varieties}.

The idea is to fix a complete intersection $X$ in a polarised variety $Y$ and a test configuration $\Y$ for $Y$. Consider then the Zariski closure of the orbit of $X$ in $\Y$ under the $\GG_m$-action. The scheme $\X$ is a test configuration for $X$ and its Donaldson-Futaki invariant depends, a priori, on the test configuration $\Y$ in a complicated way. However, in some favourable situations the Donaldson-Futaki invariant of $\X$ is related to the Donaldson-Futaki invariant of $\Y$ and topological data of $X$ in $Y$. Examples of this behaviour have been given by Stoppa-Tenni \cite{Stoppa2010} and Arezzo-Della Vedova \cite{arezzo2011k}.

The main result of this chapter is a generalisation of an example in \cite{Stoppa2010}.
\begin{theorem}[A simple limit for high genus curves]\label{thm:simplelimit}
    Let $E$ be an ample vector bundle of rank $r_E$ on a curve, and $F$ is a subbundle of $E$ of rank $r_{F}$. Assume that
    \begin{equation}
        (\Y,\scrL)=\left(\Flag_r(\extE),\shL_λ\right)
    \end{equation}
    is a test configuration for $(\Flag_r(E),\shL_λ)$ as defined in Chapter \ref{chap:flags}, and that $λ$ is in $\shP_\diamond(r)$. Let $X$ be a generic complete intersection in $\Flag_r(E)$ of codimension less than the integer $N_{λ,r_E,r_F}$ defined in Equation~\eqref{eq:lws}. Then the Donaldson-Futaki invariant of the test configuration $\X$, defined as the closure of the orbit of $X$ in $\Y$, is given by
    \begin{equation}\label{eq:futlim}
        \DF(\X) = D\left(C_E\deg E + C_F\deg F \right)g + O(g^0),
    \end{equation}
    where $D$ is a positive number and $C_E$ and $C_F$ are given in Equation~\eqref{eq:CECF}. All three numbers depend only on $\deg E,\deg F$, the codimension $u$ of $X$ and $λ$.
\end{theorem}
We may easily construct examples of K-unstable complete intersections in flag bundles over curves using Theorem \ref{thm:simplelimit}. The simplest such construction is due to Stoppa and Tenni.

Fix a positive integer $d$ and let $C(g)$ be a sequence of $d$-gonal curves of genus $g$ for all integers $g$ larger than 2, and let $L_g$ be a degree $d$ line bundle on $C(g)$. Let
\begin{align*}
    F_g=L_g \text{ and } &  E_g = \shO_{C(g)}^{\oplus r_E-1}\oplus L_g.
\end{align*}
With these choices $\deg E_g$ and $\deg F_g$ are bounded as functions of $g$ and the final term in Equation \eqref{eq:futlim} is under control. The vector bundle $E_g$ is only globally generated but we may find a test configuration for an ample polarisation on $X$ whose Donaldson-Futaki invariant is arbitrarily close to the one given by Equation~\eqref{eq:futlim} when applied to the globally generated vector bundle $E_g$. We do this by replacing the vector bundle $E_g$ with $E_g\ot A^{\frac{ε}{|λ|}}$, where $A$ is an ample line bundle on $C(g)$. Finally, we use the following Lemma which follows directly from calculations done in Sections~\ref{sec:curve} and \ref{sec:koszul}.
\begin{lemma}\label{lem:futcont}
    The Donaldson-Futaki invariant of $\left(\Y_F,\scrL_λ(εA)\right)$ is continuous in $ε$.
\end{lemma}
Using Lemma \ref{lem:futcont} and simple combinatorics outlined in Section~\ref{sec:subvarieties_in_flag_varieties} we obtain the following new examples of K-unstable varieties.
\begin{theorem}[Theorem \ref{summaryD}]\label{thm:grEx}
    Let $Y$ be the Grassmannian of $p$-dimensional quotients of $E_g$ with the polarisation $\shL_λ(εA)$, where $λ=(1^p)$. Let $s$ be a positive integer.

    Then there exists numbers $ε_0>0$ and $g_0>0$ such that a general hypersurface $H$ in $Y$ which is a defined by a section of a multiple of $s\left(\shL_λ(εA)\right)$ with the polarisation $\restr{\shL_λ(εA)}{H}$ is K-unstable for all $ε<ε_0$ and $g>g_0$.
\end{theorem}
We may also ask for $H$ to be smooth in the statement of Theorem \ref{thm:grEx} by Bertini's theorem \cite[Theorem II.8.18]{hartshorne1977algebraic}.
\begin{proposition}\label{prop:gentype}
    For $s>e$ the hypersurface $H$ is of general type.
\end{proposition}
\begin{proof}
    We prove that $K_H$ is ample. This follows directly from the adjunction formula \cite[Example 3.2.12]{Fulton1998}. In the notation of Theorem \ref{thm:grEx}, we have
    \begin{equation}
        K_H = \restr{\left(\shL_{-σ} + K_{C(g)} + s \shL_λ(εA)\right)}{H},
    \end{equation}
    where $σ$ is the partition $({r_E}^p)$. The statement then follows from Remark \ref{rem:SMult} and the preceding discussion.
\end{proof}

\section{The Donaldson-Futaki invariant of a complete intersection} 
\label{sec:koszul}
Let $ρ$ be $\GG_m$-action on a polarised variety $(Y,L)$ of dimension $n$ and let $φ_i$ be sections of $H^0(Y,L^{s_i})$ for $1\leq i\leq u$. Let $γ$ be an integer, and assume that the natural representation of $ρ$ on $H^0(Y,L^{s_i})$ acts on $φ_i$ by $t.φ_i = t^{γs_i}φ_i$ for all $i$ and $t\in\GG_m$. Denote the complete intersection of $φ_1,\dotsc,φ_u$ by $X$. The $\GG_m$-action determines a product test configurations $\Y$ for $(Y,L)$ and $\X$ for $(X,\restr{L}{X})$, since $X$ is invariant under $ρ$.

Write the Hilbert and weight functions of $\Y$ and $\X$ as
\begin{align*}
    h^0_Y(k) &=  a_0 k^n+a_1k^{n-1}+O(k^{n-2}),\\
    w_Y(k) &= b_0 k^{n+1}+b_1k^n+O(k^{n-1}),\\
    h^0_X(k) &= c_0k^{n-u}+ c_0k^{n-u-1} + O(k^{n-u-2})\\
    \intertext{and}
    w_X(k) &= d_0k^{n-u+1}+ c_0k^{n-u} + O(k^{n-u-1}),
\end{align*}
respectively. The following Proposition is a special case of \cite[Theorem 4.1]{arezzo2011k}. We present an elementary proof in Section~\ref{sec:combiproofs} of the Appendix along the lines of \cite{Stoppa2010}.
\begin{proposition}\label{prop:cidf}
The Donaldson-Futaki invariant of the test configuration $\X$ is given by
    \begin{equation}\label{eq:cidf}
        \DF(\X) = \DF(\Y) + \frac{ν_Y-γ}{n+1-u}\left(\frac{(n+1)S}{2u}-\frac{uμ_Y}{n}\right),
    \end{equation}
    where we have denoted
    \begin{align*}
        ν_Y=\frac{b_0}{a_0},  \quad S=\sum_{i=1}^u s_i \quad \text{ and } \quad μ_Y=\frac{a_1}{a_0}.
    \end{align*}
%
\end{proposition}

The result of Proposition \ref{prop:cidf} also applies also to test configurations which are not products. Assume that $(\Y,\scrL)$ is an arbitrary test configuration for $(Y,L)$. Assume for simplicity that the exponent is 1. Let
\begin{equation}
    R= \bigoplus_{k=0}^\infty R_k=\bigoplus_{k=0}^\infty H^0(Y,L^k)
\end{equation}
be the graded coordinate ring of $(Y,L)$ and let $F_\bullet R$ be a graded filtration corresponding to the test configuration $\Y$ (cf. Remark \ref{rem:ringfilt}). We have an induced map
\begin{equation}\label{eq:largestlevel}
    R\longrightarrow R_γ\defeq \bigoplus_{k=0}^\infty R_k/F_{n_k-1}R_k,
\end{equation}
where $n_k$ is the smallest integer such that $F_{n_k}R_k = R_k$, which is finite by condition (iii) of Remark \ref{rem:ringfilt}. Let $I_γ$ be the ideal generated by $\bigoplus_{k=0}^\infty F_{n_k - 1}R_k$. Define the \emph{subscheme of least weight} of the test configuration $\Y$ to be the subscheme of $Y$ determined by $R/I_γ$.

The limit of the subscheme of least weight is fixed under the $\GG_m$ action over the central fibre. Slightly more generally, the following lemma follows directly from the definition of the scheme $Y_γ$.
\begin{lemma}
    The closure of the orbit of the subscheme of least weight $Y_γ$ in $\Y$ is isomorphic to $Y_γ\times\AA^1$ as (quasi-projective) polarised varieties. Moreover, the lifting of the $\GG_m$-action on $\AA^1$ to $Y_γ\times\AA^1$ is trivial with a possibly nontrivial linearisation.
\end{lemma}
\begin{proof}
    Let $\Y_γ$ denote the closure of $Y_γ$ under the $\GG_m$-action. Consider the linear map
    \begin{equation}
        Φ\colon R\rightarrow \bigoplus_{k=0}^\infty\bigoplus_{i=0}^\infty \frac{F_iR_k/F_{i-1}R_k}{J}
    \end{equation}
    defined by the projection $R_k\rightarrow R_k/F_{n_k - 1} R_k$ and $J$ is generated by all the elements which lie in $\bigoplus_{k=0}^\infty F_{n_k - 1}R_k$. It is straightforward to see that $Φ$ is a homomorphism of graded rings whose kernel is exactly the ideal $I_γ$.  Finally, the scheme $\Y_γ$ is isomorphic to the product $Y_γ\times \AA^1$ since it is the projectivisation of the ring
    \begin{equation}
        \Rees F_\bullet R/\widetilde{J},
    \end{equation}
    where $\widetilde{J}$ is the ideal generated by $(\bigoplus_{i=1}^{n_k - 1}F_i R)t^i$. The statement about the action follows since the $\GG_m$-action simply scales any graded component of its coordinate ring with weight $-n_k$.
\end{proof}
\begin{example}
    If the filtration $F_\bullet R$ is the slope filtration from Remark \ref{rem:slope}, then the subscheme of least weight recovers the subscheme associated to the ideal $\mathscr{I}\subset \shO_B$, in the notation of Remark \ref{rem:slope}.
\end{example}
By a \emph{generic} hypersurface or complete intersection, we mean one which is contained in a dense open set of the corresponding Hilbert scheme.
\begin{lemma}\label{lem:stableCI}
    Let the dimension of the subscheme $Y_γ$ be greater than or equal to $u$. Then a generic complete intersections of codimension $u$ on $Y$ degenerates to a complete intersection on the central fibre. Moreover, if $φ$ is a generic section of $H^0(Y,L^s)$, then the limit of $φ$ has weight $-n_s$ in the $\GG_m$-representation on $H^0(\Y_0,\scrL_0)$.
\end{lemma}
\begin{proof}
    Let $Z$ be a complete intersection in $Y$ of codimension no larger than $u$. We can identify not just $Y_γ$, but $Z\cap Y_γ$, which is generically a proper intersection, with its limit in the central fibre of $\Y$. The locus $\shV$ in the Hilbert scheme of complete intersections of the same topological type as $Z$, whose the intersection with $Y_γ$ is not complete intersection, is determined by any finite set of generators of the ideal of $Y_γ$ in $Y$. By the assumption on the codimension of $Z$, the locus $\shV$ is a proper closed subset. Hence the locus where the limit is not a complete intersection is also a proper closed subset. The second claim follows from the definition of the $\GG_m$-action.
\end{proof}
A nontrivial example where the above results can be applied is given in the following section.


\section{Complete intersections in flag varieties} 
\label{sec:subvarieties_in_flag_varieties}
In this section we apply the results of Section \ref{sec:koszul} to flag bundles. Fix a smooth projective variety $(B,L)$, a line bundle $A$ on $B$ and a flag bundle $Y=\Flag_r(E)$ with an ample underlying vector bundle $E$ of rank $r_E$. Let $Y$ be polarised by its relative canonical bundle $\shL_σ$. Fix a subsheaf $\shF\subset E$ of rank $r_F$ and let $(\Y,\scrL_λ(L))$ be the test configuration of $(Y,\shL_λ(L))$ induced by the degeneration of the vector bundle $E$ into a direct sum $\shF\oplus E/\shF$ defined in Section~\ref{sec:test}. We also denote $q = \rank E/\shF$.

\begin{lemma}\label{lem:LWSS}
    The relative dimension of the least weight subscheme in the central fibre $\Flag_r(\extE)_0$ is given by
    \begin{equation}\label{eq:lws}
        N_{r,{r_E},{r_F}} = \sum_{i=1}^{p-2} r_i(r_{i+1}-r_i) + r_{p-1}(q-r_{p-1}) + \sum_{i=p}^c (r_i-q)(r_{i+1}-r_{i}),
    \end{equation}
    where $r = (0,r_1,\ldots,r_c,{r_E})$ and
    \begin{equation}\label{eq:pdef}
        p = \min \{a\colon e \geq a \geq 1, r_a > {r_E}-f  \}
    \end{equation}
\end{lemma}
\begin{proof}
    We will describe the filtration corresponding to the test configuration $\Flag_r(\extE)$ in detail in Section~\ref{sec:SW}. However, it suffices to see that the subscheme fixed by the $\GG_m$-action on the central fibre is the intersection of $\Flag_r(\extE)$ with the subscheme
    \begin{equation}\label{eq:SSLW}
        \prod_{i=1}^{p-1}\PP(\ex^{r_i} E/\shF)\times\prod_{j=p}^c\PP(\ex^q E/\shF\ot \ex^{r_j-q}E)\subset \prod_{k=1}^c\PP(\ex^{r_k} E)
    \end{equation}
    The dimension of the locus of $k$-planes containing a fixed $q$-dimensional vector space in a Grassmannian of $k$-planes in an $l$-dimensional vector space is $(k-q)(l-k)$. The dimension in Equation \eqref{eq:lws} is then calculated by considering the flag bundle as an iterated fibration of Grassmannians and using elementary geometric considerations.
\end{proof}

\begin{lemma}\label{lem:subweight}
    Let $λ$ be an element of $\shP(r)$. The lowest weight $γ$ of the $\GG_m$-action on sections of $\shL_λ$ is given by
    \begin{equation}
        γ=\sum_{i=p}^cs_i\max\{(r_i-q), 0\},
    \end{equation}
    where $s_{c-i} = λ_{i}-λ_{i-1}$ for $i\in r$ and $p$ was defined in Equation~\eqref{eq:pdef}.
\end{lemma}
\begin{proof}
    Recall that the bundle $\shL_λ$ is the restriction of the line bundle $\bigotimes_{i=1}^c\shO_{\PP(\ex^{r_i} E)}(s_i)$. By Borel-Weil (cf. Equation \ref{eqbbw}) the sections of lowest weight over the central fibre of $\Y$ are sections of
    \begin{equation}
        \bigotimes_{i=1}^{p-1} S^{s_i}(\ex^{r_i} E/\shF)\ot\bigotimes_{j=p}^cS^{s_j}\left(\ex^q E/\shF\ot \ex^{r_j-q}\shF\right).
    \end{equation}
    The statement of the Lemma follows by the definition of the action, which scales fibres of $\shF$ by weight 1 and fixes the complement $E/\shF$.
\end{proof}

For any tuple of sections
\begin{equation}
\underline{φ} = (φ_1,\ldots,φ_u)\in \prod_{i=1}^q|s_i\shL_λ(A)|
\end{equation}
we write
\begin{equation}
    X_{\underline{φ}} = Z(φ_1)\cap\ldots\cap Z(φ_u)
\end{equation}
for their intersection. Let $\X$ be the Zariski closure of the orbit of $X$ under the $\GG_m$-action inside $\Y$. Let $\shF$ be a torsion free, saturated coherent subsheaf of $E$ and assume that the sections $φ_i$ are generic and that $u< N_{r,{r_E},{r_F}}$. We are now in the situation of Lemma \ref{lem:stableCI} and hence of Proposition \ref{prop:cidf} with the weight $γ$ given by Lemma \ref{lem:subweight}. We take the polarisation on $X_{\underline{φ}}$ to be the restriction $\shL_λ(A)$.

We now revert to the notation of Sections \ref{sec:curve} and \ref{sec:anybase}, where $b_0,b_1,a_0$ and $a_1$ are the coefficients of the two highest degree terms of polynomials $χ(\Flag_r(E),\shL_λ(A)^k)/\rank E^{kλ}$ and $χ(B,\widetilde{E}^{kλ}\ot A^k)/\rank E^{kλ}$, respectively. Recall that sections of $E^{kλ}$ correspond to sections of $\shL_λ(A)^k$ and the highest order terms of the polynomial $χ(B,\widetilde{E}^{kλ}\ot A^k)$ and the weight polynomial $w(k)$ of $(\Y,\scrL_λ(A))$ agree.
\begin{proposition}\label{prop:deltaPos}
    Let $σ$ be the canonical partition $σ_{{r_F},r}$ (cf. Definition \ref{def:canonical}). The difference
    \begin{equation}
        Δ=\DF(\Y)-\DF(\X)
    \end{equation}
    is positive for the polarisation $\shL_σ$ if the base $B$ is a curve. If the dimension $\dim_\CC B$ is arbitrary, then $Δ$ is positive when the polarisation is taken to be $\shL_σ(L^m)$ on $\Flag_r(E)$ for $m\gg0$.
\end{proposition}
\begin{remark}
    If $B$ is a curve, $E$ is ample and semistable, then the complete intersection $X_{\underline{φ}}$ polarised by the restriction of the bundle $\shL_σ$ is not destabilised by test configurations induced from extensions of $E$.

     If $B$ is an arbitrary polarised manifold, the same statement is true for complete intersections of sections of $\shL_σ(L^m)^{\ot s_i}$, $1<i<u$, for $m\gg 0$. It would be more interesting, although much harder, to study the asymptotics of test configurations of a fixed complete intersection as $m$ goes to infinity.
\end{remark}
\begin{proof}[Proof of Proposition \ref{prop:deltaPos}]
    Indeed we have
    \begin{equation}
            \frac{b_0}{a_0}-γ\geq 0
    \end{equation}
    with equality only in the case of the action scaling every section with the same weight. The above inequality is equivalent to
    \begin{equation}
        \lim_{k\to\infty} \frac{w_Y(k)}{kh_Y^0(k)} - γ\geq 0
    \end{equation}
    where $h_Y^0(k)$ is Hilbert polynomial of $\shL_σ(A)$ and $w(k)$ its equivariant analogue. Write
    \begin{equation}
        w(k) = \sum_i i \dim V_i^{(k)},
    \end{equation}
    where $V_i^{(k)}$ is the $i$th weight subspace of the representation of $\GG_m$ on $H^0(B,E^{kσ}\otimes A^k)$. By definition of $γ$, we have
    \begin{equation}
        w_Y(k) \geq \sum_i γ\dim V_i^{(k)} = γ k h_Y^0(k).
    \end{equation}
    It suffices to show that we have the inequality
    \begin{equation}\label{eq:ineqclaim}
    \frac{n(n+1)}{2}\geq μ_Y
    \end{equation}
    We have
    \begin{equation}
        μ_Y=μ_{\mathbf{f}}+μ_{\mathrm{rel}},
    \end{equation}
    where $μ_\mathbf{f}$ is the \emph{slope} of a fibre defined by
    \begin{equation}
        \rank E^{kσ} = D_{σ,r}\left(k^N_{{r_E},r} + μ_\mathbf{f} k^{N_{{r_E},r}-1} + O(k^{N_{{r_E},r}-2})\right),
    \end{equation}
    for some rational number $D_{σ,r}$ and $μ_{\mathrm{rel}} = \frac{a_1}{a_0}$.
    By the choice of polarisation we have $μ_\mathbf{f}=\frac{N_{{r_E},r}}{2}$. The other term $μ_{\mathrm{rel}}$ is obtained from Riemann-Roch. In the case $\dim B = 1$ the inequality \eqref{eq:ineqclaim} is clear. Consider the line bundle $\shL_σ(L^m)$. Then by Equation~\eqref{eq:atr} we have
    \begin{equation}
        μ_\mathrm{rel} = -\frac{b\deg K_B}{2\deg L} m^{-1} + O(m^{-2}),
    \end{equation}
    so there is an $m_0>0$ such that the inequality \eqref{eq:ineqclaim} holds for $m>m_0$.
\end{proof}

In light of Proposition \ref{prop:deltaPos}, we suspect that one has to start with an unstable vector bundle $E$ in order to find K-unstable examples of complete intersections for some choices of the parameters $E,F,B$ and $s_i$. We conclude with the proof of Theorem \ref{thm:simplelimit} and explain how Theorem \ref{thm:grEx} follows from Theorem \ref{thm:simplelimit}.
\begin{proof}[Proof of Theorem \ref{thm:simplelimit}]
    By Lemma \ref{lem:stableCI}, Lemma \ref{lem:LWSS} and Lemma \ref{lem:subweight} we are in situation of Proposition \ref{prop:cidf}, so the rest of the proof reduces to a straightforward calculation.
    Recall from Chapter \ref{chap:flags} that we have
    \begin{align*}
        b_0 &= \frac{h_2(λ){r_F}({r_E}μ_E+μ_F)}{{r_E}({r_E}+1)} + \frac{c_2(λ){r_F}({r_E}μ_E-μ_F)}{{r_E}({r_E}-1)},\\
        b_1 &= H_λ A_2(\widetilde{E}) + c_1(λ)\left(μ_E + \frac{{r_F}}{{r_E}}(1-g)\right),\\
        a_0 &= c_1(λ)μ_E, \\
        \intertext{and}
        a_1 &= 1-g,
    \end{align*}
    where $c_i(λ)$ denotes the ith elementary symmetric polynomial of $λ=(λ_1,\dotsc,λ_c)$. After some algebraic manipulation we can write $\DF(\X) = Cg + O(g^0)$ where
    \begin{align}\label{eq:futCI}
         C&= D\Big(h_2(λ)(n+1)(n-u){r_F}{r_E}^2(μ_E-μ_F) - γuc_1(λ)μ_E \\
                 & - c_1(λ)^2 {r_E}({r_E}+1){r_F} \left((n^2+n-nu-r_Eu)μ_E-(n+1)(n-u)μ_F)\right)\Big)
    \end{align}
    where $n=N_{{r_E},r}+1$ and $D= \big({r_E}^2({r_E}^2-1)n(n+1-u)\big)^{-1}$. Alternatively we can write
    \begin{equation}
        \DF(\X) = D\left(C_E \deg E + C_F \deg F\right)g + O(g^0),
    \end{equation}
    where we have denoted
    \begin{align}\label{eq:CECF}
        C_E &= ({r_E}^2-1)uc_1(λ)({r_F}c_1(λ)-{r_E}γ)-\frac{{r_F}}{{r_E}}C_F \\
        \intertext{and}
        C_F &= (N_{{r_E},r}+1)(N_{{r_E},r}-u)\left(({r_E}+1)c_1(λ)^2 - 2{r_E}h_2(λ)\right).
    \end{align}
    \end{proof}
    \begin{proof}[Proof of Theorem \ref{thm:grEx}]
        In the situation of Theorem \ref{thm:grEx} we have $\deg E = \deg F$. Computing the sign of the sum $C_E+C_F$  amounts to solving a polynomial inequality in $e,λ,f$ and $u$. Let $p$ be an integer between $1$ and $e-1$. Since we are assuming $e-f\geq p$, we also have $γ=0$ by Lemma \ref{lem:subweight}. Then there exist positive constants $D'$ and $D''$ such that
    \begin{equation}
        \begin{split}
            D'(C_E+C_F) &= D''(u-1)  \\
            & -({r_E}-{r_F})({r_E}-p-1)({r_E}-p)({r_E}-p+1)(p-1)p(p+1)\\
            &- {r_E} ({r_E} - 1) ({r_E} + 1) ({r_E} - {r_F} - p) p.
        \end{split}
    \end{equation}
    Hence assuming $u=1$ implies immediately that $C_E+C_F<0$ so the test configuration induced from $(\Y,\scrL)$ as described on page \pageref{thm:simplelimit}. The code for repeating the calculations and for simulating more examples is contained in \cite[Futaki invariants of complete intersections]{codepage}.
    \end{proof}

\begin{remark}
    While the inequality $C_E+C_F<0$ seems to hold more generally we only know how to prove it in the Grassmannian case.
\end{remark}
\begin{example}[Projective bundles]
    Equation~\eqref{eq:futCI} gets a very nice form for projective bundles. In the notation used in the proof of Theorem \ref{thm:simplelimit}, letting $λ=(1)$ gives
    \begin{equation}
        \DF(\X) = \Big(\frac{({r_F}-γu)\deg E- ({r_E}-u)\deg F}{{r_E}^2({r_E}+1-u)}\Big)g + O(g^0).
    \end{equation}
    This is the example given by Stoppa-Tenni \cite{Stoppa2010}. Note that the convention the authors use for $\PP E$ is dual to ours.
\end{example}




\chapter{Filtrations and relative K-stability} 
\label{chap:filtrations}
The K-stability of a projective variety with the structure of a projective family over a base scheme is in certain cases conjecturally characterised in terms of two types of simple test configurations. On the one hand one can look at test configurations which are equivariant with respect to the projection to the base, and on the other hand one can pull back test configurations from the base. Partial results are known in the case of toric bundles \cite{apostolov2011extremal}, projective bundles \cite{RossThomas}, blowups \cite{arezzo2006blowing,stoppa2007unstable,RossThomas} and flag bundles (Chapters \ref{chap:flags} and \ref{chap:UTF}). We define the notion of \emph{relative K-stability}, which is a conjectural refinement of K-stability, defined in Chapter \ref{chap:stability}. Given a projective morphism $p\colon Y\rightarrow B$ a \emph{relative test configuration} is a projective morphism $\Y\rightarrow B\times\AA^1$, with a $\GG_m$-action inducing a test configuration on each fibre of $p$.
%
%
%
%

We introduce and study filtrations of graded coherent sheaves of algebras in Section \ref{sec:filtrations_of_sheaves} with the aim of generalising the Witt-Nyström-Székelyhidi theory of filtrations in the study of K-stability \cite{witt2012test,szekelyhidi2011filtrations} to the context of relative K-stability. We show how this relates to Székelyhidi's notion of \kbar-stability (see Remark \ref{rem:ringfilt}) in Section~\ref{sec:filtrations_and_k_stability}. The motivation for studying filtrations of sheaves is that it allows us to give a unified treatment of several constructions that have appeared in the theory of K-stability, as well as constructions which we believe to be new. Related work was done by Ross and Thomas \cite{Ross2007}.

In Section~\ref{sec:operations_on_relative_test_configurations}, we propose an algebraic solution to the problem of interpolating test configurations, which was solved analytically in \cite{ross2014analytic}. This is an application of the constructions defined in Section \ref{sec:filtrations_of_sheaves} and Section \ref{sec:filtrations_and_k_stability}. Our approach works when the test configurations are defined for different polarisations as well. As an application, we prove that the K-unstable locus in $\VV(X)$ is open in the Euclidean topology. The behaviour of convex transforms as well as further examples of the interpolation construction are studied in Section~\ref{sec:convex_transform_okounkov_bodies_and_examples}.

In Section \ref{sec:pullback_test_configurations}, we apply the constructions to give a natural definition of pulling back test configurations from the base scheme $B$. We also give an overview where test configurations of this type have appeared in the literature. Finally, we discuss natural filtrations of the coordinate algebras of flag bundles from the new point of view in Section \ref{sec:SW}.

\begin{remark}[A note on terminology]
    Throughout this chapter the word \emph{relative} refers to working over a base scheme, not to be confused with the stability notion used in the extremal YTD correspondence.
\end{remark}

\begin{remark}
    As far as we know, apart from Theorem \ref{thm:SCor} and Proposition \ref{prop:1-1SCor} (Theorem \ref{summaryE}), the content of this Chapter is new even when working over  $\Spec \CC$.
\end{remark}

\section{Filtrations and projective families} 
\label{sec:filtrations_of_sheaves}

By convention, our algebras are $\ZZ^n_{\geq 0}$-graded. Let $B$ be a scheme over the complex numbers. If $\shA$ is a graded sheaf of $\shO_B$-algebras, we assume that $\shA_0 = \shO_B$.

\begin{definition}[Admissible filtrations]\label{def:relativefilt}
    Let
    \begin{equation}
        \shA = \bigoplus_{k=0}^\infty \shA_k
    \end{equation}
    be a sheaf of quasicoherent graded $\shO_B$-algebras over a scheme $B$. Then an \emph{admissible filtration} of $\shA$ is a filtration of coherent subsheaves
    \begin{equation}
        F_\bullet \colon 0=F_{-1}\shA\subset\shO_B=F_0\shA\subset F_1\shA\subset\dotsb\subset \shA,\quad
    \end{equation}
    such that it is
    \begin{enumerate}[(i)]
        \item \emph{multiplicative}, the filtration satisfies the relation $\left(F_i\shA\right) \left(F_j\shA\right)\subset F_{i+j}\shA$,
        \item \emph{homogeneous}, if $U$ is an open set in $B$, the homogeneous parts of any section of $F_i\shA(U)$ are all in $F_i\shA(U)$, and
        \item \emph{exhaustive}, it satisfies $\bigcup_{i=0}^\infty F_i\shA = \shA$.
    \end{enumerate}
\end{definition}

\begin{remark}
    The property $F_0\shA = \shO_B$ can be replaced by saying that a filtration
    \begin{equation}
        \dotsb \subset F_i \shA \subset F_{i+1}\shA\subset\dotsb
    \end{equation}
    is \emph{discrete}, meaning that $F_j\shA = \shO_B$ for some $j$. Any such filtration can be uniquely reindexed as an admissible filtration.

    There is another equivalent convention for defining an admissible filtration by reversing the order of the filtration. Codogni and Dervan described the process of translating between the two points of view in \cite{codogni2015non} in the nonrelative case. We work with increasing filtration as a matter of convenience while developing the theory.
\end{remark}

\begin{definition}\label{def:falg}
    Let $\fAlg$ denote the category of pairs $(\shA,F_\bullet\shA)$ such that
    \begin{enumerate}[(i)]
        \item $\shA$ is a graded coherent $\shO_B$-algebra, which is locally finitely generated over $\shO_B$ and
        \item $F_\bullet\shA$ is an admissible filtration.
    \end{enumerate}
    The morphisms are grading and filtration preserving homomorphisms. We refer to the objects admissibly filtered graded $\shO_B$-algebras and often simply refer to them by the symbol $F_\bullet \shA$.
\end{definition}

\begin{definition}\label{def:image}
    Let $f\colon\shA\rightarrow\shB$ be an surjection of graded $\shO_B$-modules and $f_i$ is the restriction of $f$ to the subsheaf $F_i\shA$. We define the \emph{image filtration} $\left(f_*F\right)_\bullet\shB$ by
    \begin{equation}
        (f_*F)_i\shB = \im f_i.
    \end{equation}
\end{definition}

\begin{definition}\label{def:induced}
    Let $g:\shA\rightarrow\shB$ be a morphism of graded filtered $O_B$-algebras and let $G_\bullet\shB$ be a filtration of $\shB$. We define the \emph{induced} filtration $\left(f^*G\right)_\bullet\shA$ by
    \begin{equation}
        \left(f^*G\right)_i\shA = \shA\cap G_i\shB =  \{ a\in \shA : f(a)\in G_i\shB \}.
    \end{equation}
\end{definition}
\begin{remark}
    If $f$ is an isomorphism, these two constructions are clearly inverse to one another, that is we have identities
    \begin{equation}
        f_*f^*G_\bullet\shA=G_\bullet\shA
    \end{equation}
    and
    \begin{equation}
        f^*f_*F_\bullet\shA=F_\bullet\shA.
    \end{equation}
\end{remark}
\begin{definition}\label{def:derived}
    Let $\shE$ be a sheaf of $\shO_B$-modules and let $H_i\shA\in\fAlg$. We define the \emph{derived filtration} \cite{bourbaki1972commutative}, also denoted by $H_\bullet\shE$, by
    \begin{equation}
        H_i \shE  = (H_i\shA)\shE.
    \end{equation}
\end{definition}

\begin{lemma}\label{lem:IF}
    Let $f\colon \shA\rightarrow \shB$ be a (grading-preserving) morphism of filtered graded sheaves of $\shO_B$-algebras. Then the image filtration and induced filtration, when defined, are admissible filtrations in the sense of Definition \ref{def:relativefilt}.
\end{lemma}
\begin{proof}
    We verify the conditions in Definition \ref{def:relativefilt} starting with the image filtration. Fix a filtered algebra $F_\bullet\shA\in\fAlg$. To show $(i)$, let $s_i$ and $s_j$ be sections of $f_*F_i\shA$ and $f_*F_j\shA$ over $U\subset B$, respectively. Then making $U$ smaller if necessary, we have elements $t_i$ and $t_j$ in $F_i\shA(U)$ and $F_j\shA(U)$, respectively, such that $f(t_i) = s_i$ and $f(t_j) = s_j$. The section $t_it_j$ is in $F_{i+j}\shA(U)$, so $f(t_it_j)$ is in $(f_*F)_{i+j}\shA(U)$. Homogeneity and exhaustivity follow easily since $f$ preserves the grading and is a surjective map of sheaves.

    The induced case is similar. To check multiplicativity, let $s_i\in g^*G_i \shB(U)$ and $s_j\in g^* G_j\shB(U)$. Since $G_\bullet\shB$ is admissible and $g$ is a homomorphism, we have $g(s_i s_j)\in G_{i+j}\shB(U)$ and hence $s_i s_i \in g^*G_{i+j}(U)$. Homogeneity and exhaustivity are again trivial, since the map $g$ preserves the grading.
\end{proof}
Tensor algebras of filtered modules are naturally endowed with an admissible filtration.
\begin{definition}[The tensor algebra of a filtered module]\label{def:FTA}
    Let
    \begin{equation}
        F_{\bullet} \shE: 0=F_0\shE \subset F_1\shE\subset\dotsb\subset F_n\shE = \shE
    \end{equation}
    be a filtered sheaf of $\shO_B$-modules. The tensor algebra of $\shE$ is naturally a filtered algebra by setting
    \begin{equation}\label{eq:tensorfilt}
        G_p T(\shE) =  \shO_B\oplus\bigoplus_{k=1}^\infty\bigoplus_{i_1+\dotsb+i_k=p} F_{i_1} \shE\ot\dotsb\ot F_{i_k}\shE
    \end{equation}
    for $p\in \ZZ_{>0}$.
\end{definition}
\begin{lemma}\label{lem:FTA}
    The filtration defined in Equation~\eqref{eq:tensorfilt} is admissible.
\end{lemma}
\begin{proof}
    Follows directly from the definitions.
\end{proof}
\begin{definition}[Tensor products of filtered algebras]\label{def:FTen}
    Let $F_\bullet\shA$ and $G_\bullet\shB$ be filtered sheaves of graded $\shO_B$-algebras. Define the tensor product
    \begin{equation}
        \left(F_\bullet \ot G_\bullet\right)_p \left(\shA\ot_{\shO_B}\shB\right) = \bigoplus_{i+j =p} F_{i} \shA\ot_{\shO_B} G_{j}\shB,
    \end{equation}
    which is a filtered $\ZZ^2$-graded sheaf of coherent $\shO_B$-algebras.
\end{definition}
\begin{lemma}\label{lem:TensAssoc}
    Tensor products of filtered algebras are commutative and associative.
\end{lemma}

\begin{definition}
    The \emph{Veronese subalgebra} $\shA^{(d)}$ is defined as the subalgebra
    \begin{equation}
        \shA_{(d)} = \bigoplus_{k=0}^\infty \shA_{dk}.
    \end{equation}
    Similarly, if $\shC$ is a $\ZZ^N_{\geq 0}$-graded sheaf of algebras, define the $a=(a_1,\dotsc,a_N)$-diagonal
    \begin{equation}
        \shC_{a} = \bigoplus_{k=0}^\infty \shC_{(ka_1,\dotsc,ka_n)}.
    \end{equation}
\end{definition}

\begin{definition}[Diagonal subalgebras]\label{lem:tensor}
    Let $F_\bullet\shA$ and $G_\bullet\shB$ be filtered sheaves of graded $\shO_B$-algebras. For any pair $(a,b)$ of nonnegative integers, we define the \emph{$(a,b)$-diagonal product} of the two filtered algebras by
    \begin{equation}
        \left(F_\bullet\ot_{(a,b)}G_\bullet\right) \left(\shA\ot\shB\right) =  \left(\shA\ot_{\shO_B}\shB\right)_{(a,b)} \cap \left(F_\bullet \ot G_\bullet\right)_\bullet \left(\shA\ot_{\shO_B}\shB\right).
    \end{equation}
    We refer to this filtration the $(a,b)$-\emph{diagonal product} of two filtered algebras. Define weighted diagonal products of any finite collections of filtered sheaves of algebras similarly.
\end{definition}
\begin{lemma}\label{lem:tensorAdm}
    The diagonal product is a well-defined operation on $\fAlg$.
\end{lemma}
\begin{proof}
    This is a straightforward verification.
\end{proof}

\begin{definition}[Filtrations generated at degree 1]\label{def:FGD1}
    Let $F_\bullet \shE$ be a filtered sheaf of $\shO_B$-modules and $\shA$ a graded sheaf of $\shO_B$-algebras such that $\shA_1 = \shE$. We say that the algebra $\shA$ is generated at degree 1 so that there is a surjective morphism
    \begin{equation}
        p\colon S(\shE)\rightarrow \shA.
    \end{equation}
    Let $F_\bullet S(\shE)$ be the filtration on $S(\shE)$ induced by the filtration on $T(\shE)$ defined in Definition \ref{def:FTA}. Define the \emph{filtration $G_\bullet \shA$ of $\shA$ generated by $F_\bullet \shE$} to be the image filtration $p_*F_\bullet \shA$.
\end{definition}
\begin{lemma}\label{lem:FGen}
    A filtration generated at degree 1 is admissible.
\end{lemma}
\begin{proof}
    Follows from Lemma \ref{lem:IF} and Lemma \ref{lem:FTA}.
\end{proof}

\begin{definition}
    We define the Rees algebra and the associated graded algebra of $F_\bullet \shA$ as
    \begin{enumerate}[(i)]
    \item  $\shRees(F_\bullet \shA) = \oplus_{i\geq 0}(F_i\shA)t^i \subset \shA[t],$
    \item  $\mathit{gr}(F_\bullet\shA) = \oplus_{i\geq 0} (F_{i} \shA)/(F_{i-1} \shA),$
    \end{enumerate}
    respectively. We say that a filtration $F_\bullet \shA$ is finitely generated if $\Rees(F_\bullet \shA)$ is locally finitely generated as an $\shO_B$-algebra. Note that both objects are bigraded. We refer to the two gradings by the $\shA$-grading and the $t$-grading.
\end{definition}

\begin{lemma}\label{lem:FFG}
    Let $f\colon\shA\rightarrow \shB$ be a morphism of graded sheaves of $O_B$-algebras. The tensor product preserves finite generation of admissible filtrations. If we assume the homomorphism $f$ is surjective, the same is true for the image filtration. Similarly, if the homomorphism $f$ is injective, the induced filtration is finitely generated.
\end{lemma}
\begin{proof}
    This can be easily seen by relating the Rees algebras. Let $F_\bullet$ and $G_\bullet$ be filtrations for $\shA$ and $\shB$, respectively, and $f\colon\shA\rightarrow\shB$ is a map preserving the grading. Then we have natural morphisms
    \begin{equation}
        \shRees(F_\bullet\shA)\rightarrow \shRees(f_*F_\bullet\shB)
    \end{equation}
    and
    \begin{equation}
        \shRees(f^*G_\bullet\shA)\rightarrow \shRees(G_\bullet\shB)
    \end{equation}
    which preserve the grading. The claims for pushforwards and pullbacks then follow easily. Note that we must assume that $f$ is a surjection in the pushforward case. In the tensor product case we have a natural isomorphism
    \begin{equation}
        \shRees \left(F_\bullet \shA\ot_{\shO_B} G_\bullet \shB\right) \cong \shRees(F_\bullet \shA)\ot_{\kk[t]}\shRees(G_\bullet \shB)\subset \left(A\ot B\right)[t]
    \end{equation}
    which immediately implies the claim.
\end{proof}

\begin{remark}[Filtrations of coordinate rings]
    Let $(B,L)$ be a projective scheme and denote $R=\bigoplus_{k=1}^\infty H^0(B,L^k)$. Definition \ref{def:relativefilt} contains the special case of admissible filtrations as defined \cite{szekelyhidi2011filtrations} in of $R$ by taking the base to be a point.
\end{remark}

\section{Relative K-stability} 
\label{sec:filtrations_and_k_stability}

In this section we define relative test configurations and describe their relationship to admissible filtrations discussed in Section~\ref{sec:filtrations_of_sheaves}.

 Fix a projective scheme $B$ of dimension $b$ with an ample line bundle $L$ and a locally finitely generated graded sheaf of $\shO_B$-algebras $\shA$. Denote the relative projectivisation of $\shA$ by $Y=\shP\mathit{roj}_B(\shA)$ with the projection $p\colon Y\rightarrow B$. We assume that $\shA$ is locally finitely generated at degree 1, which means that there exists a surjective homomorphism
\begin{equation}
    S(\shA_1) \rightarrow \shA
\end{equation}
and hence an embedding
\begin{equation}
    \shProj_B \shA \rightarrow \PP \shA_1.
\end{equation}

\begin{definition}
    Define the \emph{graded algebra of sections} of $L$ by
    \begin{equation}
        R_L=\bigoplus_{k=0}^\infty H^0(B,L^k)
    \end{equation}
    and the associated graded sheaf of algebras by
    \begin{equation}
        \shR_L = \bigoplus_{k=0}^\infty L^k.
    \end{equation}
\end{definition}

\begin{proposition}\label{prop:Reesflat}
    The Rees algebra of a graded sheaf of coherent $\shO_X$-algebras
    \begin{equation}
        \shRees(F_\bullet \shA) = \bigoplus_{k=0}^\infty F_k\shA t^k
    \end{equation}
    is a flat sheaf of graded $\shO_{\AA^1}$-algebras.
\end{proposition}
\begin{proof}
    The claim is local on $B$. The Rees algebra of a $k[t]$-module is torsion free as a $k[t]$-algebra. A well known flatness criterion states that a module over a principal ideal domain is flat if and only if it is torsion free \cite[Section 6.3]{eisenbud1995commutative}.
\end{proof}

We say that $\shA$ is \emph{ample} if the $\shO(d)$-line bundle on $Y$ defines an embedding for some positive integer $d$. If this is true for $d=1$, $\shA$ is \emph{very ample}.

\begin{definition}\label{def:RelTC}
    Let $Y$ be a scheme, $p\colon Y\rightarrow B$ a projective morphism and $\shL$ a $p$-ample line bundle. A \emph{relative test configuration}, or \emph{$p$-test configuration} $(\Y,\scrL,ρ)$ for the pair $(Y,\shL)$ is defined by
    \begin{itemize}
        \item a flat morphism $f\colon \Y\rightarrow \AA^1$ which factors through $B\times\AA^1$, along with an isomorphism $φ_t:f^{-1}\{1\}\cong Y$,
        \item an $f$-ample line bundle $\scrL$ on $\Y$ such that $\scrL_t$ such that the isomorphism over the fibre $f^{-1}\{1\}$ identifies the line bundles $\scrL_1$ and $\shL$.
        \item an algebraic action $\rho:\GG_m\times\Y\rightarrow\Y$ which makes the projection to $B\times\AA^1$ equivariant with respect to the trivial action on $B$ and the standard action on $\AA^1$, together with a $\scrL$-linearisation action on $\Y$ that covers the usual action on $\AA^1$.
    \end{itemize}
    The integer $r$ is called the \emph{exponent} of the $p$-test configuration. The fibre $f^{-1}\{0\}$ is called the central fibre.  If $\scrL$ is ample, a $p$-test configuration is a test configuration in the sense of Definition \ref{def:TCdef}, in which case we say that $\Y$ is an ample $p$-test configuration.
\end{definition}

\begin{theorem}\label{thm:SCor}
    A finitely generated admissible filtration $F_\bullet\shA$ determines a $p$-test configuration
    \begin{equation}
        \left(\shP\mathit{roj}_{B\times\AA^1}\shRees F_\bullet\shA,\shO(1)\right)
    \end{equation}
    with its natural $\GG_m$-action. Conversely, a $p$-relative test configuration $(\Y',\scrL)$ of $\shProj_B\shA$ determines a finitely generated admissible filtration $G_\bullet \shA$.
\end{theorem}
\begin{proof}
    Let the group $\GG_m$ act with its natural action on the line $\AA^1$ and extend it trivially to the product $B\times\AA^1$. There is a natural linearisation of this action on the sheaf $\shRees F_\bullet\shA$ with the following local description. Let $U$ be an open set in $B$ such that the projection $\restr{p}{U}$ corresponds to a graded $A_0$-algebra $A$, where $A_0$ is the coordinate ring of $B$ over $U$. The filtration $F_\bullet\shA$ restricts to an admissible filtration $F_\bullet A$. Then we have a commutative diagram
        \[
        \begin{tikzcd}[column sep=large]
        \operatorname{Rees} F_\bullet A \arrow{r}{t\mapsto s^{-1}t} & \operatorname{Rees} F_\bullet A[s^{\pm 1}] \\
        A_0[t] \arrow{u}{} \arrow{r}{t\mapsto s^{-1}t} & A_0[t,s^{\pm 1}]\arrow{u}
        \end{tikzcd}
        \]
    with obvious notation. This defines a $\GG_m$-linearisation on $\shA$ over $U$ compatible with the grading. The morphisms $p_U$  glue as $U$ ranges over an open cover of $B$ to determine a $\GG_m$-scheme $(\shP\mathit{roj}_{B\times\AA^1}\shRees F_\bullet \shA,\shO(1))$ with an equivariant projection down to $B\times\AA^1$. The projection to $\AA^1$ is flat by Proposition \ref{prop:Reesflat} and the central fibre is isomorphic to
    \begin{equation}
        \shProj_B \mathit{gr}(F_\bullet \shA)
    \end{equation}
    with a $\GG_m$-action defined by the $t$-grading.

    Given a $p$-test configuration $(\Y,\scrL)$, we produce an admissible filtration as follows. By replacing $\scrL$ with a power if necessary, we may assume that we have an embedding
    \begin{equation}
        ι\colon\Y\lra \PP g_*\scrL,
    \end{equation}
    where $g$ is the projection $\Y\rightarrow B\times\AA^1$. Using the identification $(\Y_1,\restr{\scrL}{B\times\{1\}}) \cong (Y,\shL)$ we obtain a natural map
    \begin{equation}
        h\colon\shA\lra \bigoplus_{k=0}^\infty g_*\left(\restr{\scrL}{B\times\{1\}}\right)^{\ot k},
    \end{equation}
    which we may take to be an isomorphism by \cite[Lemma 29.14.4]{stacks-project}.

    For any sufficiently small affine neighborhood $U\cong \Spec A_0\subset B$, we have a diagram
    \[
    \begin{tikzcd}
     g^{-1}U \cong \Proj_{\Spec A_0} S \arrow{r}{ι} \arrow{dr}{\restr{g}{U}} & \PP_{\Spec A_0} S_1  \arrow{d}\\
    & \Spec A_0
    \end{tikzcd}
    \]
    where $S$ is a graded $A_0$-algebra. Since the projection $g$ is equivariant for the trivial action on $U$, the linearisation of the $\GG$-action determines a representation on $A_1$. This determines a splitting $A_1 = \bigoplus_{i=1}^rW_i$ by weight. We obtain a presheaf of filtered $\shO_B$-modules as $U$ ranges over sufficiently small affine open sets of $B$. The associated sheaf generates an admissible filtration $G_\bullet\shA$ of $\shA$ by Lemma \ref{lem:FGen}.
\end{proof}
\begin{remark}\label{rem:slope0}
    If $B=\Spec \kk$, this theorem was proved by \cite{szekelyhidi2011filtrations}.

     If $X=Y=B$, $L$ is an ample line bundle on $X$ and $p$ is the identity morphism, this theorem reduces to the blowing up formalism due to Mumford \cite{mumford1977stability}, Ross and Thomas \cite{Ross2007} and Odaka \cite{odaka2009generalization}. Up to passing to a Veronese subalgebra, any finitely generated admissible filtration of the algebra $\shR_L$ can be obtained from a filtration
     \begin{equation}
        \mathcal{I}_1\subset\dotsb\subset \mathcal{I}_N\subset\shO_X.
     \end{equation}
     See Remark \ref{rem:slope} for an outline of this construction.
\end{remark}

Given an admissible filtration $F_i\shA$ we define the associated Hilbert, weight and trace squared functions by
\begin{align*}\label{eq:FPolys}
        h(k)&=\sum_{i=1}^{\infty}χ\left(B,\frac{F_i\shA_k}{F_{i-1}\shA_k}\right)\\
        w(k)&=\sum_{i=1}^{\infty}-iχ\left(B,\frac{F_i\shA_k}{F_{i-1}\shA_k}\right)\\
        \intertext{and}
        d(k)&=\sum_{i=1}^{\infty}i^2χ\left(B,\frac{F_i\shA_k}{F_{i-1}\shA_k}\right),
\end{align*}
respectively. If the $p$-test configuration given by Theorem \ref{thm:SCor} is ample, the functions $h(k)$, $w(k)$ and $d(k)$ are equal to the functions defined in Lemma \ref{lem:polys}. In this case the Donaldson-Futaki invariant is defined normally by Equation~\eqref{eq:futdef}.

\begin{definition}[Relative K-stability]
    Let $\test_B(Y,L)$ be the set of $p$-test configurations of $(Y,L)$. We define \emph{K-stability relative to $p$} in the same way we defined K-stability in Definition \ref{def:kstab} but by restricting the set of test configurations to ones which lie in $\test_B(Y,L)$.
\end{definition}

\begin{definition}\label{def:testclass}
    Consider the equivalence relation on the set of $p$-test configurations generated by the following three relations.
    \begin{enumerate}[(i)]
        \item Identify a $p$-test configuration $\Y$ with any test configuration with which it is $\GG_m$-equivariantly isomorphic.
        \item Identify any rescaling of the $\GG_m$-action on $(\Y,\scrL)$ (pullback by a cover of $\AA^1$, cf. Remark \ref{rem:scale}).
        \item Identify any pair $(\Y,\scrL)$ and $(\Y,\scrL^d)$ of $p$-test configurations for all $d>1$.
    \end{enumerate}
    Following Odaka \cite{odaka2015parametrization} we call equivalence classes under the above identifications \emph{$p$-test classes} for test configurations. Test configurations up to the first two relations are called \emph{$p$-test degenerations}. Note that we will use the same terminology for arbitrary filtrations later, see Definition \ref{def:testdeg}.
\end{definition}
\begin{proposition}[Theorem \ref{summaryE}]\label{prop:1-1SCor}
    The two constructions in Theorem \ref{thm:SCor} induce a 1-1 correspondence between finitely generated filtrations of $\shA$ up to isomorphism and Veronese subalgebras, and $p$-test classes of $(Y,\shL)$.
\end{proposition}
\begin{proof}
    It suffices to show that the two constructions are inverses to one another up to the stated identifications.

    An automorphism $φ$ of a filtered algebra $F_\bullet \shA$ induces an automorphism of the Rees algebra, and hence of its projectivisation. Conversely, any equivariant isomorphism which preserves linearisations clearly produces an automorphism of the filtered algebra.

    Similarly, the admissibility criterion uniquely fixes the scale of the action, while the final identification corresponds to identifying Veronese subalgebras of $F_\bullet \shA$. This completes the proof.
\end{proof}
We extend the notion of ampleness to admissible filtrations through ampleness of their \emph{finitely generated approximations}.
\begin{definition}[Ampleness for filtrations]
    Let $F_\bullet \shA$ be the filtered algebra and define the filtrations $F^{(k)}_\bullet \shA$ for all $k\in\NN$ to be the filtrations of $\shA_{(k)}$ generated by the filtration $F_\bullet \shA_k$. We say that an element of $\fAlg$ is \emph{ample} if the sequence of filtrations $F^{(k)}_\bullet \shA$ determine $p$-ample test configurations for all $k\in\NN$.
\end{definition}

\begin{definition}\label{def:relpol}
    For any line bundle $A$ on $B$, define the \emph{twisted polarisation}
    \begin{equation}\label{eq:notpol}
        \shL(A) = \shL\ot p^*A.
    \end{equation}
\end{definition}
We abuse notation by denoting the twisted polarisation on any test configuration of $Y$ similarly.
\begin{lemma}\label{lem:adiample}
    Let $(\Y,\scrL)$ be a $p$-test configuration for $(Y,L)$ and let $L$ be an ample line bundle on $B$. Then $(\Y,\scrL(L^m))$ is ample for $m\gg 0$.
\end{lemma}
\begin{proof}
    It suffices to check ampleness over the central fibre $B\times \{0\}$, over which the line bundle $\scrL(L^m)$ restricts to $\shF(L^m)$ for some relatively ample line bundle $\shF$ by construction. This is ample by \cite[Proposition II.7.10]{hartshorne1977algebraic}.
\end{proof}

We close the section on a brief discussion of slope stability which provides a case where amplitude has been studied in detail in Ross and Thomas \cite{RossThomas}.
\begin{remark}[Slope stability]\label{rem:slope}
    Let $ι:B'\subset B$ is a subscheme. We define a filtration of $\shR$ by vanishing orders along $B'$. Denote the ideal sheaf of $B'$ by $\mathscr{I}_{B'}$ and consider the filtration
\begin{equation}
    G_\bullet L^a\colon\mathscr{I}^{b}L^{a}\subset\mathscr{I}^{b-1}L^{a}\subset\dotsb\mathscr{I} L^{a}\subset L^a
\end{equation}
    for any pair of natural numbers $a$ and $b$. Assume from now on that $a$ and $b$ are coprime. The tensor algebra generated by $G_\bullet L^a$ (cf. Definition \ref{def:FGD1}) is admissibly filtered by Lemma \ref{lem:FTA}.

    For example, if $a = b = 1$ we write
    \begin{align*}
         \shO_B &\subset \mathscr{I} L\oplus    \mathscr{I}^2 L^2 \oplus \mathscr{I}^3 L^3      \oplus \mathscr{I}^4 L^4 \oplus\dotsb&\\
                    &\subset L \oplus \mathscr{I} L^2 \oplus \mathscr{I}^2 L^3      \oplus \mathscr{I}^3 L^4 \oplus\dotsb&\\
                    &\subset  L \oplus  L^2\oplus \mathscr{I} L^3      \oplus \mathscr{I}^2 L^4 \oplus\dotsb&\\
                    &\qquad\qquad\vdots\qquad\qquad \vdots\qquad\qquad\vdots &\\
                    &\subset \shR =  L \oplus  L^2\oplus   L^3   \oplus L^4 \oplus\dotsb&.
    \end{align*}
    It is easy to pick out the filtration from the increasing sequence of upper triangular subsets starting from the top left corner starting with
    \begin{equation}
        \shO_B\subset \left(O_B\oplus \mathscr{I}L\right)\subset \left(\shO_B\oplus L\oplus\mathscr{I}^2L^2\right)\subset \dotsb.
    \end{equation}
     We denote the associated $p$-test configuration by $\X_c$ for $c=\tfrac{a}{b}$. If we assume that $c\leq \mathrm{Sesh}(B',L)$, where
    \begin{equation}
        \mathrm{Sesh}(B',L) = \sup\left\{c\in\QQ_{>0} \colon L^r\ot \mathscr{I}_{B'}^{cr} \text{ is globally generated for } r \gg 0\right\},
    \end{equation}
    then the $p$-test configuration $\X_c$ is ample (up to an equivariant contraction in the case $c=\mathrm{Sesh}(B',L)$). This fact is due to Ross and Thomas, who also found a beautiful formula for the Donaldson-Futaki invariant in this case in terms of the \emph{slope} of the triple $(B',L,c)$\footnote{Proposition \ref{prop:flagbase} is proved using this formula.} \cite{RossThomas}.

    More complicated filtrations of the structure sheaf also yield admissible filtrations in a similar manner. Conversely, let $F_\bullet \shR_L$ be an admissible filtration which is generated in degree 1. Let $N$ be the smallest integer such that $F_NL = L$. For any $1\leq i\leq N$, we can define the ideal sheaf $\mathscr{I}_i\subset\shO_X$ to be the ideal sheaf locally generated by sections of the subsheaf $F_i L$. We obtain a filtration
    \begin{equation}
    0\subset\mathscr{I}_1\subset\dotsb\subset \mathscr{I}_N\subset \shO_X.
    \end{equation}
     An alternative construction of the ideal sheaves $\mathscr{I}_i$, starting with an arbitrary test configuration, can be found in Odaka \cite[Proposition 3.10]{odaka2009generalization} or Ross and Thomas \cite{Ross2007}.
\end{remark}


\section{Convex combinations of test configurations} 
\label{sec:operations_on_relative_test_configurations}
The aim of this section is to define a convex structure on equivalence classes of test configurations. The idea is very simple and is based on Segre products of filtered coordinate algebras. Consider the following example.
\begin{example}[A description of the convex combination of test configurations when the base $B$ is a point]\label{ex:PDiag}
    Let $V$ and $W$ be complex vector spaces and let $X$ be a projective variety with two embeddings $ι_1\colon X\subset \PP(V)$ and $ι_2\colon X\subset\PP(W)$. Fix two 1-parameter subgroups of $\SL(V)$ and $\SL(W)$, which determine actions
    \begin{align*}
        &α:\PP(V)\times\GG_m\rightarrow \PP(V)\\
        \intertext{and}
        &β:\PP(W)\times\GG_m\rightarrow \PP(W),
    \end{align*}
    respectively, and fix two positive integers $a$ and $b$. Then we have closed immersions
    \begin{equation}
        X\xrightarrow{Δ} X\times X\rightarrow \PP(S^aV\ot S^b W)
    \end{equation}
    and an associated family
    \begin{equation}\label{eq:PF}
        X\times\GG_m\xrightarrow{Δ} X\times X\times \GG_m \subset \PP(S^a V\ot S^b W)\times\GG_m.
    \end{equation}
    Here the $\GG_m$-action on $S^a V\ot S^b W$ is induced from $α\colon t\mapsto α_t$ and $β\colon t\mapsto β_t$ by setting
    \begin{equation}
        (α,β)_t (v_1\ot\dotsm \ot v_a \ot w_1\ot\dotsm \ot w_b) = (α_tv_1\ot\dotsm\ot α_tv_a \ot β_tw_1\ot\dotsm \otβ_tw_b).
    \end{equation}
    We define the \emph{weighted product test configuration} to be the Zariski closure of the image of the diagonal in Equation~\eqref{eq:PF}. This is clearly a test configuration for $(X,L_1^a\ot L_2^b)$, where $L_1$ and $L_2$ are the two restrictions of the hyperplane bundle under the embeddings $ι_1$ and $ι_2$, respectively.

    We write the resulting test configuration additively as
    \begin{equation}
        a[α]+ b[β],
    \end{equation}
    where the brackets denote taking the product test configuration associated to the $\GG_m$-action under the respective embeddings of $X$ into projective space. The test class determined by Equation~\eqref{eq:PF} (cf. Definition \ref{def:testclass} and Remark \ref{rem:scale}) can be written as
    \begin{equation}
        (1-t)[α] + t[β],
    \end{equation}
    where the parameter $t$ is taken to be $\frac{b}{a+b}$.
\end{example}

From now on, we identify the set of $p$-test configurations of $Y$ with the set of admissibly filtered algebras $F_\bullet \shA$ which satisfy $\shProj_B\shA\cong Y$ and whose filtration $F_\bullet\shA$ is finitely generated by Theorem \ref{thm:SCor}. This justifies the following definition, modelled after Odaka \cite{odaka2015parametrization}.
\begin{definition}[Test degenerations and test classes]\label{def:testdeg}
    Let $p\colon Y\rightarrow B$ be a projective morphism of normal schemes. Define the set of \emph{$p$-test degenerations of $Y$} to be the set $\test_B(Y)$ of admissibly filtered elements $F_\bullet\shA\in\fAlg$ such that $\shProj_B\shA\cong Y$ considered up to isomorphisms.

     Also define the set $\overline{\test_p(Y)}$ of $p$-test classes by additionally identifying Veronese subalgebras in $\test_p(Y)$. We have a natural map
    \begin{equation}
        \test_p(Y) \lra \overline{\test_p(Y)}.
    \end{equation}
    If we wish to fix a relatively ample line bundle $\shL$ on $Y$ (respectively, a ray of relatively ample line bundles), we write $\test_p(Y,\shL)$ (resp. $\overline{\test_p(Y,\shL)}$) for elements of $\test_p(Y)$ (resp. $\overline{\test_p(Y)}$) which define a test degenerations (resp. test classes) for $(Y,\shL)$.

    We denote $\test_{\Spec\kk}(B) = \test(B)$.
\end{definition}

We now state and prove Theorem \ref{summaryF}. Let $I_\QQ$ denote the unit interval $[0,1]\cap \QQ$ and let $Δ_{N-1}$ be the $N-1$ dimensional simplex in $\QQ^{N}$ defined by $t_1 + \ldots + t_{N} = 1$ and $t_i\geq 0$ for $i = 1,\dotsc,N$.
\begin{theorem}[Convex combinations of test configurations]\label{thm:convTC}
    For any $N\in\ZZ_{\geq 2}$, there exists a map
    \begin{equation}
        \Conv_N\colon\test_p (Y)^N\times Δ_{N-1}\lra \overline{\test_p(Y)}
    \end{equation}
    satisfying
    \begin{enumerate}[(i)]
        \item $\Conv_N\left(τ,e_i\right) = τ_i$, where $e_i$ is the $i$th unit vector and $τ=(τ_1,\ldots,τ_N)$ are $p$-test configurations of $(Y,\shL_i)$,
        \item $\Conv_N(τ,t)$ is an element of $\overline{\test_p\left(Y,\shL_t\right)}$, where $\shL_t$ is the line bundle $\sum_{i=1}^{N}t_i \shL_i$, and
        \item if we take $B=\Spec\CC$ and assume that $τ_i$ are finitely generated, the Donaldson-Futaki invariant of $\Conv_N(τ,t)$ is continuous in the second variable.
    \end{enumerate}
\end{theorem}
\begin{theorem}[Theorem \ref{summaryG}]\label{thm:unstable_locus_is_open}
    The K-unstable locus in $\VV(X)$ (cf. Equation \eqref{eq:Vdef}) is open in the Euclidean topology.
\end{theorem}
\begin{proof}
    Fix a basis $L_1,\ldots,L_N$ of the Picard group of $X$ and let $L$ be a K-unstable polarisation. Fix a test configuration $\X$ for $(X,L)$ with negative Donaldson-Futaki invariant. Let $t$ be a point in $I_\QQ^N$, $s=1-\sum_{i=1}^N t_i$ and let $U$ be a neighbourhood of 0 in $I_\QQ^N$ such that $(1-s)L + \sum_{i=1}^N t_iL_i$ is ample for all $t\in\U$.

     For any $t\in U$, define the test class $[\X_t] = (1-s)[\X] + \sum_{i=1}^N t_i[\X_i]$, where $\X_i$ are trivial test configurations for $(X,L_i)$. By Theorem \ref{thm:convTC}, there is an open neighbourhood $V\subset U$ of $0$ such that $\DF(\X_t)$ is negative for all $t\in V$. The set $V$ determines an open neighbourhood of $L$ in $\Amp(X)$ of K-unstable polarisations. Since $L$ was an arbitrary K-unstable polarisation, this completes the proof.
\end{proof}
\begin{remark}\label{rem:irrextension}
    It makes sense to extend the definition of the Donaldson-Futaki invariant of a weighted product $(1-t)τ_1+tτ_2$ for irrational values of $t$ by continuity.
\end{remark}

For simplicity of exposition we restrict to the case a pairwise convex combination. The proof of the general case of Theorem \ref{thm:convTC} follows the same argument with minor adjustments which are outlined in Remark \ref{rem:diags} and Remark \ref{rem:aspects}.

Recall first a basic algebraic fact.
\begin{lemma}
    Let $f\colon S\rightarrow T$ be homomorphism of commutative rings and let $A$ and $B$ be $T$-algebras. Let $A_S$ and $B_S$ be the $S$-algebras determined by the map $f$. Then there is a natural surjective homomorphism
    \begin{equation}\label{eq:alghom}
        g\colon A_S\ot_S B_S\rightarrow A\ot_T B.
    \end{equation}
\end{lemma}
\begin{proof}
    The tensor product $A_S\ot_S B_S$ is a quotient of $A \ot_\ZZ B$ by the ideal generated by elements $f(s)a\ot b-a\ot f(s)b$ for $s\in S$, $a\in A$ and $b\in B$. This ideal is contained in the ideal of $A\ot_T B$ in $A\ot_\ZZ B$, hence identifying both algebras in Equation~\eqref{eq:alghom} as quotients of $A\ot_\ZZ B$ yields the claim.
\end{proof}
\begin{lemma}[{\cite[Example 1.2.22]{Lazarsfeld2004}}]\label{lem:SMM}
    Let $L_1$ and $L_2$ be ample line bundles on a projective scheme $X$. Then the natural map
    \begin{equation}
        H^0(X,L_1^{a})\ot_\kk H^0(X,L_2^b)\lra H^0(X,L_1^a \ot L_2^b)
    \end{equation}
    is surjective for $a,b\gg 0$.
\end{lemma}
\begin{corollary}\label{cor:RSMM}
    Let $\shL_1$ and $\shL_2$ be $p$-ample line bundles on $Y$. Then the natural map
    \begin{equation}\label{eq:RSMM}
        p_*\shL_1^a \ot_{\shO_B} p_*\shL_2^b \lra p_* \left(\shL_1^a \ot_{\shO_Y} \shL_2^b\right)
    \end{equation}
    is surjective for $a,b\gg 0$.
\end{corollary}
\begin{proof}
    By \cite[Corollary 12.9]{hartshorne1977algebraic} we may assume that the pushforwards $p_*\shL_1^a$, $p_* \shL_2^b$ and $p_* (\shL_1^a\ot\shL_2^b)$ are vector bundles on $B$. It suffices to check that the map in Equation~\eqref{eq:RSMM} is surjective on fibres, which follows from \ref{lem:SMM}.
\end{proof}

Let $(a,b)$ be a pair of nonnegative integers and $F_\bullet \shA$ and $G_\bullet\shB$ two elements of $\fAlg$ with chosen isomorphisms
\begin{equation}\label{eq:isos}
     \shProj_B\shA \cong Y \medspace \text{ and } \medspace \shProj_B\shB\cong Y.\\
\end{equation}
Write $\shR_\shA$ and $\shR_\shB$ for the graded algebras associated to the two Serre line bundles. We have natural morphisms
\begin{equation}\label{eq:surmaps}
    \shA \ra p_*\shR_\shA  \medspace \text{ and } \medspace \shB \ra p_*\shR_\shB.
\end{equation}
By \cite[Lemma 29.14.4]{stacks-project}, there exists a $k_0>0$ such that the maps in Equation~\eqref{eq:surmaps} are isomorphisms in degrees larger than $k_0$. Therefore the map
\begin{equation}
    φ:\shA\ot_{\shO_B} \shB\lra p_*\shR_\shA \ot_{\shO_B} p_*\shR_\shB
\end{equation}
is an isomorphism in degrees larger than $k_0$. Using the isomorphisms in Equation~\eqref{eq:isos} and Corollary \ref{cor:RSMM}, we obtain a surjective morphism
\begin{equation}
    φ:\shA_{(a)}\ot_{\shO_B} \shB_{(b)}\lra p_*\left((\shR_\shA)_{(a)}\ot_{\shO_Y} (\shR_\shB)_{(b)}\right)
\end{equation}
for $a,b>k_0$.

We will from now on use a mix of additive and multiplicative notation for both test degenerations and line bundles.
\begin{definition}
    For any nonnegative integers $a$ and $b$ we define the \emph{weighted product} of two test degenerations
    \begin{equation}\label{eq:combination}
        a[F_\bullet\shA] + b[G_\bullet\shB]
    \end{equation}
    to be given by the filtration
    \begin{equation}
        φ_* (F_\bullet\ot_{(ma,mb)} G_\bullet)(\shA\ot_{\shO_B}\shB),
    \end{equation}
    where $φ_*$ denotes taking the image filtration defined in Definition \ref{def:image} and $m$ is chosen to be the smallest integer so that the statement of Corollary \ref{cor:RSMM} and surjectivity of Equation~\eqref{eq:surmaps} hold.
\end{definition}

\begin{theorem}\label{thm:conv}
    If $τ_1$ and $τ_2$ are $p$-test degenerations for the relatively ample line bundles $\shL_1$ and $\shL_2$, the diagonal product determines a $p$-test configuration for each polarisation on the line segment between $\shL_1$ and $\shL_2$ in the cone $\VV(Y)$ of polarisations (cf. Equation \eqref{eq:Vdef}).
\end{theorem}
\begin{proof}
    This follows from Lemma \ref{lem:tensor} and the fact that we have
    \begin{equation}
        \left(\shProj_B\bigoplus_{k=0}^\infty p_*\left(\shL_1^{ak}\ot \shL_2^{bk}\right),\shO(1)\right) \cong \left(Y,\shL_1^{ak}\ot\shL_2^{bk}\right).
    \end{equation}
\end{proof}

\begin{remark}[Diagonals in finite products of algebras]\label{rem:diags}
    Diagonal products make sense for products of three or more elements of $\fAlg$. First of all, Lemma \ref{lem:SMM} and Corollary \ref{cor:RSMM} generalise to finite products of line bundles of the form $L_1^{a_1}\ot\dotsb\ot L_N^{a_N}$ by an easy induction. This avoids the difficulty of having to make a choice of integer $m$ in the construction of the convex combinations of test configurations several times.

    In particular, if $F_\bullet\shA$, $G_\bullet\shB$ and $H_\bullet\shC$ are in $\fAlg$, the $(a,b,c)$ diagonal can be written as a product pairwise diagonals as
    \begin{equation}\label{eq:frels}
        \begin{split}
            F_\bullet\ot_{(a,b)}G_\bullet\ot_{(1,c)}H_\bullet   &= F_\bullet\ot_{(a,1)}\ot G_\bullet\ot_{(b,c)}H_\bullet\\
                                                                &=F_\bullet\ot_{(a,1)}\ot H_\bullet\ot_{(c,b)}G_\bullet,
        \end{split}
    \end{equation}
    where we omit writing the algebra $\shA\ot_{\shO_B}\shB\ot_{\shO_B}\shC$. The products are clearly associative so we have omitted the parentheses. Verifying Equation~\eqref{eq:frels} only needs to be done at the level of the diagonal subalgebras, since the filtration on diagonal is simply the restriction of the tensor product filtration. The two identities generate the natural associativity and commutativity properties of the pairwise diagonal product in $\fAlg$. The same relations descend to the weighted products in $\test_B(Y)$.
\end{remark}

\begin{remark}\label{rem:aspects}
    There are several potentially confusing aspects about the previous definitions. First, it makes sense to reparametrise the \emph{test class} represented by $aτ+bτ$ by rational numbers in the interval $I_\QQ$. However, convex combinations are \emph{not} well defined for test classes since the diagonal product is clearly not invariant under replacing one of the filtered algebras by a Veronese subalgebra

    Second, in order to define the filtration associated to the weighted product, we needed to assume that $a$ and $b$ were sufficiently large in order to make the multiplication maps in Lemma \ref{lem:SMM} and Corollary \ref{cor:RSMM} surjective. This can be circumvented by replacing both underlying line bundles by a common power at the outset.

    Third, while our construction gives no way of choosing a unique convex combination in $\test_B(Y)$, we see no need to do this. We are ultimately interested in test classes. By Remark \ref{rem:diags}, a convex combination of multiple elements of $\test_B(Y)$ can be done simultaneously and there is no need to iterate a pairwise construction. For test degenerations $τ=([F^1\shA^1],\dotsc,[F^N\shA^N])$ and rational numbers
    \begin{equation}
        t=(t_1,\ldots, t_{N})\in Δ_{N-1} \subset I_\QQ^N
    \end{equation}
    we define $\Conv_N(τ,t)$ to be the test class of the $(mt_1,\ldots,mt_N)$-diagonal in the filtered algebra
    \begin{equation}
        \bigotimes_{i=1}^N F_\bullet\shA^i,
    \end{equation}
    where $m$ is a sufficiently large and divisible integer.
\end{remark}

We summarise the contents of Remark \ref{rem:diags} and Remark \ref{rem:aspects} in the following proposition.
\begin{proposition}
    Given $N$ elements of $\test_B(Y)$, there is uniquely defined map from $I^N$ to the set of test classes of $Y$ relative to $p$. This map is naturally fibred over a subset of the set of rays of $p$-ample line bundles on $Y$.
\end{proposition}

Before proving property (iii) of Theorem \ref{thm:convTC} we state the following lemmas. Donaldson reduced the calculation of the total weight to an nonequivariant calculation. The weight calculation done in Chapter \ref{chap:flags} are based on this idea. See also \cite[Section 2.8.1]{ross2011weighted} for a clear exposition.
\begin{lemma}\label{lem:DPoly}
    Let $X_0$ be a projective $\GG_m$-scheme over the complex numbers with an ample $\GG_m$-linearised line bundle $L$. Then there exists a polarised scheme $(\shY,\shH_L)$ such that the the weight polynomial is given by
    \begin{equation}
        \tr H^0(X_0,L^k) = χ(\shY,\shH_L^k) - χ(X_0,L^k).
    \end{equation}
\end{lemma}
Dervan proved the following generalisation of Donaldson's formula.
\begin{lemma}[{\cite[Lemma 2.30 (iv)]{dervan2014uniform}}]\label{lem:WPoly}
    Keep the notation of Lemma \ref{lem:DPoly} and let $A$ be a $\GG_m$-linearised line bundle on $X_0$. The total weight of the $\GG_m$-representation on the vector space $H^0(X_0,L^k\ot A)$ is given by
    \begin{equation}
        \tr H^0(X_0,L^k\ot A) = \tr H^0(X_0,L^k) - \int_\shY \frac{c_1(\shH_L)^n\cdot c_1(\shH_A)}{n!} k^n + O(k^{n-1}),
    \end{equation}
    for some line bundle $\shH_A$ on $\shY$.
\end{lemma}

\begin{corollary}\label{cor:WWeight}
    Keep the notation of Lemma \ref{lem:DPoly} and let $L_i$ be ample $\GG_m$-linearised line bundles on $X_0$ for $1\leq i \leq N$. We have an identity
    \begin{equation}
        \tr H^0(X_0,\bigotimes_{i=1}^NL_i^{a_ik}) = C_0(a_1,\dotsc,a_N) k^{n+1} + C_1(a_1,\dotsc,a_N) k^n + O(k^{n-1}).
    \end{equation}
    where $C_0(a_1,\dotsc a_N)$ and $C_1(a_1,\dotsc,a_N)$ are polynomials in $a_1,\dotsc,a_N$.
\end{corollary}
\begin{proof}
    Apply Lemma \ref{lem:WPoly} and Lemma \ref{lem:DPoly} to
    \begin{equation}
         L = L_j^k \medspace \text{ and } \medspace A=\bigotimes_{i=1,i\neq j}^N L_j^{a_i k}
    \end{equation}
    for $j=1,\dotsc,N$.
\end{proof}
\begin{claim}
    Property (iii) of Theorem \ref{thm:convTC} holds.
\end{claim}
\begin{proof}
    We show that the Donaldson-Futaki invariant is a continuous rational function in $t$ for $t \in Δ_{N-1}$.

    By the Riemann-Roch formula, there exist polynomials $c_0$ and $c_1$ in $a_i$ such that
    \begin{equation}
        h^0(X, \bigotimes L_i^{a_ik}) = c_0 k^n + c_1 k^{n-1} + O(k^{n-2}).
    \end{equation}
    In particular, there exist positive numbers $c_{0,i}$ such that
    \begin{equation}
        c_0 = \sum_{i=1}^N c_{0,i} a_i^n + O(a_1^{n-1},\dotsc,a_N^{n-1}),
    \end{equation}
    since $L_i$ are all ample.

    By Corollary \ref{cor:WWeight}, the weight function is similarly a polynomial in the $a_i$. We conclude that the function
    \begin{equation}
        t\mapsto\DF\left(t_1τ_1+\dotsb+t_{N-1}τ_{N-1}+(1-\sum_{i=1}^{N-1}t_i)τ_N\right)
    \end{equation}
    is continuous rational function in $t\in Δ_{N-1}$, since the denominator is always positive.
\end{proof}

\begin{remark}
    There is an alternative way to see that the Donaldson-Futaki invariant is continuous which uses an intersection theoretic formula for the Donaldson-Futaki invariant \cite[Proposition 6]{li2011special} which holds for normal test configurations. Assume that $L_1$ and $L_2$ are ample line bundles on $X$ and $F_\bullet R_{L_1}$ and $G_\bullet R_{L_2}$ are admissible. The bigraded Proj
    \begin{equation}
        \mathscr{Z} = \shProj_{\AA^1} \shRees F_\bullet \left(R_{L_1}\ot_{\kk[t]}R_{L_2}\right)
    \end{equation}
    with the Serre line bundle $\shO(a,b)$ is a test configuration for the product $(X\times X,L_1^a\boxtimes L_2^b)$. Restricting $\mathscr{Z}$ to the diagonal yields a test configuration $\X_{a,b}$ for $(X,L_1^a\ot L_2^b)$. The filtration associated to $\X_{a,b}$ is equal to the filtration $\left(F_\bullet\ot_{(a,b)} G_\bullet\right)\left(R_{L_1}\ot R_{L_2}\right)$ so the two test configurations are $\GG_m$-equivariantly isomorphic.

 If we assume that $\mathscr{Z}$ is normal, the intersection theoretic formula for the Donaldson-Futaki invariant \cite[p. 225]{li2011special} implies that the Donaldson-Futaki invariant is continuous in $t$.

 The above argument generalises to weighted products of a finite collection of algebras.
\end{remark}

We give a very simple example of a family of test configurations on a fixed polarised variety.
\begin{example}[A combination of two simple test configurations on a ruled surface]\label{ex:PBundle}
    Let $F$ and $Q$ be very ample line bundles on a curve $C$ of genus $g$ and consider the projective bundle $\PP (F\oplus Q)$ with its $\shO(1)$-polarisation. Let $α$ and $β$ be the $\GG_m$-actions which scale $F$ and $Q$, respectively, with positive weight 1. The two $\GG_m$-actions $α$ and $β$ determine filtrations
    \begin{equation}
        F\subset F\oplus Q
    \end{equation}
    and
    \begin{equation}
        Q\subset F\oplus Q
    \end{equation}
    and corresponding test configuration $\Y_F$ and $\Y_Q$ for $(\PP(F\oplus Q),\shO(1))$. The associated filtrations are discussed in more detail and generality in Section~\ref{sec:SW}.

     For any natural numbers $a$ and $b$ we define a test configuration of $\PP(F\oplus Q)$ by inducing a $\GG_m$-action on $\PP\left(S^{a+b}(F\oplus Q)\right)$ and restricting to the image of $\PP(F\oplus Q)$ under Veronese embedding of $\PP(F\oplus Q)$. The filtration associated to this test configuration is generated by the grading on the vector bundle $S^{a+b} \left(F\oplus Q\right)$ given in Figure~\ref{fig:PBW}.

    \begin{figure}[htbp]
        \begin{center}
            \begin{tikzpicture}[scale=1]
            \draw[color=black,dashed,thin]      (0,0) node[left,font=\small] {$a+b$} node[below=2.9mm,font=\small] {$S^{(a+b)}F$} -- (1.3, 0) node[below=4.5mm,font=\small] {$\cdots$} -- (2.9,0) node[below=3mm,font=\small]   {$S^{b}F\ot S^{a}Q$} -- (4.1,0) node[below=5mm,font=\small] {$\ldots$} -- (5,0)     node[below=3mm,font=\small] {$S^{(a+b)}Q$};
            \draw[color=black,thin]  (0,3) node[below,left,font=\small] {$a+2b$} -- (3,0) -- (5,2) node[right,font=\small] {$2a+b$};
            \end{tikzpicture}
        \caption{The $t$-grading on the $\shO_{\PP^1}$-module $S^{a+b} \left(F\oplus Q\right)$.}\label{fig:PBW}
    \end{center}
    \end{figure}
    An elementary summation shows that the Donaldson-Futaki invariant of the test configuration $aτ_F+bτ_Q$ is given by
    \begin{equation}
        \begin{split}
            \DF(aτ_F+bτ_Q) = \frac{a^3}{(a+b)^3}\DF(\Y_F) + \frac{b^3}{(a+b)^3}\DF(\Y_Q) \\
            + \frac{a^2b (μ_F+1-g)+ab^2 (μ_G+1-g))}{2 μ_E^2 (a+b)^3}.
        \end{split}
    \end{equation}
    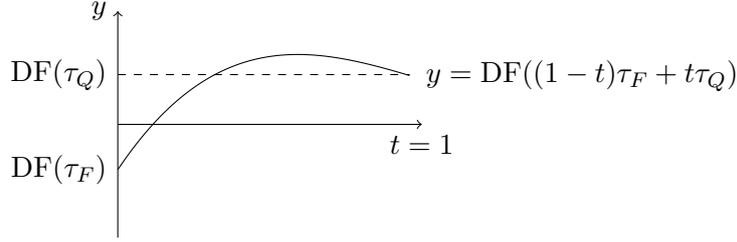
\begin{figure}[!htbp]
        \begin{center}
            \begin{tikzpicture}[yscale=3,xscale=2]
            \node at (0,-0.2) [left,font=\small] {$\DF(τ_F)$};
            \draw[->,thin] (0,0) -- (2, 0) node[below,font=\small] {$t=1$};
            \draw[dashed,thin] (0,0.22) node[left,font=\small]  {$\DF(τ_Q)$} -- (1.92, 0.22);
            \draw[->] (0,-0.5) -- (0, 0.5) node[left,font=\small] {$y$};
            \draw[smooth,color=black,domain=0:1.92,samples=100] plot (\x,{-0.2 + \x - 0.6*\x^2 + 0.1*\x^3}) node[right=0.7mm,font=\small] {$y = \DF((1-t)τ_F+tτ_Q)$};
        \end{tikzpicture}
        \end{center}
        \caption{The Donaldson-Futaki invariant of $(1-t)τ_F+tτ_Q$ plotted against $t=\frac{b}{a+b}$ when $μ_F=2$, $μ_Q=1$ and $g=2$ equals $\tfrac{1}{9}(-1 + 6t - 3t^2 - t^3)$.}\label{fig:FPlot}
    \end{figure}

For example, if $μ_F = 2$ and $μ_Q = 1$, we plot the Donaldson-Futaki invariant for different values of $a$ and $b$ in Figure~\ref{fig:FPlot}. The code for repeating the calculation be found in \cite[Ruled surface interpolations]{codepage}.
\end{example}


\section{Okounkov bodies and the convex transform of a filtrations} 
\label{sec:convex_transform_okounkov_bodies_and_examples}
In this section we describe the behaviour of the convex geometry associated to the variation of filtered linear series coming from the convex structure defined in Section~\ref{sec:operations_on_relative_test_configurations}. We give a brief review of Okounkov bodies and the convex transform associated to an admissible filtration. For more details, we refer to Lazarsfeld-Musta\c{t}\v{a} \cite{lazarsfeld2008convex}, Boucksom-Chen \cite{boucksom2011okounkov}, Witt-Nyström \cite{witt2012test} and Székelyhidi \cite{szekelyhidi2011filtrations}.

Let $X$ be a smooth complex projective variety and $L$ a line bundle on $X$ with ring of sections $R=\bigoplus_{k=0}^\infty H^0(X,L^k)$. Fix a base point $p\in X$ and holomorphic coordinates $z_1,\ldots, z_n$ centred around $p$. Given $f\in R_k$ we may write
\begin{equation}
    f = s z_1^{r_1}\dotsm z_n^{r_n},
\end{equation}
for some $(r_1,\ldots,r_n)\in \ZZ^n$, where $s$ is a holomorphic function on a neighbourhood of $p$ which does not vanish at $p$. We keep the base point and the choice of coordinates fixed throughout the section.

 We define a function $ν\colon R\rightarrow \QQ^n$ by setting
\begin{equation}
    ν(f) = \frac{(r_1,\ldots,r_n)}{k}
\end{equation}
for any such $f\in R_k$.
\begin{definition}
    Define the \emph{Okounkov body} of $L$ by $Δ(L) = \overline{ν(R)}\subset \RR^n$.
\end{definition}
It is well known that $Δ(L)$ is a convex set. Given an admissible filtration $F_\bullet R$, we define
\begin{equation}
    R^{\leq t} = \bigoplus_{k=0}^\infty F_{\lfloor tk\rfloor} R_k.
\end{equation}
This determines a closed convex subset $Δ(L)^{\leq t} = \overline{ν(R^{\leq t})}$.
\begin{definition}
    Define the \emph{convex transform} of $F_\bullet R$ to be
    \begin{equation}
        G(x) = \inf\{ t\colon x\in Δ(L)^{\leq t} \}.
    \end{equation}
\end{definition}
If $x$ is rational we have $G(x) = \inf\left\{\frac{\lev f}{\deg f}\colon ν(f) = x\right\}$. The extension to real numbers is obtained as the pointwise largest function which is lower semicontinuous and agrees with the restriction the subset $Δ(L)\cap \QQ^n$.

Suppose now that $L_1$ and $L_2$ are ample line bundles on $X$. Let $F^i_\bullet R_{L_i}$ be admissible filtrations for $i=1,2$ and let $G_i\colon Δ(L)\rightarrow \RR$ be the convex transforms of the two filtered algebras.

Let $a$ and $b$ be nonnegative integers such that there exists a surjective homomorphism
\begin{equation}
    ψ\colon S = \bigoplus_{k= 0}^\infty (R_{L_1})_{ak}\ot (R_{L_2})_{bk}\lra \bigoplus_{k=0}^\infty H^0(X,(aL_1+bL_2)^k).
\end{equation}
for all $k>0$. The ring $R_{aL_1+bL_2}$ is naturally filtered by the image of $(F^1_\bullet \ot_{(a,b) } F^2_\bullet) S$. The Okounkov body $Δ(aL_1+bL_2)$ is contained in the Minkowski sum $aΔ(L_1)+bΔ(L_2)$.

Set
\begin{equation}
    U = \left\{(x,v)\in\RR^{2n}\colon \frac{x}{2}+v\in aΔ(L_1), \frac{x}{2}-v\in bΔ(L_2) \right\}
\end{equation}
and define a real valued function $\widehat{H}:U\rightarrow \RR$ by setting
\begin{equation}
    \widehat{H}_{a,b}(x,v) = aG_1(\frac{x+2v}{2a}) + bG_2(\frac{x-2v}{2b}).
\end{equation}
\begin{theorem}\label{thm:convo}
    The convex transform $G_{a,b}(x)$ of the weighted product filtration $(F^1_\bullet \ot_{(a,b) } F^2_\bullet)\left(R_{L_1}\ot R_{L_2}\right)$ is equal to the minimiser
    \begin{equation}
        H_{a,b}(x) = \operatorname{min}_{v\in U} \widehat{H}_{a,b}(x,v)
    \end{equation}
    restricted to the Okounkov body $Δ(aL_1+bL_2)$.
\end{theorem}
\begin{proof}
    Let $G_{a,b}(x)$ be the convex transform of the filtration $(F_\bullet \ot_{(a,b)} G_\bullet) (R\ot S)$. We must show that $H_{a,b}(x)=G_{a,b}(x)$ for $x$ in
    \begin{equation}
        Δ(aL_1+bL_2)\subset aΔ(L_1)+bΔ(L_2)
    \end{equation}
 Let $x\in Δ(aL_1+bL_2)\cap\QQ^n$ and let $ν_i$ and $ν_{a,b}$ denote the convex transforms of $F^i_\bullet$ and $F^1_\bullet\ot_{(a,b)}F^2_\bullet$, respectively. We have
    \begin{equation}
        \begin{split}
            G_{a,b}(x) &= \inf \left\{\frac{\lev(f)}{k}\colon f\in \left(R_{aL_1+bL_2}\right)_k\text{ and } \frac{ν_{a,b}(f)}{k} = x\right\}    \\
                 &= \inf \left\{\frac{\lev(g) + \lev(h)}{k}\colon g\in R_{akL_1}, h\in  R_{bkL_2} \text{ and } (ψ\circ ν_{a,b})(g\ot h) = x\right\}\\
                 &\geq \inf \left\{ aG_1(ν_1(g)) + bG_2(ν_2(h))\colon g,h \text{ as above} \right\} \\
                 &\geq H_{a,b}(x).
        \end{split}
    \end{equation}
    On the other hand, let $ε>0$ and fix $y$ and $z$ such that
    \begin{equation}
        H_{a,b}(x) \geq aG_1(y) + bG_2(z) - ε.
    \end{equation}
    There exists $k>0$ such that we can find $g\in (R_{L_1})_{ak}$ and $h\in (R_{L_2})_{bk}$ such that
        \begin{align*}
                &ν_1(g) = y,\quad ν_2(h)=z\\
                &\frac{\lev(g)}{ak} \leq G_1(y) + ε,\thickspace \text{ and }\thickspace\frac{\lev(h)}{bk} \leq G_2(z) + ε,
        \end{align*}
    where $ν_i\colon R_{L_i}\ra Δ(L_i)$ are the two valuations. We have
    \begin{align*}
            G_{a,b}(x)&\leq (\lev(g) + \lev(h))/k && \\
            &\leq aG_1(y) + bG_2(z) + (a+b)ε && \text{by choice of $g$ and $h$}\\
            &\leq H_{a,b}(x) + (a+b+1)ε &&\text{by choice of $y$ and $z$}.
    \end{align*}
    Letting $ε$ tend to $0$ yields
    \begin{equation}
        G_{a,b}(x)\leq H_{a,b}(x).
    \end{equation}
    If $x$ is irrational, the value of $G_{a,b}(x)$ is obtained as the infimum
    \begin{equation}
        \liminf_{δ\to 0}\left\{G_{a,b}(x')\colon |x-x'|<δ\right\}.
    \end{equation}
    The same argument works in this case as well, bearing in mind that we may approximate the value of $G_{a,b}$ at $x$ by $G_{a,b}(x')$ arbitrarily closely since $G_{a,b}(x)$ is convex and bounded from below.
\end{proof}
\begin{remark}
    This result can easily be extended to convex combinations of arbitrary finite collections of test degenerations of $X$.
\end{remark}
\begin{remark}
    It is convenient to work instead with the $\QQ$-line bundle $\frac{aL_1+bL_2}{a+b}$ and reparametrise the family of functions $H_{(a,b)}(x)$ as a function
    \begin{equation}
        H_t\colon Δ\left((1-t)L_1 + tL_2\right)\rightarrow \RR,
    \end{equation}
     where $t$ ranges over the unit interval. We go a step further and identify the range of $H_t$ with a subset of
     \begin{equation}
        V(L_1,L_2) = \mathrm{Conv}\left(Δ(L_1)\times\{0\}, Δ(L_2)\times\{1\} \right)\subset \RR^n\times [0,1].
     \end{equation}
     It would be interesting to know what kind of behaviour the function $H_t$ can exhibit on $V(L_1,L_2)$. The variation of Okounkov bodies was studied by Lazarsfeld-Musta\c{t}\v{a} \cite[Section 4]{lazarsfeld2008convex}.
\end{remark}

If $X$ is toric, Okounkov bodies are a particularly powerful tool. The following examples use the theory of toric varieties. Briefly, the ring of sections of a polarised toric variety $(X_Δ,L)$ corresponding to a polytope $Δ=Δ(L)\subset \RR^n$, where $\RR^n$ contains a fixed lattice $\ZZ^n$, is given by
\begin{equation}
    R=\bigoplus_{k=1}^\infty \frac{\ZZ^n}{k}\cap Δ.
\end{equation}
Sections of $H^0(X,L^{k})$ are identified with points
\begin{equation}
    m/k=(m_1/k,\dotsc,m_n/k)
\end{equation}
in the polytope $Δ$, where $m_i$ are integers. Multiplication of two sections $x$ and $y$ under this identification corresponds to taking their \emph{Minkowski average} $(x+y)/2$ in $Δ$.

\begin{example}[Convex combinations of toric filtrations.]\label{ex:WToric}
    Let $X$ be a toric variety with two line bundles $L_1$ and $L_2$ with section rings $R$ and $S$ isomorphic to the sets of rational points in $Δ(L_1)$ and $Δ(L_2)$, respectively. Let $G_1:Δ(L_1)\rightarrow \RR$ and $G_2:Δ(L_2)\rightarrow \RR$ be lower semicontinuous convex functions and define filtrations
    \begin{equation}\label{eq:CTransform}
        F^f_iR_k = \mathrm{span}_\CC\{x \in P/k : f(x) \leq i \},
    \end{equation}
    and
    \begin{equation}
        F^g_i S_k = \mathrm{span}_\CC\{β \in Q/k : g(β) \leq i \}.
    \end{equation}
    In this case the $(a,b)$-weighted Minkowski average
    \begin{equation}
        \shP =\frac{aΔ(L_1)+bΔ(L_2)}{a+b},
    \end{equation}
    is precisely the Okounkov body of $\frac{aL_1+bL_2}{a+b}$ in the appropriate sense for $\QQ$-line bundles. The family of convex transforms
    \begin{equation}
        G_{a,b}\colon \shP\rightarrow \RR
    \end{equation}
    now characterises the family of test degenerations determined by the weighted product by Donaldson's theory of toric test configurations \cite{Donaldson2002}. Denote $G_t = \frac{G_{a,b}}{a+b}$, where $t=\frac{b}{a+b}$. Studying the behaviour of $G_t$ as $t$ changes may be a useful explicit way to study the variation of test configurations in the weighted product.
\end{example}

\begin{example}\label{ex:P1s}
Consider two $\GG_m$-actions $α$ and $β$ on $\PP^1=\Proj\CC[x,y]$ such that if $(x/y)$ is a local coordinate, $α$ scales $(x/y)$ by weight $c$ and $β$ by $-d$. The filtrations $F^α_\bullet$ and $F^β_\bullet$ defined by $α$ and $β$, respectively, have linear convex transforms on the polytope $P=Q=[0,1]$. Rational points in $[0,1]$ correspond to monomials $x^py^q$ by the bijection
\begin{equation}
    x^py^q\leftrightarrow p/(p+q).
\end{equation}
    It is straightforward to check, either from the definitions or by Theorem \ref{thm:convo}, that the convex transforms of $F_\bullet^α$,$F_\bullet^β$ and $[F_\bullet^{α}] + [F_\bullet^{β}]$ are
    \begin{equation}
        \begin{split}
            f_α(x) &= 1+cx,\\
            f_β(x) &= 1+d(1-x)\\
            f_{α\otβ}(x) &= \max \{1+c(x-1/2), 1-d(x-1/2)\},
        \end{split}
    \end{equation}
    respectively. Geometrically, the corresponding degeneration splits $\PP^1$ into two copies of $\PP^1$ of equal volume intersecting at a fixed point of the $\GG_m$-action. The $\GG_m$-actions on the two components are given by scaling a local coordinate by the integers $c$ and $-d$, respectively.
\end{example}
\begin{example}\label{ex:toric}
    Keep to the notation of Example \ref{ex:P1s}, except now let $c = -d = 1$ and consider the $(a,b)$-diagonal product of filtrations
    \begin{equation}
        (F_\bullet^α\ot_{(a,b)}F_\bullet^β)(\CC[x,y] \ot_\CC \CC[x,y])
    \end{equation}
    for each pair of natural numbers $(a,b)$. The total space of the toric family is, for each pair $(a,b)$, a degeneration of a rational curve into a pair of intersecting curves of lower degree whose ratio of volumes is equal to $t$. As $t$ approaches 0, the limiting convex function corresponds to the vector field $β$. This is also the natural limiting object in $\overline{\test(\PP^1)}$.
    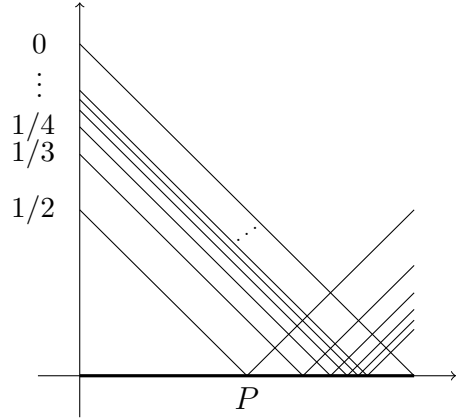
\begin{figure}[!htbp]
    \begin{center}
        \begin{tikzpicture}[scale=1.1]
        \draw[->,thin] (-0.5,0) -- (4.5,0) coordinate (x axis);
        \draw[->,thin] (0,-0.5) -- (0,4.5) coordinate (y axis);
        \draw (-0.48,3.6) node[font=\small] {$\vdots$};
        \draw (0,4.8) node[left=2.8mm] {};
        \draw[color=black,very thick] (0,0) -- node[below] {$P$} (4,0);
        \draw[color=black,thin]  (0,2) node[left=1.8mm,font=\small] {$1/2$} -- (2,0) -- (4,2);
        \draw[color=black,thin]  (0,2.67) node[left=1.8mm,font=\small] {$1/3$} -- (2.67,0) -- (4,1.33);
        \draw[color=black,thin]  (0,3) node[left=1.8mm,font=\small] {$1/4$} -- (3,0) -- (4,1);
        \draw[color=black,thin]  (0,3.2) -- (3.2,0) -- (4,0.8)          node[right] {};
        \draw[color=black,thin]  (0,3.33) -- (3.33,0) -- (4,0.67);
        \draw[color=black,thin]  (0,3.44) -- (3.44,0) -- (4,0.56)       node[right] {};
        \draw[color=black,thin]  (0,4) node[left=2.9mm,font=\small]{$0$} -- (2,2) node[below=0.6mm,font=\small] {\resizebox{3.3mm}{!}{$\Ddots$}}  -- (4,0);
        \end{tikzpicture}
        \caption{The convex functions corresponding to the product $a[F^α_\bullet] + b[G_\bullet^β]$ in $\overline{\test(\PP^1)}$ for different values of $t$, where we denote $t = b/(a+b)$.}
    \end{center}
    \end{figure}
\end{example}


\section{Pullback test configurations} 
\label{sec:pullback_test_configurations}
We fix a projective morphism $p\colon Y\rightarrow B$ and let $L$ be an ample line bundle on $B$. In Section~\ref{sec:filtrations_and_k_stability} we defined test configurations which are fibred over $B$ in a $\GG_m$-equivariant way. As a further application of the constructions of the previous sections, we construct test configurations of $Y$ which are naturally fibred over a test configuration of $B$ called \emph{pullback test configurations}.

Let $F_\bullet R_L$ be an element of $\test(B)$. After replacing $L$ with a power if necessary, we obtain an admissible filtration of $\shR_L$, also denoted by $F_\bullet \shR_L$. Let $\shL$ be a relatively ample line bundle on $Y$ and define a map
\begin{equation}
    Φ_{(a,b)}:\test(B)\rightarrow \test_B(Y)
\end{equation}
by letting $Φ(F_\bullet\shR_L)$ be the the filtration
\begin{equation}
    \bigoplus_{k=0}^\infty \shA_{ak}\ot F_\bullet L^{bk}.
\end{equation}

\begin{lemma}
    The map $Φ$ preserves admissible filtrations.
\end{lemma}
\begin{proof}
    This is a special case of Lemma \ref{lem:tensorAdm}.
\end{proof}

\begin{definition}\label{def:pullback}
    We say that $Φ_{(a,b)}(F_\bullet \shR_L)$ is the \emph{pullback of $F_\bullet \shR_L$ weight $(a,b)$}.
\end{definition}

\begin{example}[Pullbacks of test configurations]
    Assume that $F_\bullet R_L$ is a finitely generated admissible filtration and let $\mathscr{B}$ be the scheme $\Proj F_\bullet R_L$. Considering the algebra $\shRees_{\shO_B} Φ_{(a,b)}(F_\bullet \shR_L)$ as a $\shO_{\mathscr{B}}$-algebra determines a morphism
    \begin{equation}
        \Y=\shProj_B \shRees_{\shO_B} Φ_{(a,b)}(F_\bullet \shR_L)
    \end{equation}
    such that the diagram
    \[
    \begin{tikzcd}
    \Y \arrow{r}{} \arrow{dr} & \mathscr{B} \arrow{d}\\
    & \AA^1
    \end{tikzcd}
    \]
    commutes.
\end{example}

\begin{definition}\label{def:adipull}
    Define the line bundle
    \begin{equation}
        \shL_{a,b}=\shO(a)\ot p^*L^b
    \end{equation}
    on $\shProj_B \shA$. Alternatively, the line bundle $\shL_{a,b}$ is the Serre line bundle on $\shProj \left(\shA\ot_{\shO_B}\shR_L\right)_{(a,b)}$. We have already seen in Lemma \ref{lem:adiample} that given a locally finitely generated $p$-test degeneration $G_\bullet\shA\in \test_B(Y)$, the relative test configuration
    \begin{equation}
        \Y = \shProj_B \left(\shA\ot_{\shO_B}\shR_L\right)_{(a,b)}
    \end{equation}
    is ample for $d\gg 0$. Denote the Serre line bundle on $\Y$ by $\scrL_{(a,b)}$. In particular, if $a=1$ simply write $\scrL_{(a,b)} = \scrL_{b}$.
\end{definition}

We give two examples of a nice phenomenon which happens with pullback test configurations for adiabatic polarisations. The first example, due to Stoppa \cite{stoppa2007unstable}, was already mentioned in Section \ref{sub:k_stability_of_csck_manifolds}.
\begin{example}\label{ex:SPB}
     Let $p\colon Y\rightarrow B$ be a blow up of a zero dimensional subscheme $Z$ and $\shB$ a test configuration for $(B,L)$. Let $\Y$ be the pullback of $\shB$ of weight $(1,m)$. Then the Donaldson-Futaki invariant of the test configuration $\DF(\Y,\scrL_m)$ is given by
        \begin{equation}\label{eq:dfpull}
            \DF(\Y,\scrL_m) = \DF(\shB) - Cm^{1-n} + O(m^{-n}),
        \end{equation}
        where $n$ is the dimension of $B$ and $C$ is a positive constant.
\end{example}
Similar results were also proved for slope stability by Ross and Thomas \cite[Section 5.5]{RossThomas}, and later by Stoppa \cite[Lemma 3.1]{stoppa2011relative}.

 The second example is due to Ross and Thomas \cite[Section 5.4]{RossThomas}.
\begin{example}\label{ex:RTPB}
     Let $p\colon Y\rightarrow B$ be a projective bundle or a flag bundle and $B'$ a subscheme of $B$. Let $\Y$ be a pullback test configuration with weight $(1,m)$ of the slope test configuration of $\mathcal{I}_{B'}\subset \shO_B$ defined in Remark \ref{rem:slope} with slope parameter 1. Then the leading term in $m\in\NN$ of the Donaldson-Futaki invariant of the test configuration $\DF(\Y,\scrL_m)$ is given by
        \begin{equation}\label{eq:dfp}
            \DF(\Y,\shL_m) = \DF(\shB)+ O(m^{-1}),
        \end{equation}
    where $(\shB,L)$ is the test configuration determined by the pullback of $B'$.

    Ross and Thomas presented the calculation in the case of a projective bundle but the flag bundle case follows verbatim.
\end{example}

\begin{remark}
    In the following we have various spaces of sections endowed with natural $\GG_m$-actions. For each vector space we wish to have a succinct and obvious notation for the trace function defined on page \pageref{lem:polys}. Given a vector space $V$ with a natural $\GG_m$-action, we write the trace function simply as $\tr V$.
\end{remark}

\begin{remark}\label{rem:product}
    A product of two cscK polarised varieties $(X_1,L_1)$ and $(X_2,L_2)$ is cscK with respect to the product polarisation $L_1\ot L_2$. It is our hope that an algebraic proof of the K-stability of the polarisation $L_1\ot L_2$ would be found. The difficulty is having to consider test configurations which are not pullbacks from either $X_1$ or $X_2$. We believe it should not be necessary to consider these more complicated test configurations to decide whether $(X_1\times X_2, L_1\ot L_2)$ is K-stable, in contrast with the example of an unstable product of two curves in \cite{ross2006unstable}.
\end{remark}

\begin{remark}[Toric bundles]\label{rem:tbundle}
    There is a simple type of relative test configuration that has appeared in \cite{apostolov2006stability}. Let $\EE$ be a principal $\GL(n,\CC)$-bundle over $B$ and consider a torus bundle $\mathbb{T}$ in $\EE$ with fibre $(\GG_m)^{\times e}$. Then one may define a fibrewise orbit closure $Y$ of $\mathbb{T}$ using the theory of toric varieties. The theory of toric test configurations developed in \cite{Donaldson2002} generalises to this context and yields test configurations which intuitively degenerate fibres of the projection $Y\rightarrow B$ in a uniform way. The authors of \cite{apostolov2006stability} proved partial results about the extremal YTD correspondance for adiabatic polarisations on toric bundles constructed in this way.

    We think of the test configurations defined in \cite{apostolov2006stability}, which preserve the homotopy type of the associated principal bundle but degenerate the fibres of $p\colon Y\rightarrow B$, as complementary to the test configuration defined in Chapter \ref{chap:flags}. We studied test configurations which changes the homotopy type of the associated principal $\GL(n,\CC)$-bundle but preserves the fibres of $p$.
\end{remark}

In light of the previous remarks, we conclude that particularly on adiabatic polarisations of $Y$, there are two natural families of test configurations: ample $p$-test configurations and pullback test configurations. A perhaps naive conjecture we wish to make, motivated by known partial results on blowups, projective bundles, rigid toric bundles blowups and now flag bundles, is that these two test classes of test degenerations characterise the stability of adiabatic polarisations in the following sense.
\begin{conjecture}\label{conj:adistab}
    Let $p\colon Y\rightarrow B$ be a projective morphism with $(B,L)$ a polarised variety and $\shL_{(a,b)}$ as in Definition \ref{def:adipull}. Then there exists an integer $b_0>0$ such that the pair $(Y,\shL_{(a,b)})$ is K-stable (K-polystable, K-semistable) for $b > b_0$ if and only if it is K-stable with respect to test configurations in $\test_B(Y,\shL_{(a,b_0)})$ and pullback test configurations under the projection $p$ with weight $(a,b_0)$.
\end{conjecture}

\begin{remark}[Some remarks about Conjecture \ref{conj:adistab}]
    The hypothesis that projective morphism should be enough to yield the statement may be overenthusiastic as we have only studied very simple examples (flag bundles in Chapter \ref{chap:flags} and certain closed immersions in Chapter \ref{chap:sub}) in this work.

    We also conjecture that the Conjecture \ref{conj:adistab} holds with admissible filtrations and \kbar-stability in place of test configurations and K-stability.

    Finally, an example in Ross \cite{ross2006unstable} shows that the statement of the conjecture does not hold for arbitrary polarisations on $Y$.
\end{remark}


\section{Natural filtrations of shape algebras} 
\label{sec:SW}
Fix a coherent sheaf $\shE$ with a subsheaf $\shF$ on a scheme $B$, a partition $λ$ with jumps given by $r$. Then we define a filtration $W_\bullet S_λ(\shE)$ which is generated by $\shF\subset \shE$ (cf. Definition \ref{def:FTA} and Definition \ref{def:FGD1}). The basic idea goes back to Griffiths, who defined a natural filtration of an exterior power of a vector bundle \cite{griffiths1979algebraic}.

\begin{example}\label{ex:PSlope}
    The filtration of $S(\shE)$ generated by $\shF\subset \shE$ is given by
    \begin{equation}
        \begin{split}
            \shF\subset \shE \oplus S^2\shF \subset \shE \oplus \shF\cdot\shE \oplus S^3\shF\\
             \subset \shE\oplus S^2\shE \oplus \shF\cdot S^2\shE \oplus S^4\shE \subset \dotsb.
        \end{split}
    \end{equation}
    Here we have used the notation $\shF\cdot \shE$ to mean tensors in $S^2\shE$ which are in the image of the symmetrisation map $\shF\ot \shE\rightarrow S^2 E$. Note that the same filtration can be obtained from the filtration $\mathcal{I}_{\PP\shF}\subset \shO_{\PP\shE}$ using Remark \ref{rem:slope}.
\end{example}

In general, the subsheaf $\shF\subset\shE$ generates a filtration
\begin{equation}
    W_\bullet \shE^λ = (W_\bullet S_λ(\shE))_1,
\end{equation}
which we write in terms of the factors of $\shF$ and $\shE$ in the tensor algebra $T(\shE)$ as
\begin{equation}\label{eq:SF1}
    W_i \shE^λ = c_λ \left(\shF^{\ot i}\ot \shE^{\ot(l-i)}\right)\ot_{\CC[\mathfrak{S}_i]\times\CC[\mathfrak{S}_{l-i}]}\CC[\mathfrak{S}_l].
\end{equation}
Here $c_λ$ is the Young symmetriser (cf. Definition \ref{def:schur}) and $\CC[\mathfrak{S}_i]$ denotes the group algebra of the symmetric group, which acts on $T(\shE)$ by permuting the tensor factors. In other words, the module $W_i \shE^λ$ is generated by tensors with at least $i$ factors are contained in $\shF$. The filtration in Equation~\eqref{eq:SF1} is a finite decreasing filtration and a simple change of indexing yields an increasing filtration which generates an admissible filtration of the algebra $S_λ(\shE)$. We call this filtration the \emph{$\shF$-weight filtration of $S_λ(\shE)$} and denote it by $\widehat{W}^\shF_\bullet S_λ(\shE)$. In contrast, we denote the filtration generated by the descending filtration of Equation~\eqref{eq:SF1} of increasing powers of $\shF$ by $W_\bullet \shF S_λ(\shE)$.

\begin{remark}
    The test configuration determined by the subsheaf $\shF\subset \shE$ for flag bundles is not given by the theory of slope stability as it does in the case of projective bundles Example \ref{ex:PSlope}, but by a more complicated filtration of the structure sheaf $\shO_{\Flag_r(\shE)}$ (Remark \ref{rem:slope0} and Remark \ref{rem:slope}). This filtration is obtained from a flag of \emph{relative Schubert varieties} determined by increasing incidence conditions with the subsheaf $\shF$.
\end{remark}

\begin{example}[Computation of the weight function]\label{rem:decomposable}\label{rem:decom}
    Consider a direct sum $\shF\oplus \shQ$ of coherent sheaves on $B$. We write
    \begin{equation}\label{eq:little}
            S_λ(\shF\oplus \shQ)_k = (\shF\oplus \shQ)^{kλ} = \bigoplus_{|ν|+|μ|=k|λ|}M^{kλ}_{νμ} \shF^ν\otimes \shQ^μ
    \end{equation}
    using the Littlewood-Richardson rule. We have
    \begin{equation}
        W_i E^{kλ} = \bigoplus_{|ν| \leq i}M^{kλ}_{νμ} \shF^ν\otimes \shQ^μ.
    \end{equation}
    We define the corresponding weight function
    \begin{equation}\label{eq:wfun1}
        \begin{split}
            w(k) &= \sum_{i=0}^\infty i \left(χ(W_i S_{λ}(\shF\oplus \shQ)_k)   - χ(W_{i-1}S_λ(\shF\oplus \shQ)_k) \right)\\
            &=\sum_{i=0}^\infty i \bigoplus_{|ν| = i}M^{kλ}_{νμ} \shF^ν\otimes \shQ^μ
        \end{split}
    \end{equation}
    This is the weight function which appeared in Lemma \ref{lem:weight}.
\end{example}

Example \ref{ex:PBundle} generalises to more general flag bundles, but the trick we used in Chapter \ref{chap:flags} does not compute the weight function any longer.
\begin{example}[A product of two simple filtrations of a shape algebra]
    Let $E$ be a vector bundle isomorphic to a direct sum of subbundles $F\oplus Q$. Let $\shA = S_λ(E)$ be a shape algebra for $\Flag_r(E)$ with a polarisation $\shL_λ(A)$. Consider the two filtrations $W^F_\bullet \shA$ and $W^Q_\bullet \shA$. The filtration
    \begin{equation}\label{eq:s2filt}
         F\ot Q\subset F\ot E \oplus Q\ot E = S^2 E
    \end{equation}
    generates the tensor product filtration $(W^F\ot_{(1,1)} W^Q) (\shA\ot_{\shO_B} \shA)$ of the $(1,1)$-diagonal of $\shA\ot_{\shO_B}\shA$ via the projection
    \begin{equation}
        α\colon S_λ(S^2 E)\rightarrow S_{2λ}(E).
    \end{equation}
    The kernel of $α$ is a complicated object which can be described by decomposing the representation $S_λ(S^2 E)$ into irreducible representations. The composition of Schur functors is called \emph{plethysm} \cite[p. 63]{Weyman}.
\end{example}


\chapter{Further directions} 
\label{sub:further_research}
We end by outlining three directions in which this work can be developed. Fix a smooth scheme $B$ over $\CC$ and vector bundle $E$ of rank $r_E$ on $B$.
\section{Chern character formula} 
\label{sec:FCCF}
We hope to find a generalisation to the Chern character formula of Theorem \ref{thm:maina}. The proof we presented required the assumption that $λ$ is proportional to the canonical partition $σ_{r_E,r}$ for some tuple $r$, but it is easy to verify computationally that this assumption is not required for the statement to be true in many special cases. Perhaps it is possible to use Schubert calculus to reduce inductively to the case solved in this thesis. The assumption on the partition forced us to make a highly undesirable restriction in our choice of polarisation for the flag bundle in our discussion of its K-stability in Chapter \ref{chap:flags}.

Decompose $\ch E^{kλ}$ as follows
\begin{equation}\label{eq:chernApp}
    \ch E^λ =\rank E^{kλ} \sum_{i=1}^bB_i(E,kλ),
\end{equation}
where $B_i(E,λ)$ has degree $i$ in the Chow ring of $B$. Then expand $B_i(E,kλ)$ by decreasing degree in $k$ as
\begin{equation}\label{eq:chernDec}
    B_i(E,λ) = B_{i,0} k^i +B_{i,1} k^{i-1} + \dotsb + B_{i,i} k^0,
\end{equation}
It seems that a general closed formula for the polynomials $B_{ij}(E,λ)$ in the expansion \ref{eq:chernDec} should be attainable, generalising Manivel's beautiful result stated in Theorem \ref{thm:manivel}.

\section{Flag bundles and projective bundles} 
\label{sec:FRKS}
    Let $(B,L)$ be a smooth polarised variety of dimension $b$ and $E$ is a vector bundle on $B$.

    If the underlying vector bundle has higher rank, K-stability of its flag bundles depends on higher Chern classes, which were cancelled out by considering adiabatic polarisations in Section~\ref{sec:anybase}. It would be interesting to know if such dependence has a geometric interpretation. This would require generalising Theorem \ref{thm:maina} describing terms in Equation~\eqref{eq:chernDec}.

    In the adiabatic case that it suffices to calculate $B_{i,0}$ and $B_{i,1}$. While this is possible for fixed $k$ and $λ$, it does not seem easy to generalise the arguments of \cite{manivel1994theoreme} or Chapter \ref{chap:chern} to obtain the coefficients $B_{i,j}$. For general polarisations, the knowledge of the term $B_{3,1}$ would immediately allow the calculation of Donaldson-Futaki invariants of any test configuration induced a subbundle filtration $F\subset E$, and the base $B$ has dimension 2. It may be possible to extend the arguments of Chapter \ref{chap:chern} to this case.

    Classical flag varieties which are studied in this work are only one example of a more general construction. Let $G$ be a semisimple complex group. Then quotients of $G$ by subgroups containing the Borel subgroup of $G$ are projective varieties. We call such a variety a \emph{generalised flag manifold}. They are classified by subsets of nodes on Dynkin diagrams of the Dynkin diagram of the corresponding group $G$. From the point of view of Kähler geometry, generalised flag manifolds have very similar properties to the classical ones.

    A Borel-Weyl pushforward formula, similar to one stated in Section~\ref{sec:relative_flag_varieties} for classical flag bundles, also holds for the symplectic and orthogonal groups \cite[Chapter 4]{Weyman}. For example, if $F$ is a vector bundle of even rank on the base $B$ and
    \begin{equation}
        \langle\cdot,\cdot\rangle: F\times_B F\rightarrow \CC
    \end{equation}
    is a symplectic form. We define the isotropic flag variety $\mathop{\shI\shF\it{lag}}\nolimits_r(E)$ of $r$-flags of isotropic subspaces in $F^*$. Subbundles of $F$ can be used to define test configurations of $\mathop{\shI\shF\it{lag}}\nolimits_r(E)$. It would be interesting to know if the behaviour of the Donaldson-Futaki invariants is similar to that seen in Chapter \ref{chap:flags}.


\section{Relative K-stability and operations on test configurations} 
\label{sec:operations_on_test_configurations}

Ampleness of the relative test configurations was not discussed in this work. This is a fundamental property which brings us back to the theory of K-stability. An effective result is not known to us even in the flag bundle case.

We believe that explicitly computing Donaldson-Futaki invariants of families of test configurations in examples can be used to exhibit  new interesting behaviour of K-stability in the cone of polarisations. We hope this may help in establishing a conjectural picture for the behaviour of K-stability in families of polarised varieties where the polarisation $L$ varies on a fixed underlying variety $X$.

The calculations presented in this work could be generalised to give further examples of K-unstable varieties. For example, the stability of higher dimensional projective bundles is still wide open over a higher dimensional base and Donaldson-Futaki invariants have only been computed for very simple test configurations. Finding an explicit formula for the Donaldson-Futaki invariant similar to one found in Example \ref{ex:PBundle} should be possible in higher dimensions, particularly, if the vector bundle is a direct sum of two line bundle. We believe that it should be possible to, for example, find a examples of \emph{nonalgebraic obstructions} on both rational and irrational polarisations this way by using Remark \ref{rem:irrextension}.

Although we do not expect it to have applications to K-stability, describing the convex geometry associated to convex transforms on moving Okounkov bodies as the polarisation varies, discussed in Section \ref{sec:convex_transform_okounkov_bodies_and_examples}, is an interesting on its own right.



\renewcommand{\theequation}{\Alph{chapter}.\arabic{equation}}
\appendix 

\chapter{Appendix}

\section{Combiproofs} 
\label{sec:combiproofs}

We include the proofs of the combinatorial formulae for completeness.
\begin{lemma}
Let $k$ and $n$ be integers and let $p(k)={k+n-1\choose n-1}$. Then
\begin{equation}
     \sum_i i^2 {n-2+k-i \choose n-2} = \frac{(n+2k-1)(k+n-1)!}{(k-1)!(n+1)!} = ( 2k^2 + k(n-1) )p(k)
\end{equation}
and
\begin{equation}
    \sum_{i,j} ij {n-3+k-i-j \choose n-3} = \frac{(k+n-1)!}{(k-2)!(n+1)!} = k(k+1) p(k).
\end{equation}
\begin{proof}
We prove the first identity by induction on $n$ and $k$. Let
\begin{equation}
    f(k,n)=\sum_i i^2 {n-2+k-i \choose n-2}
\end{equation}
Using the identity
\begin{equation}
    {n \choose k}= {n-1 \choose k-1} + {n-1 \choose k},
\end{equation}
which holds for all $0\leq k \leq n-1$ we see that
\begin{equation}
    \begin{split}
        f(k,n)&=\sum_{i=1}^k i^2 {n-2+k-i \choose n-2}\\
        &= k^2 + \sum_{i=1}^{k-1} i^2 \left({n-3+k-i \choose n-3} + {n-2+(k-1)-i \choose n-2}\right)\\
        &= k^2 + f(k,n-1) - k^2 + f(k-1,n)\\
        &= f(k-1,n)+f(k,n-1).
    \end{split}
\end{equation}
Finally we verify that
\begin{equation}
     \frac{(n+2k-3)(k+n-2)!}{(k-2)!(n+1)!} + \frac{(n+2k-2)(k+n-2)!}{(k-1)!n!} = \frac{(n+2k-1)(k+n-1)!}{(k-1)!(n+1)!}.
\end{equation}
This completes the induction step. The base case follows from verifying the cases $f(k,2)$ and $f(1,n)$.

The proof of the second identity is almost identical. Let
\begin{equation}
    g(k,n)=\sum_{i=1}^k\sum_{j=1}^{k-i} ij {n-3+k-i-j \choose n-3}.
\end{equation}
Again we have
\begin{equation}
    \begin{split}
    g(k,n) &= \sum_{i=1}^k i(k-i) + \sum_{i=1}^{k-1}\sum_{j=1}^{k-i-1} ij \left({n-4+k-i-j \choose n-4}+{n-3+k-1-i-j\choose n-3}\right)\\
    &= \sum_{i=1}^k i(k-i) + g(k,n-1) - \sum_{i=1}^k i(k-i) + g(k-1,n)\\
    &= g(k-1,n)+g(k,n-1).
    \end{split}
\end{equation}
Verify the right hand side as above by computing
\begin{equation}
     \frac{(k+n-2)!}{(k-3)!(n+1)!} + \frac{(k+n-2)!}{(k-2)!n!} = \frac{(k+n-1)!}{(k-2)!(n+1)!}.
\end{equation}
The base case follows from verifying the cases $g(k,2)$ and $g(1,n)$.
\end{proof}
\end{lemma}

\begin{remark}
    Let $μ$ be a partition. Higher degree terms of Chern characters of symmetric bundles can be computed from a more general formula for    $f(k,n,μ)$, where
\begin{equation}
    f(k,n,λ)=\sum_{i_1=1}^k\sum_{i_2=1}^{k-i_1}\text{ }\cdots\sum_{i_u=1}^{k-(i_1+\dotsb+i_{u-1})}
         i_1^{j_1}\cdots i_u^{j_u}{n + k - c_1(λ) - u - 1 \choose n - u - 1},
\end{equation}
    where we denote $c_1(λ) = u$.
\end{remark}


\section{An elementary proof of Arezzo-Della-Vedova's formula} 
\label{sec:an_elementary_proof_of_arezzo_della_vedova_s_formula}
For completeness, we present an elementary derivation of the formula for the Futaki invariant of a complete intersection along the same lines as \cite[Section 4]{arezzo2011k}.
\begin{definition}
  Let $p(k)$ be a polynomial in $k$ with coefficients in an arbitrary ring and $\um$ a vector of $u$ natural numbers $(s_1\ldots s_u)$. Define
    \begin{equation}
        \begin{split}
            p^{\um}(k)=&p(k)-p(k-s_1)+\cdots+p(k-s_u) + \sum_{i\neq j} p(k-s_i-s_j)\\
            &+\cdots+ (-1)^q p(k-s_1-\cdots-s_u)
        \end{split}
    \end{equation}
\end{definition}
\begin{lemma}\label{lem:hilb}
    Let
    \begin{equation}
        p(k) = a_0k^n+a_1k^{n-1}+O(k^{n-2})
    \end{equation}
    be a polynomial of degree $n$ with coefficients in an arbitrary ring and $\um=(s_1\ldots s_q)$. Then $p^{\um}(k)$ is a polynomial of degree $n-q$ and if we write
    \begin{equation}
        p^{\um}(k) = \sum_{i=0}^{n-u}c_i k^{n-u-i}
    \end{equation}
    the first two coefficients are given by
    \begin{equation}\label{eq:topc}
        c_0=C(\um) a_0
    \end{equation}
    and
    \begin{equation}\label{eq:secondc}
        c_1=C(\um)\left(\frac{n-u}{n}\right)\left(a_1-\frac{n\sum_{i=1}^us_i}{2}a_0 \right),
    \end{equation}
    where
    \begin{equation}
        C(\um)=\left(\prod_{i=1}^us_i\right)\frac{n!}{(n-u)!}.
    \end{equation}
\end{lemma}
\begin{proof}
    The proof is an easy induction on $u$. If $r=1$ the statement is easy to verify. Let $m\in\NN$. We have
    \begin{equation}
        \begin{split}
        p(k)-p(k-m)&=nma_0k^{n-1} + m\left((n-1)a_1 - \binom{n-1}{2}a_0\right)k^{n-2}\\&+O(k^{n-3})
        \end{split}
    \end{equation}
     as required. Assume that the statement holds for all $u$-tuples and let $\um=(s_1,\ldots,s_{u})$ and $\um'=(s_1,\ldots,s_{r+1})$. Notice that
    \begin{equation}
        p^{\um'}(k)=p^{\um}(k)-p^{\um}(k-s_{u+1}).
    \end{equation}
    so by the inductive hypothesis we have
    \begin{equation}
        \begin{split}
            p^{\um}(k-s_{u+1})&= c_0k^{n-u} + \left( c_1-(n-u)s_{r+1} c_0\right) k^{n-u-1}\\ &+ \left(c_2-c_1(n-u-1)s_{u+1}  + c_0\binom{n-u}{2}s_{u+1}^2\right)k^{n-u-2} \\&+ O(k^{n-u-3}),
        \end{split}
    \end{equation}
    where $c_0$ and $c_1$ are as in the statement of the Lemma. Finally, we verify that
    \begin{equation}
        \begin{split}
            (n-u)s_{u+1}r!\left(\prod_{i=1}^us_i\right)\binom{n}{u} a_0 = (u+1)!\left(\prod_{i=1}^{u+1}s_i\right)\binom{n}{u+1} a_0
        \end{split}
    \end{equation}
    and
    \begin{equation}
        \begin{split}
            &c_1(n-u-1)s_{u+1} - c_0\binom{n-u}{2}s_{u+1}^2 \\
            &=(u+1)!\left(\prod_{i=1}^{u+1}s_i\right)\binom{n-1}{u+1}\left( a_1-\frac{n\sum_{i=1}^{u+1} s_i}{2}a_0\right)
        \end{split}
    \end{equation}
    as required.
\end{proof}
\begin{proof}[Proof of Proposition \ref{prop:cidf}]
We will use the Koszul resolution to compute the Hilbert and trace polynomials. We have the exact sequence
\begin{gather}{\label{eq:koszulseq}}
        0\rightarrow\shO_Y(k-\sum_{j=1}^u{s_j})\rightarrow\dotsb\rightarrow  \oplus_{i=1}^u\shO_Y(k-s_i)\rightarrow\shO_Y(k)\rightarrow \shO_X(k)\rightarrow 0.
\end{gather}
Thus the Hilbert polynomial of $X$ is given by
\begin{equation}\label{eq:koszul}
    h^0(X,\shO_X(k))=\sum_{j=0}^u\sum_{|I|=j}(-1)^{j}h^0(Y,\shO_Y(k-\sum_{l\in I}s_{l}))
\end{equation}
where the summation is over all subsets $I$ of $\{1,\ldots,r\}$ of size $j$. We denote the Hilbert polynomial of $\shO_Y(1)$ by $h^0_Y(k)$ and expand it as
\begin{equation}
    h^0_Y(k)= a_0 k^n+a_1 k^{n-1}+O(k^{n-2}).
\end{equation}
The highest order terms of the Hilbert polynomial
\begin{equation}
    h^0(X,\shO_X(k))= c_0k^{n-u}+c_1k^{n-u-1}+O(k^{n-u-2}).
\end{equation}
of $\shO_X(k)$ are given by Lemma \ref{lem:hilb}. The trace of the $\GG_m$-action on $X$ is computed similarly. Let $w_Y(k)$ and $w_X(k)$ be the weight polynomials of the $\GG_m$-representations on $H^0(\shO_Y(k))$ and $H^0(\shO_X(k))$, respectively, and write them as
\begin{equation}
    w_Y(k)= b_0k^{n+1}+b_1k^{n}+O(k^{n-1})
\end{equation}
and
\begin{equation}
    w_X(k) = d_0k^{n-u+1}+d_1k^{n-u}+O(k^{n-u-1}).
\end{equation}
By keeping track of the $\ZZ$-grading in the exact sequence in Equation~\eqref{eq:koszulseq}, we find
\begin{equation}
    \begin{split}
        w_X(k)&=\sum_{j=0}^m\sum_{|I|=j}(-1)^{j}w_Y(k- s_{i_1}-\dotsb-s_{i_j})\\
        +&\sum_{j=1}^m\sum_{|I|=j}(-1)^{j}\left(s_1+\dotsb+s_j\right)h^0(Y,\shO(k- s_{i_1}-\dotsb-s_{i_j})).
    \end{split}
\end{equation}
We rewrite this as
\begin{equation}
    w_X(k)=w_Y^{\um}(k) - \sum_{i=1}^u γs_i\left(h^0_Y\right)^{\um_{\hat{i}}}(k-s_i)
\end{equation}
where the hat notation means that the $i$th member of the tuple is omitted, that is
\begin{equation}
    \um_{\widehat{i}}=(s_1,\ldots, s_{i-1},s_{i+1},\ldots s_u).
\end{equation}
Using Lemma \ref{lem:hilb} on $w_Y^{\um}(k)$ and $\left(h^0_Y\right)^{\um_{\hat{i}}}(k-s_i)$, we see that
\begin{equation}
    \begin{split}
        d_0&=\frac{(n+1)!}{(n-u+1)!}\left(\prod_{i=1}^us_i\right)b_0-γ\sum_{j=1}^u\left(s_j\frac{n!}{(n-u +1)!}\frac{\prod_{i=1}^{u}s_i}{s_j}\right)a_0\\
        &=C(\um)\frac{n+1}{n-u+1}\left(b_0-\frac{γu}{n+1}a_0\right)
    \end{split}
\end{equation}
and
\begin{equation}
    \begin{split}
        d_1&=\frac{n!}{(n-u)!}\left(\prod_{i=1}^us_i\right)\left(b_1-\frac{(n+1)\sum_{j=1}^us_j}{2}b_0\right)\\
        &-γ\sum_{i=1}^u(s_i\frac{(n-1)!}{(u-1)!}\frac{\prod_{j=1}^{r}s_j}{s_i}\left(a_1-\frac{n\sum_{l=1}^us_l-s_i}{2}a_0\right)\\
        &+γ\sum_{i=1}^u(n-u+1)s_i^2\frac{n!}{(n-u+1)!}\frac{\prod_{j=1}^{u}s_j}{s_i}a_0\\
        &=C(\um)\left(b_1-γ\frac{u}{n}a_1 + \frac{\sum_{l=1}^us_l}{2}\left((u+1)γa_0-(n+1)b_0)\right)\right).
    \end{split}
\end{equation}

Denote $μ_Y=a_1/a_0$, $ν_Y=b_0/a_0$ and $S=\sum_{i=1}^u s_i$. Notice that
\begin{equation}
    \DF(\Y)=\frac{b_0a_1}{a_0^2}-\frac{b_1}{a_0}=μ_Yν_Y-\frac{b_1}{a_0}.
\end{equation}
The Donaldson-Futaki invariant of $\X$ is therefore given by
\begin{equation}\label{eq:fut}
    \begin{split}
        \DF(\X) &= \frac{d_0c_1}{c_0^2}-\frac{d_1}{c_0}\\
        &=\frac{n+1}{n-u+1}\left(ν_Y-\frac{γu}{n+1}\right) \frac{n-u}{n}\left(μ_Y-\frac{nS}{2} \right)\\
        &-\frac{b_1}{a_0}+γ\frac{u}{n}μ_Y - \frac{S}{2}\left((u+1)γ-(n+1)ν_Y)\right)\\
        &=\frac{(n+1)(n-u)}{(n-u+1)n}\left(ν_Yμ_Y + \frac{nuSγ}{2(n+1)}-\frac{nSν_Y}{2}-\frac{uγμ_Y}{n+1}\right)\\
        &-μ_Yν_Y+\DF(\Y)+γ\frac{u}{n}μ_Y - \frac{(u+1)Sγ}{2} + \frac{(n+1)Sν_Y}{2}\\
        &=\DF(\Y)+\frac{ν_Y-γ}{n-u+1}\left(\frac{(n+1)S}{2} -  \frac{uμ_Y}{n}\right).
    \end{split}
\end{equation}
This completes the proof.
\end{proof}

\bibliography{Stability}

\begin{thebibliography}{10}

\bibitem{akhiezer1995lie}
D.~Akhiezer.
\newblock {\em Lie group actions in complex analysis}.
\newblock Springer, 1995.

\bibitem{alekseevskii1986invariant}
D.~Alekseevskii and A.~Perelomov.
\newblock Invariant {K}{\"a}hler-{E}instein metrics on compact homogeneous
  spaces.
\newblock {\em Functional Analysis and Its Applications}, 20(3):171--182, 1986.

\bibitem{apostolov2006stability}
V.~Apostolov, D.~Calderbank, P.~Gauduchon, and C.~T{\o}nnesen-Friedman.
\newblock \href{http://people.bath.ac.uk/dmjc20/Papers/sem.pdf}{Stability and
  extremal metrics on toric and projective bundles}.
\newblock 2006.

\bibitem{apostolov2008hamiltonian}
V.~Apostolov, D.~Calderbank, P.~Gauduchon, and C.~T{\o}nnesen-Friedman.
\newblock \href{http://arxiv.org/abs/math/0511118}{Hamiltonian 2-forms in
  {K}{\"a}hler geometry, {III} Extremal metrics and stability}.
\newblock {\em Inventiones mathematicae}, 173(3):547--601, 2008.

\bibitem{apostolov2011extremal}
V.~Apostolov, D.~Calderbank, P.~Gauduchon, and C.~T{\o}nnesen-Friedman.
\newblock \href{http://arxiv.org/abs/0905.0498}{Extremal {K}{\"a}hler metrics
  on projective bundles over a curve}.
\newblock {\em Advances in Mathematics}, 227(6):2385--2424, 2011.

\bibitem{apostolov2006remark}
V.~Apostolov and C.~T{\o}nnesen-Friedman.
\newblock \href{http://arxiv.org/pdf/math/0411271.pdf}{A remark on {K{\"a}}hler
  metrics of constant scalar curvature on ruled complex surfaces}.
\newblock {\em Bulletin of the London Mathematical Society}, 38(03):494--500,
  2006.

\bibitem{arezzo2011k}
C.~Arezzo and A.~Della~Vedova.
\newblock \href{http://arxiv.org/abs/0810.1473}{On the {K}-stability of
  complete intersections in polarized manifolds}.
\newblock {\em Advances in Mathematics}, 226(6):4796--4815, 2011.

\bibitem{arezzo2006blowing}
C.~Arezzo and F.~Pacard.
\newblock \href{http://arxiv.org/abs/math/0411522}{Blowing up and
  desingularizing constant scalar curvature {K}{\"a}hler manifolds}.
\newblock {\em Acta mathematica}, 196(2):179--228, 2006.

\bibitem{arezzo2009blowing}
C.~Arezzo and F.~Pacard.
\newblock \href{http://arxiv.org/abs/math/0504115}{Blowing up {K{\"a}hler}
  manifolds with constant scalar curvature, {II}}.
\newblock {\em Annals of Mathematics}, pages 685--738, 2009.

\bibitem{bando1985uniqueness}
S.~Bando and T.~Mabuchi.
\newblock {\em Uniqueness of Einstein K{\"a}hler metrics modulo connected group
  actions}.
\newblock Universit{\"a}t Bonn. SFB 40. Theoretische
  Mathematik/Max-Planck-Institut f{\"u}r Mathematik, 1985.

\bibitem{berman2012k}
R.~Berman.
\newblock \href{http://arxiv.org/abs/1205.6214}{K-polystability of {Q-Fano}
  varieties admitting {K{\"a}hler-Einstein} metrics}, 2012.

\bibitem{berman2004convexity}
R.~Berman and B.~Berndtsson.
\newblock \href{http://arxiv.org/abs/1405.0401}{Convexity of the {K-energy} on
  the space of {K}ahler metrics and uniqueness of extremal metrics}.
\newblock {\em arXiv:1405.0401 [math.DG]}, 83(1):387--404, 1994.

\bibitem{bott1957homogeneous}
R.~Bott.
\newblock Homogeneous vector bundles.
\newblock {\em Annals of Mathematics}, pages 203--248, 1957.

\bibitem{boucksom2011okounkov}
S.~Boucksom and H.~Chen.
\newblock
  \href{http://webusers.imj-prg.fr/~sebastien.boucksom/publis/BC.pdf}{Okounkov
  bodies of filtered linear series}.
\newblock {\em Compositio Mathematica}, 147(04):1205--1229, 2011.

\bibitem{bourbaki1972commutative}
N.~Bourbaki.
\newblock {\em Commutative algebra}, volume~8.
\newblock Hermann Paris, 1972.

\bibitem{Bruckmann2008}
P.~Br\"{u}ckmann and H.~Rackwitz.
\newblock \href{http://ijpam.eu/contents/2008-48-4/3/3.pdf}{Chern classes of
  {S}chur powers of locally free sheaves}.
\newblock {\em International Journal of Pure and Applied Mathematics},
  48(4):477--482, 2008.

\bibitem{chen2000space}
X.~Chen.
\newblock The space of {K{\"a}hler} metrics.
\newblock {\em Journal of Differential Geometry}, 56(2):189--234, 2000.

\bibitem{chen2012kahler}
X.~Chen, S.~Donaldson, and S.~Sun.
\newblock \href{arxiv.org/abs/1210.7494}{{K{\"a}hler-Einstein} metrics and
  stability}, 2012.

\bibitem{chen2012kahler2}
X.~Chen, S.~Donaldson, and S.~Sun.
\newblock \href{arxiv.org/abs/1212.4714}{{K{\"a}hler-Einstein metrics on Fano
  manifolds, II}: limits with cone angle less than {2π}}, 2012.

\bibitem{chen2013kahler3}
X.~Chen, S.~Donaldson, and S.~Sun.
\newblock \href{arXiv.org/abs/1302.0282}{K{\"a}hler-Einstein metrics on Fano
  manifolds, III}: limits as cone angle approaches 2π and completion of the
  main proof, 2013.

\bibitem{codogni2015non}
G.~Codogni and R.~Dervan.
\newblock \href{arXiv.org/abs/1501.03372}{Non-reductive automorphism groups,
  the {Loewy filtration and K}-stability}, 2015.

\bibitem{dervan2014uniform}
R.~Dervan.
\newblock \href{arXiv.org/abs/1412.0648}{Uniform stability of twisted constant
  scalar curvature K\"{a}ahler metrics}, 2014.

\bibitem{dervan2014alpha}
R.~Dervan.
\newblock \href{http://arxiv.org/abs/1307.6527}{Alpha Invariants and
  {K}-Stability for General Polarizations of {F}ano Varieties}.
\newblock {\em International Mathematics Research Notices}, page rnu160, 2014.

\bibitem{ding1992kahler}
W.~Ding and G.~Tian.
\newblock {K{\"a}hler-Einstein} metrics and the generalized {Futaki invariant}.
\newblock {\em Inventiones mathematicae}, 110(1):315--335, 1992.

\bibitem{donaldson1987infinite}
S.~Donaldson.
\newblock Infinite determinants, stable bundles and curvature.
\newblock {\em Duke Math. J}, 54(1):231--247, 1987.

\bibitem{DonaldsonScalar}
S.~Donaldson.
\newblock
  \href{http://www.emis.de/journals/NYJM/JDGonline/p/2001/59-3-3.pdf}{Scalar
  curvature and projective embeddings, {I}}.
\newblock {\em Journal of Differential Geometry}, 59(3):479--522, 2001.

\bibitem{Donaldson2002}
S.~Donaldson.
\newblock \href{http://www.emis.de/journals/NYJM/JDG/p/2002/62-2-5.pdf}{Scalar
  curvature and stability of toric varieties}.
\newblock {\em Journal of Differential Geometry}, 62(2):289--349, 2002.

\bibitem{Donaldson2008}
S.~Donaldson.
\newblock \href{http://de.arxiv.org/abs/math/0506501.pdf}{Lower bounds on the
  {C}alabi functional}.
\newblock {\em J. Differential Geometry}, 70(3):453--472, 2005.

\bibitem{donaldson2012kahler}
S.~Donaldson.
\newblock \href{http://de.arxiv.org/abs/1102.1196.pdf}{K{\"a}hler metrics with
  cone singularities along a divisor}.
\newblock In {\em Essays in mathematics and its applications}, pages 49--79.
  Springer, 2012.

\bibitem{donaldson2012gromov}
S.~Donaldson and S.~Sun.
\newblock \href{arXiv.org/abs/1206.2609}{Gromov-Hausdorff limits of
  {K{\"a}hler} manifolds and algebraic geometry}.
\newblock 2012.

\bibitem{eisenbud1995commutative}
D.~Eisenbud.
\newblock {\em Commutative Algebra: with a view towbard algebraic geometry},
  volume 150.
\newblock Springer Science \& Business Media, 1995.

\bibitem{Fine2004}
J.~Fine.
\newblock \href{http://arxiv.org/abs/math/0401275}{Constant scalar curvature
  {K{\"a}hler} metrics on fibred complex surfaces}.
\newblock {\em Journal of Differential Geometry}, 68(3):397--432, 2004.

\bibitem{fine2005fibrations}
J.~Fine.
\newblock \href{http://arxiv.org/abs/math/0510075}{Fibrations with constant
  scalar curvature {K{\"a}hler metrics and the CM}-line bundle}.
\newblock {\em Math. Res. Lett.}, 14(2):239--247, 2007.

\bibitem{Fulton1997}
W.~Fulton.
\newblock {\em Young tableaux: with applications to representation theory and
  geometry}, volume~35.
\newblock Cambridge University Press, 1997.

\bibitem{Fulton1998}
W.~Fulton.
\newblock {\em Intersection theory}, volume 1998.
\newblock Springer Berlin, 1998.

\bibitem{FultonRep}
W.~Fulton and J.~Harris.
\newblock {\em Representation theory: a first course}, volume 129.
\newblock Springer, 1991.

\bibitem{futaki1983obstruction}
A.~Futaki.
\newblock An obstruction to the existence of {Einstein-K{\"a}hler} metrics.
\newblock {\em Inventiones mathematicae}, 73(3):437--443, 1983.

\bibitem{griffiths1979algebraic}
P.~Griffiths and J.~Harris.
\newblock Algebraic geometry and local differential geometry.
\newblock In {\em Annales Scientifiques de l'{\'E}cole Normale Sup{\'e}rieure},
  volume~12, pages 355--452. Soci{\'e}t{\'e} math{\'e}matique de France, 1979.

\bibitem{grothendieck1957classification}
A.~Grothendieck.
\newblock Sur la classification des {fibr{\'e}s} holomorphes sur la sphere de
  {Riemann}.
\newblock {\em American Journal of Mathematics}, 1957.

\bibitem{grothendieck1958theorie}
A.~Grothendieck.
\newblock La th{\'e}orie des classes de {C}hern.
\newblock {\em Bulletin de la soci{\'e}t{\'e} math{\'e}matique de France}, 86,
  1958.

\bibitem{Hartshorne1966}
R.~Hartshorne.
\newblock Ample vector bundles.
\newblock {\em Publications Math{\'e}matiques de l'IH{\'E}S}, 29(1):64--94,
  1966.

\bibitem{Hartshorne1971}
R.~Hartshorne.
\newblock Ample vector bundles on curves.
\newblock {\em Nagoya Mathematical Journal}, 43:73--89, 1971.

\bibitem{hartshorne1977algebraic}
R.~Hartshorne.
\newblock {\em Algebraic geometry}.
\newblock Springer, 1977.

\bibitem{Hong2002}
Y.-J. Hong.
\newblock Gauge-fixing constant scalar curvature equations on ruled manifolds
  and the {F}utaki invariants.
\newblock {\em Journal of Differential Geometry}, 60(3):389--453, 2002.

\bibitem{huybrechts2006complex}
D.~Huybrechts.
\newblock {\em Complex geometry: an {I}ntroduction}.
\newblock Springer Science \& Business Media, 2006.

\bibitem{Huybrechts2010}
D.~Huybrechts and M.~Lehn.
\newblock {\em The geometry of moduli spaces of sheaves}.
\newblock Cambridge Univ Pr, 2010.

\bibitem{codepage}
A.~Isopoussu.
\newblock Code section.
\newblock \href{https://www.dpmms.cam.ac.uk/~aai22/index.html}{Personal
  homepage}, March 2015.

\bibitem{kellermemoire}
J.~Keller.
\newblock
  \href{http://www.latp.univ-mrs.fr/~jkeller/Papers/KELLER-hdr.pdf}{M{\'e}moire
  d'Habilitationa Diriger des Recherches}.

\bibitem{keller2014projectivisation}
J.~Keller.
\newblock \href{http://arXiv.org/abs/1407.7062}{About projectivisation of
  {Mumford} semistable bundles over a curve}, 2014.

\bibitem{KellerRoss}
J.~Keller and J.~Ross.
\newblock \href{http://arxiv.org/abs/1110.4489}{A Note on {C}how Stability of
  the Projectivization of {G}ieseker Stable Bundles}.
\newblock {\em Journal of Geometric Analysis}, pages 1--21, 2011.

\bibitem{kobayashi2014differential}
S.~Kobayashi.
\newblock {\em Differential geometry of complex vector bundles}.
\newblock Princeton University Press, 2014.

\bibitem{lam1975formula}
K.~Lam.
\newblock A formula for the tangent bundle of flag manifolds and related
  manifolds.
\newblock {\em Transactions of the American Mathematical Society},
  213:305--314, 1975.

\bibitem{Lazarsfeld2004}
R.~Lazarsfeld.
\newblock {\em Positivity in algebraic geometry {I}: {C}lassical setting: line
  bundles and linear series}, volume~48.
\newblock Springer, 2004.

\bibitem{lazarsfeld2008convex}
R.~Lazarsfeld and M.~Musta{\c{t}\v{a}}.
\newblock \href{http://arXiv.org/abs/0805.4559}{Convex bodies associated to
  linear series}.
\newblock {\em Ann. Sci. {\'E}c. Norm. Sup{\'e}r.}, (42), 2009.

\bibitem{lebrun1994extremal}
C.~LeBrun and S.~Simanca.
\newblock Extremal {K{\"a}hler} metrics and complex deformation theory.
\newblock {\em Geometric \& Functional Analysis GAFA}, 4(3):298--336, 1994.

\bibitem{li2011special}
C.~Li and C.~Xu.
\newblock
  \href{http://www.math.stonybrook.edu/~chili/papers/Kstability.pdf}{Special
  test configuration and K-stability of Fano varieties}.
\newblock {\em Annals of mathematics}, 180(1):197--232, 2014.

\bibitem{lu2014extremal}
Z.~Lu and R.~Seyyedali.
\newblock \href{http://arxiv.org/abs/1310.3006}{Extremal metrics on ruled
  manifolds}.
\newblock {\em Advances in Mathematics}, 258:127--153, 2014.

\bibitem{mabuchi2004uniqueness}
T.~Mabuchi.
\newblock Uniqueness of extremal {K{\"a}hler} metrics for an integral
  {K{\"a}hler} class.
\newblock {\em International Journal of Mathematics}, 15(06):531--546, 2004.

\bibitem{Mabuchi2008}
T.~Mabuchi.
\newblock Chow-stability and {Hilbert-stability in Mumford}'s geometric
  invariant theory.
\newblock {\em Osaka Journal of Mathematics}, 45(3):833--846, 2008.

\bibitem{manivel1994theoreme}
L.~Manivel.
\newblock Un th{\'e}or{\`e}me d'annulation ``{\`a} la {Kawamata-Viehweg}''.
\newblock {\em manuscripta mathematica}, 83(1):387--404, 1994.

\bibitem{mumford1977stability}
D.~Mumford.
\newblock {\em Stability of projective varieties}.
\newblock L'Enseignement math{\'e}matique, 1977.

\bibitem{mumford1994geometric}
D.~Mumford, J.~Fogarty, and F.~Kirwan.
\newblock {\em Geometric invariant theory}, volume~34.
\newblock Springer, 1994.

\bibitem{Narasimhan1965}
M.~Narasimhan and C.~Seshadri.
\newblock Stable and unitary vector bundles on a compact {R}iemann surface.
\newblock {\em The Annals of Mathematics}, 82(3):540--567, 1965.

\bibitem{odaka2009generalization}
Y.~Odaka.
\newblock \href{arXiv.org/abs/0910.1794}{A generalization of {Ross-Thomas}'
  slope theory}.
\newblock {\em Osaka J. Math}, 50(1), 2013.

\bibitem{odaka2015parametrization}
Y.~Odaka.
\newblock \href{http://arxiv.org/abs/1201.0692}{On parametrization,
  optimization and triviality of test configurations}.
\newblock {\em Proceedings of the American Mathematical Society},
  143(1):25--33, 2015.

\bibitem{odaka2012alpha}
Y.~Odaka and Y.~Sano.
\newblock \href{http://arxiv.org/abs/1011.6131v2}{Alpha invariant and
  {K}-stability of {Q-Fano} varieties}.
\newblock {\em Advances in Mathematics}, 229(5):2818--2834, 2012.

\bibitem{ross2006unstable}
J.~Ross.
\newblock \href{http://arxiv.org/abs/math/0506447}{Unstable products of smooth
  curves}.
\newblock {\em Inventiones mathematicae}, 165(1):153--162, 2006.

\bibitem{RossThomas}
J.~Ross and R.~Thomas.
\newblock \href{http://arxiv.org/abs/math/0412518}{An obstruction to the
  existence of constant scalar curvature {K}{\"a}hler metrics}.
\newblock {\em Journal of Differential Geometry}, 72(3):429--466, 2006.

\bibitem{Ross2007}
J.~Ross and R.~Thomas.
\newblock \href{http://arxiv.org/abs/math/0412519}{A study of the
  {H}ilbert-{M}umford criterion for the stability of projective varieties}.
\newblock {\em Journal of Algebraic Geometry}, 16(2):201--255, 2007.

\bibitem{ross2011weighted}
J.~Ross and R.~Thomas.
\newblock \href{http://arxiv.org/abs/0907.5214}{Weighted projective embeddings,
  stability of orbifolds, and constant scalar curvature {K}{\"a}hler metrics}.
\newblock {\em Journal of Differential Geometry}, 88(1):109--159, 2011.

\bibitem{ross2014analytic}
J.~Ross and D.~Witt-Nystr{\"o}m.
\newblock \href{http://arxiv.org/abs/1101.1612}{Analytic test configurations
  and geodesic rays}.
\newblock {\em Journal of Symplectic Geometry}, 12(1):125--169, 2014.

\bibitem{serre1954representations}
J.-P. Serre.
\newblock
  \href{http://www.numdam.org/item?id=SB_1951-1954__2__447_0}{Repr{\'e}sentations
  lin{\'e}aires et espaces homogenes {K}{\"a}hl{\'e}riens des groupes de {L}ie
  compacts}.
\newblock {\em S{\'e}minaire Bourbaki}, 2:447--454, 1954.

\bibitem{serre1958espaces}
J.-P. Serre.
\newblock
  \href{http://archive.numdam.org/ARCHIVE/SCC/SCC_1958__3_/SCC_1958__3__A1_0/SCC_1958__3__A1_0.pdf}{Espaces
  fibr{\'e}s alg{\'e}briques}.
\newblock {\em S{\'e}minaire Claude Chevalley}, 3:1--37, 1958.

\bibitem{Seyyedali2010}
R.~Seyyedali.
\newblock \href{http://projecteuclid.org/euclid.dmj/1275671398}{{Balanced
  metrics and {C}how stability of projective bundles over {K}\"{a}hler
  manifolds}}.
\newblock {\em Duke Mathematical Journal}, 153(3):573--605, June 2010.

\bibitem{stacks-project}
T.~{Stacks Project Authors}.
\newblock \href{http://stacks.math.columbia.edu}{Stacks Project}, April 2015.

\bibitem{stoppa2007unstable}
J.~Stoppa.
\newblock \href{http://www-dimat.unipv.it/stoppa/research/jag.pdf}{Unstable
  blowups}.
\newblock 2007.

\bibitem{Stoppa2009}
J.~Stoppa.
\newblock \href{http://www-dimat.unipv.it/stoppa/research/adv.pdf}{K-stability
  of constant scalar curvature {K}{\"a}hler manifolds}.
\newblock {\em Advances in Mathematics}, 221(4):1397--1408, 2009.

\bibitem{stoppa2009twisted}
J.~Stoppa.
\newblock \href{http://www-dimat.unipv.it/stoppa/research/jdg.pdf}{Twisted
  constant scalar curvature {K}{\"a}hler metrics and {K}{\"a}hler slope
  stability}.
\newblock {\em Journal of Differential Geometry}, 83(3):663--691, 2009.

\bibitem{stoppa2011note}
J.~Stoppa.
\newblock \href{arXiv.org/abs/1111.5826}{A note on the definition of
  {K}-stability}, 2011.

\bibitem{stoppa2011relative}
J.~Stoppa and G.~Sz{\'e}kelyhidi.
\newblock \href{http://arxiv.org/abs/0912.4095}{Relative {K}-stability of
  extremal metrics}.
\newblock {\em J. Eur. Math. Soc}, 13(4):899--909, 2011.

\bibitem{Stoppa2010}
J.~Stoppa and E.~Tenni.
\newblock \href{http://www-dimat.unipv.it/stoppa/research/imrn.pdf}{A simple
  limit for slope instability}.
\newblock {\em International Mathematics Research Notices},
  2010(10):1816--1830, 2010.

\bibitem{szekelyhidi2007extremal}
G.~Sz{\'e}kelyhidi.
\newblock \href{http://arxiv.org/abs/math/0410401}{Extremal metrics and
  {K}-stability}.
\newblock {\em Bulletin of the London Mathematical Society}, 39(1):76--84,
  2007.

\bibitem{szekelyhidi2008optimal}
G.~Sz\'{e}kelyhidi.
\newblock \href{http://arxiv.org/absf/0709.2687}{Optimal test-configurations
  for toric varieties}.
\newblock {\em J. Differential Geom.}, 80:501--523, 2008.

\bibitem{szekelyhidi2009calabi}
G.~Sz{\'e}kelyhidi.
\newblock \href{http://arxiv.org/pdf/math.DG/0703562}{The {Calabi} functional
  on a ruled surface}.
\newblock {\em Ann. Sci. {\'E}c. Norm. Sup{\'e}r.}, 42:837--856, 2009.

\bibitem{szekelyhidi2011filtrations}
G.~Sz{\'e}kelyhidi.
\newblock \href{http://arxiv.org/abs/1111.4986}{Filtrations and
  test-configurations}.
\newblock {\em Mathematische Annalen}, pages 1--34, 2014.

\bibitem{towber1977two}
J.~Towber.
\newblock Two new functors from modules to algebras.
\newblock {\em Journal of Algebra}, 47(1):80--104, 1977.

\bibitem{uhlenbeck1986existence}
K.~Uhlenbeck and S.-T. Yau.
\newblock On the existence of {H}ermitian-{Y}ang-{M}ills connections in stable
  vector bundles.
\newblock {\em Communications on Pure and Applied Mathematics},
  39(S1):S257--S293, 1986.

\bibitem{Weyman}
J.~Weyman.
\newblock {\em Cohomology of vector bundles and syzygies}, volume 149.
\newblock Cambridge University Press, 2003.

\bibitem{witt2012test}
D.~Witt-Nystr{\"o}m.
\newblock \href{http://arxiv.org/abs/1001.3286}{Test configurations and
  {O}kounkov bodies}.
\newblock {\em Compositio Mathematica}, 148(06):1736--1756, 2012.

\bibitem{yau1978ricci}
S.-T. Yau.
\newblock On the {R}icci curvature of a compact {K}{\"a}hler manifold and the
  complex {M}onge-{A}mp{\'e}re equation, {I}.
\newblock {\em Communications on pure and applied mathematics}, 31(3):339--411,
  1978.

\bibitem{yau2000open}
S.-T. Yau.
\newblock Open problems in geometry.
\newblock {\em Ramanujan Mathematical Society}, 15(2):125--134, 2000.

\end{thebibliography}
\bibliographystyle{abbrv}
\end{document}